\documentclass[a4paper,11pt,english]{smfart}
\usepackage{amsfonts,skt}
\usepackage{amsmath} 
\usepackage{hyperref} 
\usepackage{latexsym}
\usepackage{array}
\usepackage{amssymb}
\usepackage{enumerate}
\usepackage{smfthm}
\usepackage{mathrsfs}
\usepackage[T1]{fontenc}
\usepackage{pdfsync}
\theoremstyle{plain}

\newtheorem*{acknowledgements}{Acknowledgements}

\newcommand{\ph}{  \phantomsection   }
\newcommand{\R}{  \mathbb{R}   }
\newcommand{\E}{  \mathbb{E}   }

\newcommand{\eps}{\varepsilon}
\newcommand{\lessim}{  \lesssim   }
\newcommand{\e}{  \text{e}   }
\newcommand{\C}{  \mathbb{C}   }
\newcommand{\Z}{  \mathbb{Z}   }
\newcommand{\N}{  \mathbb{N}   }

\renewcommand{\H}{  \mathcal{H}   }

\newcommand{\T}{  \mathbb{S}^{1}   }

\newcommand{\im}{  \text{Im}\;   }

\newcommand{\om}{  \omega   }
\newcommand{\p}{  \partial   }
\newcommand{\wt}{  \widetilde   }
\newcommand{\dis}{  \displaystyle   }
\newcommand{\ov}{  \overline  }
\renewcommand{\a}{  \alpha   }
\renewcommand{\p}{  \bf p  }
\renewcommand{\b}{  \beta   }
\newcommand{\s}{  \sigma   }

\newcommand{\<}{  \langle   }
\renewcommand{\>}{  \rangle   }
\renewcommand{\phi}{  \varphi   }
\renewcommand{\S}{  \mathbb{S}   }
\numberwithin{equation}{section}

\author{ Nicolas Burq}
\address{Laboratoire de Math\'ematiques, B\^at. 425,
Universit\'e Paris Sud, 91405 Orsay Cedex, France}
\email{nicolas.burq@math.u-psud.fr}
\author{ Laurent Thomann }
\address{Laboratoire de Math\'ematiques J. Leray, Universit\'e de Nantes, UMR CNRS 6629\\
2, rue de la Houssini\`ere,
44322 Nantes Cedex 03, France}
\email{laurent.thomann@univ-nantes.fr}
\author{ Nikolay Tzvetkov}
\address{University of Cergy-Pontoise, UMR CNRS 8088, Cergy-Pontoise, F-95000}
\email{nikolay.tzvetkov@u-cergy.fr}

\title[Remarks on the gibbs measures for nonlinear dispersive equations]{Remarks on the gibbs measures for nonlinear dispersive equations} 

\begin{document}
\frontmatter
 \begin{abstract}
We show, by the means of several  examples,  how we can use Gibbs measures to construct global solutions to dispersive equations at low regularity.  The construction relies on  the Prokhorov compactness theorem combined with the Skorokhod convergence theorem. To begin with, we  consider  the non linear Schr\"odinger equation (NLS) on the tri-dimensional sphere. Then we focus on  the Benjamin-Ono equation and on  the derivative nonlinear Schr\"odinger equation on the circle. Next, we construct a Gibbs measure and global solutions to the  so-called periodic half-wave equation.   
Finally, we consider the cubic $2d$ defocusing NLS on an arbitrary spatial domain and
we construct  global solutions on the support of the associated Gibbs measure.  
\end{abstract}
\thanks{N. B. and L.T. are  supported   by the  grant  ``ANA\'E'' ANR-13-BS01-0010-03.   and N.T. by an  ERC grant.}
\maketitle
\mainmatter
 
\tableofcontents
 
\section {Introduction and main results}

\subsection{General introduction}
A Gibbs measure can be an interesting tool to show that local solutions to some dispersive PDEs are indeed global.  Once we have a suitable local existence and uniqueness  theory on the support of such a measure, we can expect to globalise these solutions; this measure    in some sense compensates the lack of conservation law at some level of Sobolev regularity. See~\cite{Bourgain2, Bourgain1,Zhidkov,Tzvetkov2, Tzvetkov1,BT3,Oh1, Oh2,BTT}  where this approach has been fruitful.

Assume now that we have   a  Gibbs measure, but that we are not able to show that the equation is locally well-posed on its  support. The aim of this paper is to show -  through several examples - that in this case we can use some compactness methods to construct global (but non unique) solutions on the support of the measure.
Although this method of construction of solutions is well-known in other  contexts, like for the Euler equation (see Albeverio-Cruzeiro~\cite{AC}) or for the Navier-Stokes equation   (see Da Prato-Debussche~\cite{DPD}),  it seems to be not exploited in the context of dispersive equations.

In~\cite{BTT2} we have constructed global rough solutions to the periodic wave equation in any dimension with  stochastic tools. While in~\cite{BTT2} we used the energy conservation and a regularisation property of the wave equation in the argument, here we use instead the invariance of the measure by the non linear flow. As a consequence we also obtain that the distribution of the solutions we construct  is independent of time. 

Our first example concerns the non linear Schr\"odinger equation on the sphere $\S^{3}$ restricted to zonal functions (the functions which only depend on the geodesic distance to the north pole). For sub-quintic nonlinearities, we are able to define a Gibbs measure  with support in $H^{\s}(\S^{3})$ for any $\s<1/2$, and to construct global solutions in this space. This is the result of Theorem~\ref{thmNLS}.  In~\cite{BourgainBulut},  Bourgain-Bulut have considered a similar equation (the radial NLS on $\R^{3}$) in the case of the  cubic nonlinearity. 
The solutions obtained in~\cite{BourgainBulut} are certainly "stronger" compared to the ones obtained in the present paper, the uniqueness statement being however not 
explicited in~\cite{BourgainBulut}.

In a second time we deal with the Benjamin-Ono equation on the circle $\T=\R/(2\pi\Z)$. This model arises in the study of one-dimensional internal long waves.  In~\cite{Molinet1,Molinet2} L. Molinet has shown that the equation is globally well-posed in $L^{2}(\T)$ and that this result is sharp.
For this problem, a Gibbs measure with support in $H^{-\s}(\T)$, for any $\s>0$ has already been constructed by N.\;Tzvetkov in~\cite{Tzvetkov3}. In this case, we also construct global solutions on the support of the measure and prove its invariance (Theorem~\ref{thmBO}). A  uniqueness result of the dynamics on the support of the measure was recently proven in a remarkable paper by Y. Deng~\cite{Deng}.

Our third example concerns the periodic derivative Schr\"odinger equation. Here we  use the measure constructed by  Thomann-Tzvetkov~\cite{ThTz}. We construct a dynamics for which the  measure is invariant (Theorem~\ref{thmDNLS}). This result may be seen as a consequence of a recent work by Nahmod, Oh, Rey-Bellet and Staffilani~\cite{NORS} and Nahmod, Ray-Bellet, Sheffield and Staffilani~\cite{NRSS}. Their approach  is based on the local deterministic theory of Gr\"unrock-Herr~\cite{GrunHerr} which gauges out (the worst part of) the nonlinearity, and the uniqueness is only proved in this gauged-out context. 

Next, we consider the so-called half-wave equation on the circle, which can be seen as a limit model of Schr\"odinger-like equations for which one has very few dispersion. This model has been studied by G\'erard-Grellier~\cite{GerGre} who showed that it is well-posed in $H^{1/2}(\T)$ (see also O. Pocovnicu~\cite{Pocovnicu1} and more recently Krieger-Lenzmann-Rapha\"el~\cite{KLR} for a study of the equation on the real line). Here a Gibbs measure with support in $H^{-\s}(\T)$, for any $\s>0$ can be defined, and global solutions (see Theorem\;\ref{thmHW}) can be constructed. 

Finally, we consider the NLS on an arbitrary $2d$ spatial domain. Here the construction of the Gibbs measure goes back to the works in QFT (see~\cite{Simon} and the references therein). For the sake of completeness, we shall present a proof below based on (precise versions of) the Weyl formula. However we want to stress that the ideas used in the construction of the Gibbs measure  are in the spirit of~\cite{Simon}.
The support of the measure is again $H^{-s}$ for any $s>0$. 
In the case of the torus as spatial domain, J.~Bourgain~\cite{Bourgain1} constructs strong global solutions on the support of the measure. This remarkable result relies on the local theory in  $H^\sigma$, $\sigma>0$ and on a probabilistic regularization property. In the case of an arbitrary spatial domain the local theory is much more involved  (and presently restricted to much higher regularity) compared to the case of the torus. Consequently, it seems natural to turn to weak solution techniques.

Therefore our results on the half-wave equation and the $2d$ NLS are out of reach of the present "strong solutions" 
methods and as such they should be seen as the main result of this article. 

Summarizing the previous discussion, one may also conclude that the main point to be discussed when applying strong solutions techniques in all considered examples is the {\it uniqueness}. 
\subsection{The Schr\"odinger equation on ${\mathbf{\S^3}}$}
Let $\S^{3}$ be the unit sphere in $\R^{4}$.  We then  consider the  non linear Schr\"odinger equation
 \begin{equation}\label{NLS00}
\left\{
\begin{aligned}
&i\partial_{t}u+ \Delta_{\S^{3}} u= |u|^{r-1}u, \quad   (t,x)\in \R\times \S^{3},\\
&u(0,x)=  f(x) \in H^{\s}(\S^{3}),
\end{aligned}
\right.
\end{equation}
for $1\leq r<5$. In~\cite{BGT} N. Burq, P. G\'erard and N. Tzvetkov have shown that~\eqref{NLS00} is globally well-posed in the energy space $H^{1}(\S^{3})$. In this paper we address the question of the existence of global solutions at regularity below the energy space.
Denote by  $Z(\S^{3})$ the space of  the zonal functions, {\it i.e.} the space of the functions which only depend on the geodesic distance to the north pole of $\S^{3}$. Set $H^{\s}_{rad}(\S^{3}):=H^{\s}(\S^{3})\cap Z(\S^{3})$, $L^{2}_{rad}(\S^{3})=H^{0}_{rad}(\S^{3})$ and 
$$X^{1/2}_{rad}=X^{1/2}_{rad}(\S^{3})=\bigcap_{\s<1/2}H_{rad}^{\s}(\S^{3}).$$

 For $x\in \S^{3}$, denote by $\theta=\text{dist}(x,N)\in [0,\pi]$ the geodesic distance of $x$ to the north pole and define 
\begin{equation}\label{def.p}
P_{n}(x)=\sqrt{\frac2{\pi}}\frac{\sin n\theta}{\sin \theta},\quad \quad n\geq 1.
\end{equation}
Then, $(P_{n})_{n\geq 1}$ is a Hilbertian basis of $L^{2}_{rad}(\S^{3})$, which will be used in the sequel. 
Next, in order to avoid the issue  with the 0-frequency, we make the change of unknown $u\longmapsto \e^{-it}u$, so that we are reduced to consider the equation
 \begin{equation}\label{NLS0}
\left\{
\begin{aligned}
&i\partial_{t}u+ (\Delta_{\S^{3}} -1)u= |u|^{r-1}u, \quad   (t,x)\in \R\times \S^{3},\\
&u(0,x)=  f(x) \in H^{\s}(\S^{3}).
\end{aligned}
\right.
\end{equation}


Let $({\Omega, \mathcal{F},{\bf p}})$ be a probability space and $\big(g_{n}(\om)\big)_{n\geq 1}$ a sequence of independent complex normalised Gaussians, $g_{n}\in \mathcal{N}_{\C}(0,1)$, which means that $g_{n}$ can be written  
\begin{equation*} 
g_{n}(\om)=\frac1{\sqrt{2}}\big(h_{n}(\om)+i\ell_{n}(\om)\big),
\end{equation*}
where$\big(h_{n}(\om),\, \ell_{n}(\om)\big)_{n\geq 1}$ are independent standard real Gaussians  ($\mathcal{N}_{\R}(0,1)$).

For $N\geq 1$ we define the random variable
\begin{equation*} 
\om\mapsto\varphi_{N}(\omega,x)= \sum_{n=1}^{N}\frac{g_{n}(\om)}{n}P_{n}(x),
\end{equation*}
and we can show that if  $\sigma<\frac12$, then $(\varphi_N)_{N\geq 1}$ is a Cauchy sequence in $L^2\big(\Omega;\,{H}^{\sigma}(\S^{3})\big)$: this
enables us to define its limit
\begin{equation}\label{phi.NLS} 
\om\mapsto\phi(\om,x)=\sum_{n\geq 1}\frac{g_{n}(\om)}{n}P_{n}(x) \in L^2\big(\Omega;\,{H}^{\s}(\S^{3})\big).
\end{equation}
We then define the Gaussian probability  measure $\mu$ on $ X_{rad}^{1/2}(\S^{3})$ by $\mu={\p}\circ \phi^{-1}$. In other words, $\mu$ is the image of the measure ${\bf p}$ under the map 
\begin{equation*}
 \begin{array}{rcl}
\Omega&\longrightarrow&X_{rad}^{1/2}(\S^{3})\\[3pt]
\dis  \omega&\longmapsto &\dis\phi(\om,\cdot)=\sum_{n\geq 1}\frac{g_{n}(\om)}{n}P_{n}.
 \end{array}
 \end{equation*}
     We now  construct a Gibbs measure for the equation~\eqref{NLS0}. For $u\in L^{r+1}(\S^{3})$ and $\b>0$, define the density
\begin{equation}\label{defG}
G(u)=
\b\e^{-\frac1{r+1}\int_{\S^{3}}|u|^{r+1}}, 
\end{equation}
and with a suitable choice of $\beta>0$, this enables to construct a   probability measure $\rho$ on $X_{rad}^{1/2}(\S^{3})$ by 
\begin{equation*}
\text{d}{\rho }(u)=G(u)\text{d}\mu(u).
\end{equation*} 
Then we can prove
 \begin{theo}\ph\label{thmNLS} 
Let  $1\leq r<5$. The measure $\rho$ is invariant under a dynamics of~\eqref{NLS00}. More precisely,  there exists a set $\Sigma$ 
of full $\rho$ measure   so that for every $f\in \Sigma$ the
equation~\eqref{NLS00} with 
initial condition $u(0)=f$ has a   solution 
\begin{equation*}
u\in \mathcal{C}\big(\R\, ; X_{rad}^{1/2}(\S^{3}) \big).
\end{equation*}
The distribution of the random variable $u(t)$ is equal to  $\rho$ (and thus independent of $t\in \R$):
\begin{equation*}
\mathscr{L}_{X_{rad}^{1/2}}\big(u(t)\big)=\mathscr{L}_{X_{rad}^{1/2}}\big(u(0)\big)=\rho,\quad \forall\,t\in \R.
\end{equation*} 
\end{theo}
Here and after, we abuse notation and write
\begin{equation*}
 \mathcal{C}\big(\R\, ; X_{rad}^{1/2}(\S^{3}) \big)= \bigcap_{\s<1/2}\mathcal{C}\big(\R\, ; H_{rad}^{\s}(\S^{3}) \big).
\end{equation*}

In our work, the only point where we need to restrict to zonal functions is for the construction of the Gibbs measure. The other arguments do not need any radial assumption. The result of Theorem~\ref{thmNLS} can not be extended to the case $r=5$. Indeed, it is shown in~\cite[Theorem 4]{AyTz} that $\|u\|_{L^{6}(\S^3)}=+\infty$, $\mu-$a.s.\\
Since $G(u)>0$, $\mu-$a.s., both measures $\mu$ and $\rho$ have same support. Indeed,  $\mu(X_{rad}^{1/2}(\S^{3}))=\rho(X_{rad}^{1/2}(\S^{3}))=1$, but we can check  that  $\mu(H_{rad}^{1/2}(\S^{3}))=\rho(H_{rad}^{1/2}(\S^{3}))=0$ (see~\cite[Proposition C.1]{BL}).\medskip

Let us compare our result to the result given by the usual deterministic compactness methods. The energy of the equation~\eqref{NLS00} reads 
$$H(u)=\frac12\int_{\S^{3}} |\nabla u|^{2}+\frac{1}{r+1}\int_{\S^{3}}|u|^{r+1}.$$
Then, one can prove (see {\it e.g.}~\cite{Cazenave}) that for all $f\in H^{1}(\S^{3})\cap L^{r+1}(\S^{3})$ there exists a solution to~\eqref{NLS00} so that 
\begin{equation}\label{espace}
u\in \mathcal{C}_{w}\big(\R; H^{1}(\S^{3}) \big)\cap\, \mathcal{C}_{w}\big(\R; L^{r+1}(\S^{3}) \big),
\end{equation}
(here $\mathcal{C}_{w}$ stands for weak continuity in time) and so that for all $t\in \R$, $H(u)(t)\leq H(f)$. Notice that in~\eqref{espace} we can replace the space $H^{1}$ with $H^{1}_{rad}$ if $f\in H_{rad}^{1}$. \\The advantage of this method is that there is no restriction on $r\geq 1$ and no radial assumption on the initial condition. However this strategy asks more regularity on $f$. We also point out that with the deterministic method one loses the conservation of the energy, while in Theorem~\ref{thmNLS} we obtain an invariant probability measure (see also  Remark~\ref{marker}).
 \subsection{The Benjamin-Ono equation}
Recall that $\T:=\R/(2\pi\Z)$ and let us define
$$\dis \|f\|^{2}_{L^{2}(\T)}=(2\pi)^{-1}\int_{0}^{2\pi}|f(x)|^{2}\text{d}x.$$
 For $\dis f(x)=\sum_{k\in \Z}\a_{k}\e^{ikx}$ and $N\geq 1$ we define the spectral projector $\Pi_{N}$ by 
$\dis \Pi_{N}f(x)=\sum_{|k|\leq N}\a_{k}\e^{ikx}$. We also define the space $\dis X^{0}(\S^1)=\bigcap_{\s>0}H^{-\s}(\S^{1})$.

Denote by  $\mathcal{H}$ the Hilbert transform, which is defined by 
\begin{equation*}
\mathcal{H} u(x)=-i\sum_{n\in \Z^{\star}} \text{sign}(n)c_{n}\e^{inx}, \quad \text{for}\quad u(x)=\sum_{n\in \Z^{\star}} c_{n}\e^{inx}.
\end{equation*}
 In this section, we are interested in the periodic  Benjamin-Ono equation
  \begin{equation}\label{BO}
\left\{
\begin{aligned}
&\partial_t u + \mathcal{H}\partial^{2}_{x} u   +\partial_{x}\big(u^{2}\big) =0,\quad 
(t,x)\in\R\times \mathbb{S}^1,\\
&u(0,x)= f(x).
\end{aligned}
\right.
\end{equation}
Let $({\Omega, \mathcal{F},{\bf p}})$ be a probability space and $\big(g_{n}(\om)\big)_{n\geq 1}$ a sequence of independent complex normalised Gaussians, $g_{n}\in \mathcal{N}_{\C}(0,1)$. Set $g_{-n}(\om)=\ov{g_{n}(\om)}$. For any $\s>0$, we can define the random variable
\begin{equation}\label{phi.BO}
\om\mapsto\phi(\om,x)=\sum_{n\in \Z^{*}}\frac{g_{n}(\om)}{2|n|^{\frac12}}\e^{inx}\in L^2\big(\Omega;\,{H}^{-\s}(\T)\big),
\end{equation}
and the measure $\mu$ on ${ X}^{0}(\T)$ by $\mu={\bf p}\circ \phi^{-1}$. 
Next, as in~\cite{Tzvetkov3} define the measure $\rho_{N}$ on $X^{0}(\T)$ by
 \begin{equation}\label{def.rho.BO}
 \text{d}{\rho}_{N}(u)=\Psi_{N}(u)\text{d}\mu(u),
 \end{equation}
 where the weight $\Psi$ is given by
  \begin{equation*} 
 \Psi_{N}(u)=\beta_{N}\chi\big(\|u_{N}\|^{2}_{L^{2}}-\alpha_{N}\big)\e^{-\frac23\int_{\T}u^{3}_{N}(x)\text{d}x},\quad u_{N}=\Pi_{N}u,
 \end{equation*}
with $\chi \in \mathcal{C}^{\infty}_{0}(\R)$, 
$$\dis \a_{N}=\int_{X^{0}(\S^{1})}\|u_{N}\|^{2}_{L^{2}(\S^{1})}\text{d}\mu(u)=\int_{\Omega}\|\phi_{N}(\om,.)\|^{2}_{L^{2}(\S^{1})}\text{d}{\p}(\om)=\sum_{1\leq n\leq N}\frac{1}{n},$$
and where the constant $\beta_{N}>0$ is chosen so that $\rho_{N}$ is a probability measure on $X^{0}(\T)$. Then the result of N. Tzvetkov~\cite{Tzvetkov3} reads: There exists $\Psi(u)$ which satisfies for all $p\in [1,+\infty[$, $\Psi(u)\in L^{p}(\text{d}\mu)$ and 
 \begin{equation}\label{Tz}
\Psi_{N}(u)\longrightarrow \Psi(u)\quad \text{ in} \quad L^{p}(\text{d}\mu(u)).  
\end{equation}
As a consequence, we can define a probability measure $\rho$ on $X^{0}(\S^{1})$ by  $\text{d}\rho(u)=\Psi(u)\text{d}\mu(u)$. Then our result is the following

  \begin{theo}\ph\label{thmBO} 
There exists a set $\Sigma$ 
of full $\rho$ measure   so that for every $f\in \Sigma$ the
equation~\eqref{BO} with 
initial condition $u(0)=f$ has a   solution 
\begin{equation*}
u\in \mathcal{C}\big(\R\, ; X^{0}(\T) \big).
\end{equation*}
 For all $t\in \R$, the distribution of the random variable $u(t)$ is $\rho$. 
\end{theo}

Some care has to be given for the definition of the non linear term in~\eqref{BO}, since $u$ has a negative Sobolev regularity. Here we can define $\partial_{x}(u^{2})$ on the support of $\mu$ as a limit of a Cauchy sequence (see Lemma~\ref{Lem.cauchy}).

As in~\cite[Proposition 3.10]{BTT} we can prove that 
\begin{equation*}
\bigcup_{\chi \in \mathcal{C}^{\infty}_{0}(\R)} \text{supp} \,\rho=\text{supp} \,\mu.
\end{equation*}
The cut-off $\chi\big(\|u_{N}\|^{2}_{L^{2}}-\alpha_{N}\big)$ can not be avoided here because the term  $\int_{\T}u^{3}_{N}(x)\text{d}x$ does not have a sign (compare with the analysis of the half-wave equation and the defocusing NLS below where it can be avoided, after a suitable renormalisation of the potential energy).  

Observe that $\phi$ in~\eqref{phi.BO} has mean 0, thus $\mu$ and $\rho$ are supported on 0-mean functions. This is not a restriction since the mean $\int_{\S^{1}}u$ is an invariant of~\eqref{BO}.

We complete this section by mentioning~\cite{TV,TV1,TV2, DTV} where the authors construct Gibbs type measures associated with each conservation law of the Benjamin-Ono equation.
\subsection{The derivative non linear Schr\"odinger equation}
We consider the periodic DNLS equation. 
\begin{equation}\label{DNLS}
\left\{
\begin{aligned}
&i\partial_t u+\partial_{x}^{2} u   = i\partial_{x}\big(|u|^{2}u\big),\;\;
(t,x)\in\R\times \mathbb{S}^1,\\
&u(0,x)= u_{0}(x).
\end{aligned}
\right.
\end{equation}
Here, for $\s<1/2$ we define  the random variable ($\<n\>=(1+n^{2})^{1/2}$)
\begin{equation}\label{phi.DNLS}
\om\mapsto\phi(\om,x)=\sum_{n\in \Z}\frac{g_{n}(\om)}{\<n\>}\e^{inx}\in L^2\big(\Omega;\,{H}^{\sigma}(\T)\big),
\end{equation}
and the measure $\mu$ on $ \dis X^{1/2}(\T)=\bigcap_{\s<1/2}H^{\s}(\S^{1})$ by $\mu={\bf p}\circ \phi^{-1}$. Next, denote by
\begin{equation*} 
f_{N}(u)=\im \int_{\T}\ov{u_{N}^{2}(x)}\,\partial_{x}(u_{N}^{2}(x))\text{d}x.
\end{equation*} 
 Let $\kappa>0$, and let $\chi \;:\;\R\longrightarrow \R$, $0\leq \chi \leq 1$ be a continuous function with support $\text{supp}\;\chi\subset [-\kappa, \kappa]$ and so that $\chi =1$ on $[-{\kappa}/2, {\kappa}/2]$. We define the density 
\begin{equation*}
\Psi_{N}(u)=\b_{N}\chi\big(\|u_{N}\|_{L^2(\T)}\big)
\e^{\frac{3}{4}f_{N}(u)-\frac12\int_{\T}|u_{N}(x)|^6\text{d}x},
\end{equation*}
and the measure $\rho_{N}$ on $X^{1/2}(\T)$ by 
\begin{equation}\label{rho.DNLS}
\text{d}\rho_{N}(u)=\Psi_{N}(u)\text{d}\mu(u),
\end{equation}
and where  $\b_{N}>0$ is  chosen so that $\rho_{N}$ is a probability measure on $X^{1/2}(\T)$.
By Thomann-Tzvetkov~\cite[Theorem 1.1]{ThTz}, $\rho_{N}$ converges to a probability measure $\rho$ so that $\text{d}\rho(u)=\Psi(u)\text{d}\mu(u)$. Moreover, for all $p\geq 2$, if $\kappa\leq \kappa_{p}$, then  $\Psi(u)\in L^{p}(\text{d}\mu)$. Then our result reads
  \begin{theo}\ph\label{thmDNLS} 
Assume that $\kappa\leq \kappa_{2}$. Then there exists a set $\Sigma$ 
of full $\rho$ measure   so that for every $f\in \Sigma$ the
equation~\eqref{DNLS} with 
initial condition $u(0)=f$ has a     solution 
\begin{equation*}
u\in \mathcal{C}\big(\R \,; X^{1/2}(\T) \big).
\end{equation*}
For all $t\in \R$, the distribution of the random variable $u(t)$ is  $\rho$. 
\end{theo}
Here, for $\kappa\leq \kappa_{2}$,  we have 
 \begin{equation*}
 \bigcup_{\chi \in \mathcal{C}^{\infty}_{0}([-\kappa,\kappa])} \text{supp} \,\rho= \big\{\|u\|_{L^{2}}\leq \kappa \big\}\bigcap\,  \text{supp} \,\mu. 
 \end{equation*}

\subsection{The half-wave equation}
The periodic cubic Schr\"odinger on the circle has been much studied and in particular rough solutions have been constructed. See Christ~\cite{Christ}, Colliander-Oh~\cite{CO}, Kwon-Oh~\cite{KwOh},  and Bourgain~\cite{Bourgain1} in the 2-dimensional case. \\ Here we investigate a related equation where one has no more dispersion: We replace the Laplacian with the operator   $|D|$, {\it i.e.} the operator defined by $|D| \e^{inx}=|n|\e^{inx}$, and we  consider  the  following half-wave Cauchy problem 
\begin{equation*} 
\left\{
\begin{aligned}
&i\partial_t u-|D| u   = |u|^{2}u,\quad 
(t,x)\in\R\times \mathbb{S}^1,\\
&u(0,x)= f(x).
\end{aligned}
\right.
\end{equation*}
This model has been studied by P. G\'erard and S. Grellier~\cite{GerGre} who showed that it is well-posed in\;$H^{1/2}(\T)$. However, the Sobolev space which is invariant by scaling is $L^{2}(\T)$, hence it is natural to try to construct solutions which have low regularity.   In the sequel, in order to avoid trouble with the 0-frequency, we make the change of unknown $u\longmapsto \e^{-it}u$, so that we are reduced to consider the equation
\begin{equation*}
i\partial_t u-\Lambda u   = |u|^{2}u,\quad (t,x)\in\R\times \mathbb{S}^1,
\end{equation*}
where $\Lambda:=|D|+1$.\medskip
 
   Let $({\Omega, \mathcal{F},{\bf p}})$ be a probability space and $\big(g_{n}(\om)\big)_{n\in \Z}$ a sequence of independent complex normalised Gaussians.
 Here we define the random variable
\begin{equation}\label{phi.HW}
\om\mapsto\phi(\om,x)=\sum_{n\in \Z}\frac{g_{n}(\om)}{(1+|n|)^{\frac12}}\e^{inx}\in L^2\big(\Omega;\,{H}^{-\sigma}(\T)\big),
\end{equation}
for any $\s>0$, and we then define the measure $\mu$ on ${ X}^{0}(\T)$ by $\mu={\bf p}\circ \phi^{-1}$.  \medskip 
 
 We need to give a sense to $|u|^{2}u$ on the support of $\mu$. In order to avoid the worst interaction term, we rather consider a gauged version of the equation for which the nonlinearity is formally\;${|u|^{2}u-2\|u\|^{2}_{L^{2}(\S^{1})}u}$. More precisely, define the Hamiltonian
\begin{equation*}
H_{N}(u)=\ \int_{\T}|\Lambda u|^{2}+\frac1{2}\int_{\T}|\Pi_{N}u|^{4}-\Big(\int_{\T} \big|\Pi_{N}u\big|^{2}\Big)^{2},
\end{equation*}
and consider the equation 
\begin{equation*}
i\partial_{t}u=\frac{\delta H_{N}}{\delta \ov{u}},
\end{equation*}
which reads 
\begin{equation}\label{HW.N}
\left\{
\begin{aligned}
&i\partial_t u-\Lambda u   = G_{N}(u),\;\;
(t,x)\in\R\times \mathbb{S}^1,\\
&u(0,x)= f(x),
\end{aligned}
\right.
\end{equation}
 where $G_{N}$ stands for 
\begin{equation}\label{gauge}
G_{N}(u)=\Pi_{N}\big(|\Pi_{N} u |^{2}\Pi_{N} u\big)-2\|\Pi_{N} u\|^{2}_{L^{2}(\T)}\Pi_{N} u.
\end{equation}
This modification of the nonlinearity is classical, and is the Wick ordered version of the usual cubic nonlinearity (see  Bourgain~\cite{Bourgain1}, Oh-Sulem~\cite{OhSulem}). Recall, that since the $L^{2}$ norm of~\eqref{HW.N} is preserved by the flow, one can recover the standard cubic nonlinearity with the change of function ${\dis v_{N}(t)=u_{N}(t)\exp\big(-2i\int_{0}^{t}\|u_{N}(\tau)\|^{2}_{L^{2}}\text{d}\tau\big)}$ with the notation $u_{N}=\Pi_{N}u$.

Here,  the main  interest for introducing the  gauge transform in~\eqref{gauge} is to define the limit equation, when $N\longrightarrow +\infty$.
\begin{prop}\ph\label{prop.NL}
For all $p\geq 2$, the sequence $\big(G_{N}(u)\big)_{N\geq1}$ is  a Cauchy sequence  in the space $L^{p}\big(X^{0}(\T),\mathcal{B},d\mu; H^{-\s}(\T)\big)$. Namely, for all $p\geq	 2$, there exist $\eta>0$ and $C>0$ so that for all $1\leq M<N$,
 \begin{equation*}
\int_{X^{0}(\T)}\|G_{N}(u)-G_{M}(u)\|^{p}_{H^{-\s}(\T)}\text{d}\mu(u)\leq \frac{C}{M^{\eta}}.
\end{equation*}
We denote by $G(u)$ the limit of this sequence.
\end{prop}
It is then natural to consider the  equation
\begin{equation}\label{HW}
\left\{
\begin{aligned}
&i\partial_t u-\Lambda u   =   G(u),\;\;
(t,x)\in\R\times \mathbb{S}^1,\\
&u(0,x)= f(x).
\end{aligned}
\right.
\end{equation}
 We now define a Gibbs measure for~\eqref{HW} as a limit of Gibbs measures for~\eqref{HW.N}. In the sequel we use the notation $u_{N}=\Pi_{N}u$. Let $\chi \in \mathcal{C}_{0}^{\infty}(\R)$ so that $0\leq \chi\leq 1$. Define
 $$\dis \a_{N}=\int_{X^{0}(\S^{1})}\|u_{N}\|^{2}_{L^{2}(\S^{1})}\text{d}\mu(u)=\sum_{|n|\leq N}\frac{1}{1+|n|},$$
  consider the density 
  \begin{equation}\label{def.theta}
\Theta_{N}(u)=\beta_{N}\chi\big(\|u_{N}\|^{2}_{L^{2}}-\alpha_{N}\big)\e^{-{\big(\|u_{N}\|^{4}_{L^{4}}-2\|u_{N}\|^{4}_{L^{2}}\big)}},
 \end{equation}
 and define the measure
 \begin{equation*}
\text{d}\rho_{N}(u)=\Theta_{N}(u)\text{d}\mu(u),
\end{equation*}
where $\beta_{N}>0$ is chosen so that $\rho_{N}$ is a probability measure.
\begin{rema}We could avoid the cut-off procedure, $\chi$, above, by using another renormalization, namely defining 
\begin{equation}\label{gaugebis}
\widetilde{G}_{N}(u)=\Pi_{N}\big(|\Pi_{N} u |^{2}\Pi_{N} u\big)-2\alpha_N\Pi_{N} u, \qquad \alpha_N= \mathbb{E}_{\mu}\Big[\|\Pi_{N} u\|^{2}_{L^{2}(\T)}\Big],
\end{equation}
see the construction in Section~\ref{sec.8.2} for NLS on a bounded domain.  \end{rema}

 In our next result, we define a weighted Wiener measure for the equation~\eqref{HW}.

 \begin{theo} \ph\label{thm2}
The sequence $\Theta_{N}(u)$ defined in~\eqref{def.theta} converges in measure, as $N\rightarrow\infty$, with respect to the measure $\mu$.
Denote by $\Theta(u)$ the limit and define the probability measure 
\begin{equation}\label{def.rho}
\text{d}\rho(u)\equiv \Theta(u)\text{d}\mu(u).
\end{equation}
Then  for every $p\in [1,\infty[$, $\Theta(u)\in L^{p}(\text{d}\mu(u))$ and  the  sequence $\Theta_{N}$ converges   to $\Theta$ in $L^{p}(\text{d}\mu(u))$, as $N$ tends to
infinity.
\end{theo}
The sign of the nonlinearity in~\eqref{HW} (defocusing) plays a role. Indeed, Theorem~\ref{thm2} does not hold when $G(u)$ is replaced with $-G(u)$.

Again, with the arguments of~\cite[Proposition 3.10]{BTT}, we can prove that 
\begin{equation*}
\bigcup_{\chi \in \mathcal{C}^{\infty}_{0}(\R)} \text{supp} \,\rho=\text{supp} \,\mu.
\end{equation*}

Consider the measure $\rho$ defined in~\eqref{def.rho}, then 
\begin{theo}\ph\label{thmHW} 
There exists a set $\Sigma$ 
of full $\rho$ measure   so that for every $f\in \Sigma$ the
equation~\eqref{HW} with 
initial condition $u(0)=f$ has a   solution 
\begin{equation*}
u\in \mathcal{C}\big(\R\,; X^{0}(\T) \big).
\end{equation*}
For all $t\in \R$, the distribution of the random variable $u(t)$ is $\rho$. 
\end{theo}

In equation~\eqref{HW} the dispersive effect is weak and it seems difficult to deal with the regularities on the support of the measure by deterministic methods. 

\begin{rema} More generally, we can consider the equation
\begin{equation*} 
i\partial_t u-\Lambda^{\alpha} u   =   |u|^{p-1}u,\quad 
(t,x)\in\R\times \mathbb{S}^1,
\end{equation*}
with $\alpha>1$ and $p\geq 1$. Define $X^{\beta}(\S^{1})=\bigcap_{\tau<\beta}H^{\tau}(\T)$. In this case, the situation is better since the series 
\begin{equation*}
\om\mapsto\phi_{\alpha}(\om,x)=\sum_{n\in \Z}\frac{g_{n}(\om)}{(1+|n|)^{{\alpha}/2}}\e^{inx},
\end{equation*}
are so that $\phi_{\alpha}\in L^{2}\big(\Omega;H^{\beta}(\T)\big)$ for all $0<\beta<(\a-1)/2$. Here we should be able to construct solutions 
$$u\in \mathcal{C}\big(\R; X^{(\a-1)/2}(\T) \big).$$
\end{rema}
 
\subsection{The $\boldsymbol {2d}$ NLS on an arbitrary spatial domain}
We assume that $(M,g)$ is either  a two dimensional compact Riemannian manifold without boundary or a bounded domain in $\R^2$. We suppose that ${\rm vol}(M)=1$. This assumption is not a restriction since
we can always reduce the analysis to this case by rescaling the metric. We impose it since some of the computations simplify a little under this assumption.\medskip

Denote by $-\Delta_g$ the Laplace-Beltrami operator on $M$ (with Dirichlet  boundary conditions in the case of a domain in $\R^2$).  Consider the nonlinear Schr\"odinger equation \begin{equation} \label{NLS0_bis}
\left\{
\begin{aligned}
&i\partial_t u+(\Delta_g-1) u   = |u|^{2}u,\quad 
(t,x)\in\R\times M,\\
&u(0,x)= f(x).
\end{aligned}
\right.
\end{equation}
Our aim is to construct a Gibbs measure $\rho$ associated to~\eqref{NLS0_bis} and to construct global solutions to~\eqref{NLS0_bis} on the support of $\rho$. 
Let $(\varphi_n)_{n\geq 0}$ be an orthonormal basis of $L^2(M)$ of eigenfunctions of $-\Delta_g$ associated with increasing eigenvalues $(\lambda_n^2)_{n\geq0}$.
By the Weyl asymptotics $\lambda_n\approx n^{1/2}$.
Let $(g_n(\omega))_{n\geq 0}$ be a sequence of independent standard complex Gaussian variables on a probability space $(\Omega, \mathcal{F},{\bf p})$.
We denote by $\mu$ the Gaussian measure induced by the mapping
$$
\omega\longmapsto \Psi(\om,x):=\sum_{n\geq 0}\frac{g_{n}(\omega)}{(\lambda^2_n+1)^{\frac{1}{2}}}\, \varphi_n(x)\,,
$$
and we can interpret $\mu$ as the Gibbs measure which is associated to the linear part of~\eqref{NLS0_bis}. One may see $\mu$ as a measure on $H^{-s}(M)$ for any fixed $s>0$, and we can  check that $\mu(L^2(M))=0$. Notice also that thanks to the invariance of the Gaussians under rotations the measure $\mu$ is
independent of the choice of the orthonormal basis  $(\varphi_n)_{n\geq 0}$. 

Now, if one wishes to define a Gibbs measure $\rho$ (which is a density measure w.r.t. $\mu$) associated to~\eqref{NLS0_bis}, namely corresponding to the Hamiltonian
\begin{equation*}
H(u)=\frac12\int_M(\vert \nabla u\vert^2+\vert   u\vert^2)+\frac14 \int_M\vert u\vert^4,
\end{equation*}
we have  $\int_M\vert u\vert^4=+\infty$ on the support of $\mu$. 
Therefore a suitable renormalization is needed.
For $u=\sum_{n\geq 0} c_n\varphi_n$, we set
$u_N:=\Pi_{N}u=\sum_{n\leq N} c_n\varphi_n$. Denote by
 \begin{equation}\label{defa}
\a_{N}=\int_{X^{0}(M)}\|u_{N}\|^{2}_{L^{2}(M)}\text{d}\mu(u)=\sum_{0\leq n\leq N}\frac{1}{1+\lambda^2_n}.
 \end{equation}
 Then we define  the renormalized Hamiltonian
\begin{equation*}
H_{N}(u)=\ \int_{M}(|\nabla u|^{2}+\vert u\vert^2) +\frac1{2}\int_{M}|\Pi_{N}u|^{4}-2\a_N\int_{M} \big|\Pi_{N}u\big|^{2}+\a_N^2,
\end{equation*}
and consider the equation 
\begin{equation*}
i\partial_{t}u=\frac{\delta H_{N}}{\delta \ov{u}},
\end{equation*}
which reads 
\begin{equation}\label{NLSN}
\left\{
\begin{aligned}
&i\partial_t u+(\Delta_g -1) u   = F_{N}(u_{N}),\;\;
(t,x)\in\R\times M,\\
&u(0,x)= f(x),
\end{aligned}
\right.
\end{equation}
 where $F_{N}$ stands for 
\begin{equation}\label{gauge_bis}
F_{N}(u_{N})=\Pi_{N}\big(|u_{N}|^{2}u_{N}\big)-2 \a_Nu_{N}.
\end{equation}
 Recall, that since the $L^{2}$ norm of~\eqref{NLSN} is preserved by the flow, one can recover the standard cubic nonlinearity with the change of function $\dis v_{N}(t)=u_{N}(t)\exp(-2\a_Nt)$.

Thanks to the gauge    transform in~\eqref{gauge_bis} we are able  to define the limit equation, when $N\longrightarrow +\infty$.  Set  $ \dis X^{0}(M)=\bigcap_{\s<0}H^{\s}(M)$, then
\begin{prop}\ph\label{prop.NL_bis}
 For all $p\geq 2$ and $\s>2$, the sequence $\big(F_{N}(u_{N})\big)_{N\geq1}$ is  Cauchy  in $L^{p}\big(X^{0}(M),\mathcal{B},d\mu; H^{-\s}(M)\big)$. Namely, for all $p\geq	 2$, there exist ${\eta>0}$ and $C>0$ so that for all $1\leq M<N$,
 \begin{equation*}
\int_{X^{0}(M)}\|F_{N}(u_{N})-F_{M}(u_{M})\|^{p}_{H^{-\s}(M)}\text{d}\mu(u)\leq \frac{C}{M^{\eta}}.
\end{equation*}
We denote by $F(u)$ the limit of this sequence.
\end{prop}

We now   consider the  equation
\begin{equation}\label{NLS}
\left\{
\begin{aligned}
&i\partial_t u +(\Delta_g -1)  u   =   F(u),\;\;
(t,x)\in\R\times M,\\
&u(0,x)= f(x).
\end{aligned}
\right.
\end{equation}
 We now define a Gibbs measure for~\eqref{NLS} as a limit of Gibbs measures for~\eqref{NLSN}. 
Set 
\begin{equation}\label{def.fn}
f_N(u)=\frac1{2}\int_{M}|\Pi_{N}u|^{4}-2\a_N\int_{M} \big|\Pi_{N}u\big|^{2}+\a_N^2,
\end{equation}
and consider the measure
 \begin{equation*}
\text{d}\rho_{N}(u)=C_{N} \e^{- f_N(u)}\text{d}\mu(u),
\end{equation*}
where $C_{N}>0$ is chosen so that $\rho_{N}$ is a probability measure. Then

 \begin{theo}\ph\label{gib}
Let us fix $1\leq p<\infty$.
The sequence $(f_{N}(u))_{N\geq 1}$ converges in $L^p(d\mu(u))$ to some limit denoted by $f(u)$.
Moreover
$$
e^{-f(u)}\in L^p(d\mu(u))\,.
$$
Therefore, we can  define a probability measure on $X^{0}(M)$ by
\begin{equation}\label{def.rho_bis}
\text{d}\rho(u)= C_{\infty}e^{-f(u)}\text{d}\mu(u).
\end{equation}

\end{theo}
The result of Theorem~\ref{gib} is not new. One may find a proof of it in the book by B.~Simon~\cite[page~229]{Simon}.
The approach in~\cite{Simon} is using the control of the singularity on the diagonal of the Green function associated with $\Delta_g^{-1}$. 
Here we present a slightly different proof based on spectral consideration via the Weyl asymptotics. 
We decided to include this proof since the argument is in the spirit of the analysis of the other models considered in this article. \medskip

We are able to use the measure $\rho$ defined in~\eqref{def.rho_bis} to define global dynamics for the equation~\eqref{NLS}.
  \begin{theo}\ph\label{thmNLS_bis} 
There exists a set $\Sigma$ 
of full $\mu$ measure   so that for every $f\in \Sigma$ the
equation~\eqref{NLS} with 
initial condition $u(0)=f$ has a   solution 
\begin{equation*}
u\in \mathcal{C}\big(\R\,; X^{0}(M) \big).
\end{equation*}
For all $t\in \R$, the distribution of the random variable $u(t)$ is $\rho$. 
\end{theo}
\begin{rema}
Another choice of renormalization procedure, namely defining
\begin{equation}\label{gauge_ter}
\widetilde{F}_{N}(u_{N})=\Pi_{N}\big(|u_{N}|^{2}u_{N}\big)-2\| u_N\|_{L^2}^2u_{N}.
\end{equation}
would lead to another limit equation 
$$i\partial_t u +(\Delta_g -1)  u   =   \widetilde{F}(u),$$
with 
$$ \widetilde{F}(u) = \lim_{N\rightarrow + \infty } \widetilde {F}_N(u)\;  "=" \;|u|^2u- 2\|u\|_{L^2}^2u,$$
which is slightly more natural as if $v$ is a ($L^2$-bounded) solution of 
$$ i\partial_t v +(\Delta_g -1)  v= |v|^2 v, $$
then $ u= e^{-2it\| v\|_{L^2}^2 } v$ satisfies 
$$ i\partial_t u +(\Delta_g -1) u= \widetilde{F} (u).$$
However, this renormalization would require another cut-off (see Section~\ref{Sect.7}) because the main contribution of the potential energy in the Hamiltonian is no longer positive.
Finally, yet another renormalization is possible: by the Weyl formula~\eqref{spect} we know $\dis \a_N =\frac{1}{4\pi}\ln N + C+{o}(1)$, and it is easy to check that  it is possible to replace $\a_N$ by its equivalent in the definition of $H_N(u)$ and in~\eqref{gauge_bis}. In this case, the renormalisation does not depend on $M$ (but only on its volume which if fixed to $1$).
\end{rema}
\begin{rema} In Theorems~\ref{gib},~\ref{thmNLS_bis}, the sign of the nonlinearity (defocusing) plays key  a role and we do not know how to define a probability Gibbs measure if $F$ is replaced by $-F$.
\end{rema}

Let us recall some deterministic results on the nonlinear cubic Schr\"odinger equation on a compact surface. The equation is locally well-posed in $H^{s}$ for all $s>0$ when $M=\mathbb{T}^{2}$ (Bourgain \cite{Bou}), in $H^{s}$ for all $s>1/4$ when $M=\S^{2}$ and $H^{s}$ for all $s>1/2$ in the case of a general surface $M$ without boundary (Burq-G\'erard-Tzvetkov \cite{BGT3}).  In the case of a surface with boundary, the equation is locally well-posed in   $H^{s}$ for all $s>2/3$ (Blair-Smith-Sogge \cite{BSS}). In each of the previous cases, thanks to an interpolation argument, S.~Zhong \cite{Zhong} has shown that the (defocusing) equation is globally well-posed in $H^{s}$ for some $1-\delta <s<1$ with $\delta>0$ sufficiently small.

\subsection{Notations and structure of the paper} 

\begin{enonce*}{Notations}
 In this paper $c,C>0$ denote constants the value of which may change
from line to line. These constants will always be universal, or uniformly bounded. For $n\in \Z$, we write $\<n\>=(1+|n|^{2})^{1/2}$. In Section~\ref{Sect.7} we use the notation $[n]=1+|n|$, while in Section~\ref{Sect.8} we write $[n]=1+\lambda^{2}_{n}$.
We will sometimes use the notations $L^{p}_{T}=L^{p}(-T,T)$ for $T>0$. For a manifold $M$, we write $L^{p}_{x}=L^{p}(M)$ and for  $s\in \R$ we define the   Sobolev space $H_{x}^{s}=H^{s}(M)$  by the norm  $\|u\|_{H_{x}^{s}}=\|\big(1-\Delta\big)^{s/2}u\|_{L^{2}(M)}$.  If $E$ is a Banach space and  $\mu$ is a measure on $E$, we write  $L^{p}_{\mu}=L^{p}(\text{d}\mu)$ and $\|u\|_{L^{p}_{\mu}E}=\big\|\|u\|_{E}\big\|_{L^{p}_{\mu}}$. For $M$ a manifold, we define $X^{\s}(M)=\bigcap_{\tau<\s} H^{\tau}(M)$, and if $I\subset \R$ is an interval,  $\mathcal{C}\big(I; X^{\s}(M)\big)=\bigcap_{\tau<\s}\mathcal{C}\big(I; H^{\tau}(M)\big)$. If $X$ is a random variable, we denote by $\mathscr{L}(X)$ its law (its distribution).
\end{enonce*}

The rest of the paper is organised as follows. In Section~\ref{Sect.2} we recall the Prokhorov and the Skorokhod theorems which  are the crucial tools for the proof of our results. In  Section~\ref{Sect.3} we present the general strategy for the construction of the weak stochastic solutions. Each of the remaining sections is devoted to a different  equation.
 
\begin{acknowledgements}
The authors want to thank Arnaud Debussche for pointing out the reference~\cite{DPD}. The second author is very grateful to Philippe Carmona for many clarifications on measures. We also benefitted from discussions with Christian G\'erard and Igor Chueshov. 
\end{acknowledgements}

 \section{The Prokhorov and Skorokhod theorems}
 \label{Sect.2}
In this section, we state two basic results, concerning the convergence of random variables. To begin with,  recall the following definition  (see {\it e.g.}~\cite[page 114]{KoraSinai})
\begin{defi}
Let $S$ be a metric space and $(\rho_{N})_{N\geq 1}$ a family of probability measures on the Borel $\sigma-$algebra $\mathcal{B}(S)$. The family $(\rho_{N})$ on $(S,\mathcal{B}(S))$ is said to be tight if for any $\eps>0$ one can find a compact set $K_{\eps}\subset S$ such that $\rho_{N}(K_{\eps})\geq 1-\eps$ for all $N\geq 1$.
\end{defi}
Then, we have the following compactness criterion (see {\it e.g.}~\cite[page 114]{KoraSinai} or~\cite[page 309]{Kallenberg})
\begin{theo}[Prokhorov]
Assume that the family $(\rho_{N})_{N\geq 1}$ of probability measures on the metric space $S$ is tight. Then it is  weakly compact, {\it i.e.} there is a subsequence 
$(N_k)_{k\geq 1}$ and a limit measure~$\rho_{\infty}$ such that
for every bounded continuous function
$f:S\rightarrow \R$,
$$
\lim_{k\rightarrow\infty}\int_{S}f(x)d\rho_{N_k}(x)=\int_{S}f(x)d\rho_{\infty}(x).
$$
\end{theo}
In fact, the Prokhorov theorem is stronger: In the case where the space $S$ is separable and complete, the converse of the previous statement holds true, 
but we will not use this here.

\begin{rema}\ph\label{marker} Let us make a remark on the case $S=\R^{n}$. The measure given by the theorem allows mass concentration in a point and the tightness condition forbids the escape of mass to infinity. 

The Prokhorov theorem is of different nature compared to the compactness theorems giving the deterministic weak solutions: In the latter case there can be a loss of energy (as mentioned below~\eqref{espace}). A weak limit of $L^{2}$ functions may lose some mass whereas in the Prokhorov theorem  a limit measure is a probability measure.
\end{rema}
We now state the Skorokhod theorem

\begin{theo}[Skorokhod]
Assume that $S$ is a separable metric space. Let $(\rho_{N})_{N\geq 1}$ and $\rho_{\infty}$ be probability measures on  $S$. Assume that $\rho_{N}\longrightarrow \rho_{\infty}$ weakly. 
Then there exists a probability space on which there are $S-$valued random variables $(X_{N})_{N\geq 1}$, $X_{\infty}$ such that $\mathcal{L}(X_{N})=\rho_{N}$ for all $N\geq 1$, $\mathcal{L}(X_{\infty})=\rho_{\infty}$  and $X_{N}\longrightarrow X_{\infty}$ a.s.  
\end{theo}
For a proof, see {\it e.g.}~\cite[page 79]{Kallenberg}.  We illustrate this result with two elementary but significant examples:
\begin{itemize}
\item  Assume that $S=\R$. Let $(X_{N})_{1\leq N\leq \infty}$ be standard Gaussians, {\it i.e.}  $\mathcal{L}(X_{N})=\mathcal{L}(X_{\infty})=\mathcal{N}_{\R}(0,1)$. Then the convergence in law obviously holds, but in general we can not expect the almost sure convergence of the $X_{N}$ to $X_{\infty}$ (define for example $X_{N}=(-1)^{N}X_{\infty}$).
\item  Assume that $S=\R$. Let $(Y_N)_{1\leq N\leq \infty}$ be random variables.
 For any random variable $Y$ on $\R$ we denote by $F_{Y}(t)=P(Y\leq t)$ its cumulative distribution function.   Here  we assume that for all $1\leq N\leq \infty$, 
  $F_{Y_{N}}$ is bijective and continuous, and we prove the Skorokhod theorem in this case.                 
Let $X$ be a r.v. so that $\mathcal{L}(X)$ is the uniform distribution on $[0,1]$ and define the r.v. $\wt{Y}_{N}=F^{-1}_{Y_{N}}(X)$. We now check that the $\wt{Y}_{N}$ satisfy the conclusion of the theorem. To begin with, 
\begin{equation*}
F_{\wt{Y}_{N}}(t)=P(\wt{Y}_{N}\leq t)=P(X\leq F_{{Y}_{N}}(t))=F_{Y_{N}}(t),
\end{equation*}
therefore we have for $1\leq N\leq \infty$, $\mathcal{L}(Y_{N})=\mathcal{L}(\wt{Y}_{N})$. Now if we assume that $Y_{N}\longrightarrow Y_{\infty}$ in law, we have for all $t\in \R$, $F_{Y_{N}}(t)\longrightarrow F_{Y_{\infty}}(t)$ and in particular $\wt{Y}_{N}\longrightarrow \wt{Y}_{\infty}$ almost surely.
\end{itemize}

 \section{General strategy}\label{Sect.3}
Let $({\Omega, \mathcal{F},{\bf p}})$ be a probability space and $\big(g_{n}(\om)\big)_{n\geq 1}$ a sequence of independent complex normalised Gaussians, $g_{n}\in \mathcal{N}_{\C}(0,1)$.  Let $M$ be a Riemanian compact manifold and let $(e_{n})_{n \geq 1}$ be an Hilbertian basis of $L^{2}(M)$ (with obvious changes, we can allow $n\in \Z$). Consider one of the equations mentioned in the introduction. Denote by 
\begin{equation*}
X^{\s}=X^{\s}(M)=\bigcap_{\tau<\s}H^{\tau}(M).
\end{equation*}

\subsection{General strategy of the proof}
The general strategy for proving a global existence result is the following:\medskip

{\bf Step 1: The Gaussian measure  $ \boldsymbol \mu$}: We define a measure $\mu$ on $X^{\s}(M)$ which is invariant by the flow of the linear part of the equation. The index $\s_{c} \in \R$ is determined by the equation and the manifold $M$. Indeed this measure can be defined as $\mu={\p}\circ \phi^{-1}$, where $\phi \in L^2\big(\Omega;\,H^{\sigma}(M)\big)$ for all~$\s<\s_{c}$ is a Gaussian random variable which takes the form 
\begin{equation*}
\phi(\om,x)=\sum_{n\geq 1}\frac{g_{n}(\om)}{\lambda_{n}}e_{n}(x).
\end{equation*}
Here the  $(\lambda_{n})$ satisfy $\lambda_{n}\sim cn^{\alpha}$, $\a>0$ and are given by the linear part and the Hamiltonian structure of the equation. Notice in particular  that for all measurable $f: X^{\s_{c}}(M)  \longrightarrow \R $
\begin{equation}\label{A}
\int_{  X^{\s_{c}}(M)}f(u)\text{d}\mu(u) =\int_{\Omega} f\big(\phi(\om,\cdot)\big)\text{d}{\bf p}(\om).
\end{equation}

{\bf Step 2: The invariant measure $ \boldsymbol{ \rho_{N}}$}:
By working on the Hamiltonian formulation of the equation, we introduce an approximation of the initial problem which has a global flow $\Phi_{N}$, and for which we can construct a measure $\rho_{N}$ on $X^{\s_{c}}(M)$ which has the following properties 
\begin{enumerate}[(i)]
\item The measure $\rho_{N}$ is a probability measure which is absolutly continuous with respect to $\mu$
\begin{equation*}
\text{d}\rho_{N}(u)=\Psi_{N}(u)\text{d}\mu(u).
\end{equation*}
\item The measure $\rho_{N}$ is invariant by the flow $\Phi_{N}$ by the Liouville theorem. 
\item There exists $\Psi\not \equiv 0$ such that for all $p\geq 2$, $\Psi(u)\in L^{p}(\text{d}\mu)$  and 
\begin{equation*}
\Psi_{N}(u) \longrightarrow \Psi(u),\quad \text{in}\quad L^{p}(\text{d}\mu).
\end{equation*}
(In particular $\|\Psi_{N}(u)\|_{L^{p}_{\mu}}\leq C$ uniformly in $N\geq1$.) This enables to define a probability measure on $X^{\s_{c}}(M)$ by
\begin{equation*}
\text{d}\rho(u)=\Psi(u)\text{d}\mu(u),
\end{equation*}
which is formally invariant by the equation.
\end{enumerate}

{\bf Step 3: The  measure $\boldsymbol {\nu_{N}}$}: We abuse notation and write 
\begin{equation*}
  \mathcal{C}\big([-T,T]; X^{\s_{c}}(M)\big)=\bigcap_{\s<\s_c} \mathcal{C}\big([-T,T]; H^{\s}(M)\big).
\end{equation*}

We denote by $\nu_{N} ={\rho_{N}}\circ \Phi_{N}^{-1}$ the measure on $\mathcal{C}\big([-T,T]; X^{\s_{c}}(M)\big)$, defined as the image measure of~$\rho_{N}$ by the map 
\begin{equation*}
 \begin{array}{rcc}
X^{\s_{c}}(M)&\longrightarrow& \mathcal{C}\big([-T,T]; X^{\s_{c}}(M)\big)\\[3pt]
\dis  v&\longmapsto &\dis \Phi_{N}(t)(v).
 \end{array}
 \end{equation*}
In particular, for any measurable $F: \mathcal{C}\big([-T,T]; X^{\s_{c}}(M)\big) \longrightarrow \R $
\begin{equation}\label{transp}
\int_{  \mathcal{C}\big([-T,T]; X^{\s_{c}}\big)}F(u)\text{d}\nu_{N}(u) =\int_{X^{\s_{c}}} F\big(\Phi_{N}(t)(v)\big)\text{d}\rho_{N}(v).
\end{equation}

For each model we consider,  we show that the corresponding sequence of measures $(\nu_{N})$ is tight in $\mathcal{C}\big([-T,T]; H^{\s}(M)\big)$ for all $\s<\s_{c}$. Therefore, for all $\s<\s_{c}$, by the Prokhorov theorem, there exists a measure $\nu_{\s}=\nu$ on ~$\mathcal{C}\big([-T,T]; H^{\s}(M)\big)$ so that the weak convergence holds (up to a sub-sequence): For all  $\s <\s_{c}$ and all bounded continuous   $F: \mathcal{C}\big([-T,T]; H^{\s}(M)\big) \longrightarrow \R $
$$
\lim_{N\rightarrow\infty}\int_{\mathcal{C}\big([-T,T]; H^{\s}\big) }F(u)\text{d}\nu_{N}(u)=\int_{\mathcal{C}\big([-T,T]; H^{\s}\big) }F(u)\text{d}\nu(u).
$$
At this point, observe that if $\s_1<\s_2$, then $\nu_{\s_1}\equiv \nu_{\s_2}$ on $\mathcal{C}\big([-T,T]; H^{\s_1}(M)\big)$. Moreover, by the standard diagonal argument, we can ensure that $\nu$ is a measure on $\mathcal{C}\big([-T,T]; X^{\s_c}(M)\big)$. 

Finally, with the Skorokhod theorem, we can construct a sequence of random variables which converges to a solution of the  initial problem. \medskip

We now state a result which will be useful in the sequel. Assume that $\rho_{N}$ satisfies the properties mentioned in Step 2. 
\begin{prop}\ph\label{Prop.Fonda}
 Let $\s<\s_{c}$. Let  $p\geq2$ and $r>p$. Then for all $N\geq 1$
\begin{equation}\label{Bor1}
\big\|\|u\|_{L^{p}_{T}H^{\s}_{x}}\big\|_{L^{p}_{\nu_{N}}}\leq CT^{1/p}  \big\|\|v\|_{H^{\s}_{x}}\big\|_{L^{r}_{\mu}}.
\end{equation}
Let $q\geq 1$,  $p\geq 2$ and  $r>p$. Then for all $N\geq 1$
\begin{equation}\label{Bor2}
\big\|\|u\|_{L^{p}_{T}L^{q}_{x}}\big\|_{L^{p}_{\nu_{N}}}\leq CT^{1/p}  \big\|\|v\|_{L^{q}_{x}}\big\|_{L^{r}_{\mu}}.
\end{equation}
In case $\Psi_{N}\leq C$, one can take $r=p$ in the previous inequalities.
\end{prop}
\begin{proof}
We apply~\eqref{transp} with the function  $u\longmapsto F(u)=\|u\|^{p}_{L^{p}_{T}H^{\s}_{x}}$. Here and after, we make the abuse of notation
\begin{equation*}
\big\|\|u\|_{L^{p}_{T}H^{\s}_{x}}\big\|_{L^{p}_{\nu_{N}}}=\|u\|_{L^{p}_{\nu_{N}}L^{p}_{T}H^{\s}_{x}}.
\end{equation*}
Then 
\begin{eqnarray}\label{mar1}
\|u\|^{p}_{L^{p}_{\nu_{N}}L^{p}_{T}H^{\s}_{x}}&=&\int_{  \mathcal{C}\big([-T,T]; X^{\s_{c}}\big)}\|u\|^{p}_{L^{p}_{T}H^{\s}_{x}}\text{d}\nu_{N}(u)\nonumber\\
&=&\int_{X^{\s_{c}}} \|\Phi_{N}(t)(v)\|^{p}_{L^{p}_{T}H^{\s}_{x}} \text{d}\rho_{N}(v)\nonumber\\
&=&\int_{X^{\s_{c}}} \Big[\int_{-T}^{T}\|\Phi_{N}(t)(v)\|^{p}_{H^{\s}_{x}} \text{d}t \Big]\text{d}\rho_{N}(v)\nonumber\\
&=&\int_{-T}^{T} \Big[ \int_{X^{\s_{c}}}\|\Phi_{N}(t)(v)\|^{p}_{H^{\s}_{x}} \text{d}\rho_{N}(v)  \Big]\text{d}t,
\end{eqnarray}
where in the last line we used Fubini. Now we use the invariance of $\rho_{N}$ under $\Phi_{N}$, and we deduce that for all $t\in [-T,T]$
\begin{equation*}
 \int_{X^{\s_{c}}}\|\Phi_{N}(t)(v)\|^{p}_{H^{\s}_{x}} \text{d}\rho_{N}(v) = \int_{X^{\s_{c}}}\|v\|^{p}_{H^{\s}_{x}} \text{d}\rho_{N}(v). 
\end{equation*}
Therefore, from~\eqref{mar1} and H\"older we obtain with $1/r_{1}+1/r_{2}=1$
\begin{eqnarray*}
\|u\|^{p}_{L^{p}_{\nu_{N}}L^{p}_{T}H^{\s}_{x}}&=&2T \int_{X^{\s_{c}}}\|v\|^{p}_{H^{\s}_{x}} \text{d}\rho_{N}(v)\\
&=&2T \int_{X^{\s_{c}}}\|v\|^{p}_{H^{\s}_{x}}\Psi_{N}(v) \text{d}\mu(v)\\
&\leq&C\|v\|^{p}_{L^{pr_{1}}_{\mu}H^{\s}_{x}}\|\Psi_{N}(v)\|_{L^{r_{2}}_{\mu}}.
\end{eqnarray*}
Now, let $r>p$, take $r_{1}=r/p$ and we can conclude since $\Psi_{N}\in L^{r_{2}}(\text{d}\mu)$. \\
For the proof of~\eqref{Bor2}, we proceed similarly. We take $F(u)=\|u\|^{p}_{L^{p}_{T}L^{q}_{x}}$ in~\eqref{transp}, and use the same arguments    as previously.
\end{proof}

\subsection{Some deterministic estimates} We now state an   interpolation result, which will be useful for the study of each model. Consider  $(e_{n})_{n\geq 1}$   a Hilbertian basis of $L^{2}=L^{2}(M)$ of eigenfunctions of~$\Delta$: 
\begin{equation*}
-{\Delta} e_{n}=\lambda^{2}_{n}e_{n}, \quad n\geq 1.
\end{equation*}
For $\dis u=\sum_{n\geq 1}\a_{n}e_{n}$, we define the spectral projector
\begin{equation*}
\Delta_{j}u=\sum_{n\geq 1\;:\; 2^{j}\leq \<\lambda_{n}\><2^{j+1}}\alpha_{n}e_{n},
\end{equation*}
so that we have $\dis u=\sum_{j\geq 0} \Delta_{j}u$ and for $\s\in \R$
\begin{equation*}
\dis C_{1} 2^{j\s}\|\Delta_{j}u\|_{L^{2}}\leq \|\Delta_{j}u\|_{H^{\s}(M)}\leq C_{2} 2^{j\s}\|\Delta_{j}u\|_{L^{2}}.
\end{equation*}
Define the space $W_{T}^{1,p}$ by the norm $\|u\|_{W_{T}^{1,p}}=\|u\|_{L_{T}^{p}}+\|\partial_{t}u\|_{L_{T}^{p}}$. Then
\begin{lemm}\ph\label{lemm.interp}
Let $T>0$ and $p\in[1,+\infty]$. Assume that $u\in  L^{p}\big([-T,T]; L^{2}\big)$ and $\partial_{t}u\in  L^{p}\big([-T,T]; L^{2}\big)$. Then $u\in  L^{\infty}\big([-T,T]; L^{2}\big)$ and 
\begin{equation*}
\|u\|_{L^{\infty}_{T}L^{2}}\leq C \|u\|^{1-1/p}_{L^{p}_{T}L^{2}} \| u\|^{1/p}_{W_{T}^{1,p}L^{2}}.
\end{equation*}
\end{lemm}
\begin{proof}
Let $\gamma\in L^{2}(M)$ be so that $\|\gamma\|_{L^{2}}=1$, and define $v(t)=\<u(t),\gamma\>$. Then we clearly have 
\begin{equation*}
\|v\|_{L^{p}_{T}}\leq \|u\|_{L^{p}_{T}L^{2}},\qquad \|\partial_{t}v\|_{L^{p}_{T}}\leq \|\partial_{t}u\|_{L^{p}_{T}L^{2}},
\end{equation*}
and from the Gagliardo-Nirenberg inequality we deduce 
\begin{equation}\label{gn*}
\|v\|_{L^{\infty}_{T}}\leq C \|v\|^{1-1/p}_{L^{p}_{T}} \|v\|^{1/p}_{W^{1,p}_{T}}\leq C \|u\|^{1-1/p}_{L^{p}_{T}L^{2}} \| u\|^{1/p}_{W^{1,p}_{T}L^{2}}.
\end{equation}
Now from~\eqref{gn*} we get 
\begin{eqnarray*}
\|u\|_{L^{\infty}_{T}L^{2}}&=&\sup_{t\in [-T,T]} \|u(t)\|_{L^{2}}\nonumber \\
&=&\sup_{t\in [-T,T]}\sup_{\|\gamma\|_{L^{2}}=1} v(t)\\
&=&\sup_{\|\gamma\|_{L^{2}}=1}\sup_{t\in [-T,T]} v(t)\leq C \|u\|^{1-1/p}_{L^{p}_{T}L^{2}} \| u\|^{1/p}_{W^{1,p}_{T}L^{2}}.
\end{eqnarray*}
This completes the proof of Lemma~\ref{lemm.interp}.
\end{proof}

Denote by $H^{\s}=H^{\s}(M)$. Using the previous result  we can prove
\begin{lemm}\ph\label{lem.33} 
Let $T>0$ and $p\in[1,+\infty]$. Let $-\infty<\s_{2}\leq \s_{1} <+\infty$ and assume that $u\in  L^{p}\big([-T,T]; H^{\s_{1}}\big)$ and $\partial_{t}u\in  L^{p}\big([-T,T]; H^{\s_{2}}\big)$. Then for  all $\eps>\s_{1}/p-\s_{2}/p$, $u\in  L^{\infty}\big([-T,T]; H^{\s_{1}-\eps}\big)$ and 
\begin{equation}\label{res.1}
\|u\|_{L^{\infty}_{T}H^{\s_{1}-\eps}}\leq C\|u\|^{1-1/p}_{L^{p}_{T}H^{\s_{1}}} \|u\|^{1/p}_{W_{T}^{1,p}H^{\s_{2}}}.
\end{equation}
Moreover, there exists $\eta>0$ and $\theta\in[0,1]$ so that for all $t_{1},t_{2}\in [-T,T]$
\begin{equation*} 
\|u(t_{1})-u(t_{2})\|_{H^{\s_{1}-2\eps}}\leq C|t_{1}-t_{2}|^{\eta}\|u\|^{1-\theta}_{L^{p}_{T}H^{\s_{1}}} \|  u\|^{\theta}_{W_{T}^{1,p}H^{\s_{2}}}.
\end{equation*}
\end{lemm}
\begin{proof}
We use the frequency decomposition as recalled at the beginning of the section, and apply Lemma~\ref{lemm.interp} to $\Delta_{j}u$
\begin{eqnarray*}
\|\Delta_{j}u\|_{L^{\infty}_{T}H^{\s_{1}-\eps}}&\leq & C 2^{j(\s_{1}-\eps)}\|\Delta_{j}u\|_{L^{\infty}_{T}L^{2}}\\
&\leq &C2^{j(\s_{1}-\eps)}\|\Delta_{j}u\|^{1-1/p}_{L^{p}_{T}L^{2}} \big(\|\partial_{t} \Delta_{j}u\|_{L^{p}_{T}L^{2}}+\| \Delta_{j}u\|_{L^{p}_{T}L^{2}}\big)^{1/p}\\
&\leq &  C2^{j(\s_{1}-\eps)}2^{-j\s_{1}(1-1/p)}2^{-j\s_{2}/p}\|\Delta_{j}u\|^{1-1/p}_{L^{p}_{T}H^{\s_{1}}} \|  \Delta_{j}u\|^{1/p}_{W^{1,p}_{T}H^{\s_{2}}}\\
&\leq &  C2^{-j(\eps-\s_{1}/p+\s_{2}/p)}\|u\|^{1-1/p}_{L^{p}_{T}H^{\s_{1}}} \|  u\|^{1/p}_{W^{1,p}_{T}H^{\s_{2}}}.
\end{eqnarray*}
 This inequality together with $\dis \|u\|_{L^{\infty}_{T}H^{\s_{1}-\eps}}\leq \sum_{j\geq0 }\|\Delta_{j}u\|_{L^{\infty}_{T}H^{\s_{1}-\eps}}$ yields~\eqref{res.1}.
By H\"older we get
\begin{equation}\label{hold}
\|u(t_{1})-u(t_{2})\|_{H^{\s_{2}}}=\|\int_{t_{1}}^{t_{2}}\partial_{\tau}u(\tau)\text{d}\tau\|_{H^{\s_{2}}}\leq |t_{1}-t_{2}|^{1-1/p} \|\partial_{t} u\|_{L^{p}_{T}H^{\s_{2}}}.
\end{equation}
Next  by interpolation, there exists $\theta_{0}\in (0,1)$ so that 
\begin{eqnarray*}
\|u(t_{1})-u(t_{2})\|_{H^{\s_{1}-2\eps}}&\leq& \|u(t_{1})-u(t_{2})\|^{1-\theta_{0}}_{H^{\s_{1}-\eps}} \|u(t_{1})-u(t_{2})\|^{\theta_{0}}_{H^{\s_{2}}}\\
&\leq& C \|u\|^{1-\theta_{0}}_{L_{T}^{\infty}H^{\s_{1}-\eps}} \|u(t_{1})-u(t_{2})\|^{\theta_{0}}_{H^{\s_{2}}},
\end{eqnarray*}
and the result follows from this latter inequality combined with~\eqref{res.1} and~\eqref{hold}.
\end{proof}

 \section{The nonlinear Schr\"odinger equation on the three dimensional sphere}\label{Sect.4}
\subsection{The setting} Let $\S^{3}$ be the unit sphere in $\R^{4}$. Consider the  non linear Schr\"odinger equation
 \begin{equation}\label{NLS_bis}
\left\{
\begin{aligned}
&i\partial_{t}u+ (\Delta -1)u= |u|^{r-1}u, \quad   (t,x)\in \R\times \S^{3},\\
&u(0,x)=  f(x) \in H^{\s}(\S^{3}),
\end{aligned}
\right.
\end{equation}
 where $\Delta=\Delta_{\S^{3}}$ stands for the Laplace-Beltrami operator, and where $1\leq r<5$.
 In the sequel 
we consider functions which only depend on the geodesic distance to the north pole, these are called zonal functions. Denote by $Z(\S^{3})$ this space. Roughly speaking, this is the same type of reduction as restricting to radial functions in $\R^{3}$. Denote by $L^{2}_{rad}(\S^{3})=L^{2}(\S^{3})\cap Z(\S^{3})$. We endow this space with the natural norm  
\begin{equation*}
\|f\|_{L^{2}_{rad}(\S^{3})}=\Big(\int_{\S^{3}}|f|^{2}\Big)^{\frac12}=\Big(\int_{0}^{\pi}|f(x)|^{2}(\sin x)^{2}\text{d}x\Big)^{\frac12},
\end{equation*} 
where $x\in [0,\pi]$ represents the geodesic distance to the north pole of $\S^{3}$. The operator $\Delta$ can be restricted to $L^{2}_{rad}$, and it reads
\begin{equation*}
\Delta=\frac{\partial^{2}}{\partial x^{2}}+\frac{2}{\tan x}\frac{\partial}{\partial x}.
\end{equation*}
One of the main interests to restrict to zonal functions, is that the eigenvalues of $\Delta$ in $L^{2}_{rad}(\S^{3})$ are simple. The family $(P_{n})_{n\geq 1}$ defined in~\eqref{def.p} is a Hilbertian basis of $L^{2}_{rad}(\S^{3})$ of eigenfunction of the Laplacian: For all $n\geq 1$, $-\Delta P_{n}=(n^{2}-1)P_{n}$. We define the operator $\Lambda=\big(1-\Delta\big)^{\frac12}$, in particular~$\Lambda P_{n}=nP_{n}$.

  Let us define the complex vector space $\dis {E_N={\rm span}\Big( (  P_{n})_{1\leq n\leq N}\Big)}$. Then we introduce  a smooth version of the usual spectral projector  on $E_{N}$.   Let $\chi\in \mathcal{C}_{0}^{\infty}(-1,1)$, so that $\chi\equiv 1$ on $(-1/2,1/2)$. We then define
\begin{equation*} 
S_{N}\Big(\sum_{n\geq 1}c_n P_{n}\Big)=\chi(\frac{\Lambda}{N})\sum_{n\geq 1}c_n P_{n}=\sum_{n\geq 1}\chi(\frac{n}{N})c_n P_{n}.
\end{equation*}
One of the advantages of this operator compared with  the usual spectral projector, is the following result. See Burq-G\'erard-Tzvetkov~\cite{BGT2} for a proof.
\begin{lemm}\ph\label{lem.sn}  Let $1<p<\infty$. Then $S_{N}: L^{p}(\S^{3})\longrightarrow  L^{p}(\S^{3})$ is continuous and there exists $C>0$ so that for all $N\geq 1$, 
\begin{equation*} 
\|S_{N}\|_{L^{p}(\S^{3})\to L^{p}(\S^{3})}\leq C.
\end{equation*}
Moreover, for all $f\in L^{p}(\S^{3})$, $S_{N}f\longrightarrow f$ in $L^{p}(\S^{3})$, when $N\longrightarrow +\infty$.
\end{lemm}

 \subsection{Preliminaries: Some   estimates}

 In the sequel, we will need a particular case of Sogge's estimates.

\begin{lemm}\ph
The following bounds hold true for $n\geq 1$ 
 \begin{equation}\label{1.10}
\|P_{n}\|_{L^{p}(\S^{3})}\leq\left\{\begin{array}{ll} 
 C n^{1/2-1/p},\quad &\text{if} \quad 2\leq p\leq 4, \\[6pt]  
 C n^{1-3/p},  &\text{if} \quad 4\leq p\leq \infty.
\end{array} \right.
\end{equation}
\end{lemm}

\begin{proof}
The bound for $p=\infty$ is clear by the definition~\eqref{def.p}. The case $p=4$ is proved in~\cite[Lemma 10.1]{Tzvetkov2}
thanks to the formula 
\begin{equation*}
P_{k}P_{\ell}=\sqrt{\frac2\pi}\sum_{j=1}^{\min{(k,\ell)}}P_{|k-\ell|+2j-1},\quad k,\ell\geq 1.
\end{equation*}
The general case follows from H\"older.
\end{proof}

The next Lemma (Khinchin inequality) shows a smoothing property of the random series in the $L^{p}$ spaces. See {\it e.g.}~\cite[Lemma 4.2]{BT2} for the  proof.

\begin{lemm}\ph
There exists $C>0$ such that for all $p\geq 2$ and $(c_{n})\in \ell^{2}(\N)$
\begin{equation}\label{contrac}
\|\sum_{n\geq 1}g_{n}(\om)\,c_{n}\|_{L_{\p}^{p}}\leq C\sqrt{p}\Big(\sum _{n\geq 1}|c_{n}|^{2}\Big)^{\frac12}.
\end{equation}
\end{lemm}

Define $\mu=\p\circ \phi^{-1}$, with $\phi$ given in~\eqref{phi.NLS}. Then we can state
\begin{lemm}\ph Let $\s<\frac12$, then there exists $C>0$ so that for all $p\geq 2$
\begin{equation}\label{CT1}
\big\|\|v\|_{H^{\s}_{x}}\big\|_{L^{p}_{\mu}}\leq C \sqrt{p}.
\end{equation}
Let $2\leq q<6$, then there exists $C>0$ so that for all $p\geq q$
\begin{equation}\label{CT2}
\big\|\|v\|_{L^{q}_{x}}\big\|_{L^{p}_{\mu}}\leq C \sqrt{p}.
\end{equation}
\end{lemm}

\begin{proof}
We prove~\eqref{CT1}. Let $\s<1/2$ and apply~\eqref{contrac} to $\dis (1-\Delta)^{\s/2} \phi=\sum_{n\geq 1}\frac{g_{n}}{n^{1-\s}}P_{n}$. Then 
\begin{equation*}
\|(1-\Delta)^{\s/2} \phi\|_{L_{\p}^{p}} \leq C\sqrt{p}\Big(\sum _{n\geq 1}\frac{|P_{n}|^{2}}{n^{2(1-\s)}}\Big)^{\frac12}.
\end{equation*}
Take the $L^{2}(\S^{3})$ norm of the previous inequality, and by the Minkowski inequality the claim follows.  The proof of~\eqref{CT2} is similar, using~\eqref{1.10} and the Minkowski inequality.
\end{proof}
We will also need the next result. See~\cite[Lemma 3.3]{BTT} for the  proof.
\begin{lemm}\ph
Let $2\leq q<6$. Then there exist $c,C>0$ so that for all $N\geq 1$ and $\lambda>0$
\begin{equation*}
\mu\big(\,u\in X^{1/2}(\S^{3})\;:\;\|S_{N}u\|_{L^{q}(\S^{3})}>\lambda\big)\leq C\e^{-c\lambda^{2}}.
\end{equation*}
Moreover  there exist $\alpha,c,C>0$ so that for all $1\leq M\leq N$ and $\lambda>0$
\begin{equation}\label{mesure}
\mu\big(\,u\in X^{1/2}(\S^{3})\;:\;\|S_{N}u-S_{M}u\|_{L^{q}(\S^{3})}>\lambda\big)\leq C\e^{-cM^{\a}\lambda^{2}}.
\end{equation}
\end{lemm}

 \subsection{A convergence result}
Let $1\leq r<5$ and  recall the definition~\eqref{defG} of $G$. Let $N\geq 1$ and set $G_{N}=\b_{N}G\circ S_{N}$,     where  $\b_{N}>0$ is chosen such that 
\begin{equation*}
\text{d}{\rho_{N} }(u)=G_{N}(u)\text{d}\mu(u),
\end{equation*} 
defines    a probability measure   on $X^{1/2}(\S^{3})$.   The next statement shows that we can pass to the limit $N\longrightarrow +\infty$ in the previous expression.

\begin{prop}\ph\label{prop.46} 
Let  $p\in [1,\infty[$, then 
\begin{equation*} 
G_{N}(u)\longrightarrow G(u),\quad \text{in}\quad L^{p}(\text{d}\mu(u)),
\end{equation*}
when $N\longrightarrow +\infty$.  
\end{prop}


In particular, for any Borel set $A\subset X^{1/2}_{}(\S^{3})$, $\dis \lim_{N\to \infty} \rho_{N}(A)=\rho(A)$. Observe that for all $N\geq 1$, $\rho_{N}\big(X^{1/2}\backslash X^{1/2}_{rad}\big)=0$, as well as $\rho\big(X^{1/2}\backslash X^{1/2}_{rad}\big)=0$.

\begin{proof}
Let $q<6$. By~\eqref{mesure}, we deduce that $\|S_{N}u\|_{L^{q}_{x}} \longrightarrow \|u\|_{L^{q}_{x}}$ in mesure, w.r.t. $\mu$, hence $G_{N}(u)=G(S_{N}u)\longrightarrow G(u)$. In other words, if for  $\eps>0$ and $N\geq 1$ we denote by
\begin{equation*}
A_{N,\eps}=\big\{  \,u\in X^{1/2}(\S^{3})\;:\;|G_{N}(u)-G(u)| \leq\eps  \},
\end{equation*}
then $\mu({A^{c}_{N,\eps}})\longrightarrow 0$, when $N\longrightarrow +\infty$. Now use that $0\leq G,G_{N}\leq 1$
\begin{eqnarray*}
\|G-G_{N}\|_{L^{p}_{\mu}}&\leq &\|(G-G_{N}){\bf 1}_{A_{N,\eps}}\|_{L^{p}_{\mu}}+\|(G-G_{N}){\bf 1}_{{A^{c}_{N,\eps}}}\,\|_{L^{p}_{\mu}}\\
&\leq & \eps \big(\,\mu(A_{N,\eps}\,)\big)^{1/p}+2\big(\,\mu({A^{c}_{N,\eps}})\,\big)^{1/p}\leq C\eps,
\end{eqnarray*}
for $N$ large enough. This ends the proof.
\end{proof}
\subsection{Study  of the measure $\boldsymbol {\nu_{N}}$}
 
 Let $N\geq 1$. We then consider the following approximation of~\eqref{NLS_bis}
 \begin{equation}\label{ODE}
\left\{
\begin{aligned}
&i\partial_{t}u+ (\Delta -1)u=  S_{N}\big( |S_{N} u|^{r-1}S_{N}u  \big),\;\;
(t,x)\in\R\times \S^{3},\\
&u(0,x)=  v(x) \in X_{rad}^{1/2}(\S^{3}).
\end{aligned}
\right.
\end{equation}

 The main motivation to introduce this system is the following proposition, which is directly inspired from~\cite[Section 8]{BTT}. Therefore we omit the proof.

\begin{prop}\ph\label{prop.inv}
The equation~\eqref{ODE} has a global flow $\Phi_{N}$. Moreover, the measure $\rho_{N}$ is invariant under  $\Phi_{N}$  : For any Borel set $A\subset X_{rad}^{1/2}(\S^{3})$ and for all $t\in \R$, $\rho_{N}\big(\Phi_{N}(t)(A)\big)=\rho_{N}\big( A\big)$.
\end{prop}
In particular if $\mathscr{L}_{X_{rad}^{1/2}}(v)=\rho_{N}$ then for all $t\in \R$, $\mathscr{L}_{X_{rad}^{1/2}}(\Phi_{N}(t)v)=\rho_{N}$. 

\begin{rema} Observe that~\eqref{ODE} is not a finite dimensional system of  ODE, but its flow restricted to high frequencies is linear.
\end{rema}

We denote by $\nu_{N}$ the measure on $\mathcal{C}\big([-T,T]; X^{1/2}(\S^{3})\big)$, defined as the image measure of $\rho_{N}$ by the map 
\begin{equation*}
 \begin{array}{rcc}
X^{1/2}(\S^{3})&\longrightarrow& \mathcal{C}\big([-T,T]; X^{1/2}(\S^{3})\big)\\[3pt]
\dis  v&\longmapsto &\dis \Phi_{N}(t)(v).
 \end{array}
 \end{equation*}



\begin{lemm}\ph Let $\s<\frac12$ and $p\geq 2$. Then for all $N\geq1$
\begin{equation}\label{Bor1.NLS}
\big\| \| u\|_{L^{p}_{T}H^{\s}_{x}}\big\|_{L^{p}_{\nu_{N}}}\leq  C.
\end{equation}
Let $2\leq q<6$ and  $p\geq q$. Then for all $N\geq1$
\begin{equation}\label{Bor2.NLS}
\big\| \|u\|_{L^{p}_{T}L^{q}_{x}}\big\|_{L^{p}_{\nu_{N}}} \leq  C.
\end{equation}
\end{lemm}

\begin{proof}
By~\eqref{Bor1} and the fact that  $G\leq C$ we already have 
\begin{equation*}
\|  u\|_{L^{p}_{\nu_{N}}L^{p}_{T}H^{\s}_{x}}\leq  C\|  v\|_{L^{p}_{\mu}H^{\s}_{x}}=C \|\phi\|_{L^{p}_{\p}H^{\s}_{x}},
\end{equation*}
where we used   the transport property~\eqref{A} with the map $\dis f:u\longmapsto \|u\|^{p}_{H^{\s}_{x}}$. Finally we conclude with~\eqref{CT1}.  \\
For the proof of~\eqref{Bor2.NLS}, we use~\eqref{Bor2} and~\eqref{CT2}.
\end{proof}

\begin{lemm}\ph Let $\s>\frac32$ and $p\geq 2$. Then there exists $C>0$ so that for all $N\geq1 $
\begin{equation}\label{Bor3.NLS}
\big\|\| u\|_{W^{1,p}_{T}H^{-\s}_{x}}\big\|_{L^{p}_{\nu_{N}}}\leq  C.
\end{equation}
\end{lemm}  
  
\begin{proof}
By~\eqref{Bor1.NLS} it is enough to show that $\dis \big\|  \| \partial_{t}u\|_{L^{p}_{T}H^{-\s}_{x}}\big\|_{L^{p}_{\nu_{N}}}\leq  C. $ By definition 
\begin{eqnarray*}
\|\partial_{t}u\|^{p}_{L^{p}_{\nu_{N}}L^{p}_{T}H^{-\s}_{x}}&=& \int_{\mathcal{C}\big([-T,T]; X^{1/2}(\S^{3})\big)}\|\partial_{t} u\|^{p}_{L^{p}_{T}H^{-\s}_{x}} \text{d}\nu_{N}(u)\\
&=& \int_{X^{1/2}(\S^{3})}\|\partial_{t} \Phi_{N}(t)(v)\|^{p}_{L^{p}_{T}H^{-\s}_{x}} \text{d}\rho_{N}(v).
\end{eqnarray*}
Now we use that $w_{N}:=\Phi_{N}(t)(v)$ satisfies~\eqref{ODE}  to get
 \begin{equation*} 
\|\partial_{t} w_{N}\|_{L^{p}_{\rho_{N}}L^{p}_{T}H^{-\s}_{x}}\leq \|(\Delta -1) w_{N}\|_{L^{p}_{\rho_{N}}L^{p}_{T}H^{-\s}_{x}}+\|S_{N}\big( |S_{N}w_{N}|^{r-1}S_{N}w_{N}\big)\|_{L^{p}_{\rho_{N}}L^{p}_{T}H^{-\s}_{x}},
\end{equation*}
which in turn implies  
 \begin{equation}\label{sim} 
\|\partial_{t} u\|_{L^{p}_{\nu_{N}}L^{p}_{T}H^{-\s}_{x}}\leq \|(\Delta -1) u\|_{L^{p}_{\nu_{N}}L^{p}_{T}H^{-\s}_{x}}+\|S_{N}\big( |S_{N}u|^{r-1}S_{N}u\big)\|_{L^{p}_{\nu_{N}}L^{p}_{T}H^{-\s}_{x}}.
\end{equation}
Firstly, by~\eqref{Bor1.NLS} we get for $\s>1/2$
\begin{equation}\label{CC}
\|(\Delta -1) u\|_{L^{p}_{\nu_{N}}L^{p}_{T}H^{-\s}_{x}}= \|u\|_{L^{p}_{\nu_{N}}L^{p}_{T}H^{2-\s}_{x}}\leq C.
\end{equation}
Then  by Sobolev, since $\s>3/2$, we get $\dis \|g\|_{H_{x}^{-\s}} \leq C \|g\|_{L^{1}_{x}}$. Therefore 
\begin{eqnarray*}
\|S_{N}\big( |S_{N}u|^{r-1}S_{N}u\big)\|_{L^{p}_{\nu_{N}}L^{p}_{T}H^{-\s}_{x}}&\leq  &C\|S_{N}\big( |S_{N}u|^{r-1}S_{N}u\big)\|_{L^{p}_{\nu_{N}}L^{p}_{T}L^{1}_{x}}\nonumber\\
 &\leq&C\|S_{N}u\|^{r}_{L^{rk}_{\nu_{N}}L^{rk}_{T}L^{r}_{x}}\nonumber\\
  &\leq&C\|u\|^{r}_{L^{rk}_{\nu_{N}}L^{rk}_{T}L^{r}_{x}},
\end{eqnarray*}
where we used twice the  continuity of $S_{N}$ on $L^{p}_{x}$ spaces (see Lemma~\ref{lem.sn}). Now, since $1\leq r<5$ we can apply~\eqref{Bor2.NLS} and this together with~\eqref{CC} implies the result.

\end{proof}

\subsection{The convergence argument}~

\begin{prop}\ph\label{Prop.tight}
Let $T>0$ and $\s<\frac12$. Then  the family of measures 
$$ \nu_{N}=\mathscr{L}_{\mathcal{C}_{T}H^{\s}}\big(u_{N}(t);t\in [-T,T]\big)_{N\geq 1}$$
 is tight in $\mathcal{C}\big([-T,T]; H^{\s}(\S^{3})\big)$.
\end{prop}

\begin{proof} 
Let $\s<\frac12$. Fix $\s<s'<s''<\frac12$ and $\a>0$. We  define the  space $\nolinebreak[4] \mathcal{C}_{T}^{\a}H^{s'}=\mathcal{C}^{\a}\big([-T,T]; H^{s'}(\S^{3})\big)$ by the norm 
\begin{equation*}
\|u\|_{\mathcal{C}_{T}^{\a}H^{s'}}=\sup_{t_{1},t_{2}\in [-T,T], \,t_{1}\neq t_{2}}\frac{\|u(t_{1})-u(t_{2})\|_{H_{x}^{s'}}}{|t_{1}-t_{2}|^{\a}}+\|u\|_{L^{\infty}_{T}H_{x}^{s'}},
\end{equation*}
and it is classical that the  embedding $\mathcal{C}_{T}^{\a}H^{s'}\subset \mathcal{C}\big([-T,T]; H^{\s}(\S^{3})\big)$ is compact.
 
 We now claim that   there exists $0<\a\ll1 $    so that for all $p\geq 1$ we have the bound 
 \begin{equation}\label{BOR}
\|  u\|_{L^{p}_{\nu_{N}}\mathcal{C}_{T}^{\a}H^{s'}}\leq  C.
\end{equation}
Indeed  apply Lemma~\ref{lem.33} with   $\s_{1}=s''$ and $\s_{2}=\s$. Then  for $p$ large enough we have 
\begin{equation*}
\|u\|_{ \mathcal{C}^{\a}_{T}H^{s'}} \leq C \|u\|^{1-\theta}_{L^{p}_{T}H^{s''}} \| u\|^{\theta}_{W^{1,p}_{T}H^{-\s}}\leq C \|u\|_{L^{p}_{T}H^{s''}} +C\|  u\|_{W^{1,p}_{T}H^{-\s}},
\end{equation*}
for some small $\a>0$. By~\eqref{Bor1.NLS} and~\eqref{Bor3.NLS} we then deduce  $\|  u\|_{L^{p}_{\nu_{N}}\mathcal{C}^{\a}_{T}H^{s'}}\leq  C$. (The fact that~\eqref{BOR} is indeed true for any $p\geq1 $ is a consequence of H\"older.) 
Let $\delta>0$ and  define 
\begin{equation*}
K_{\delta}=\big\{u\in \mathcal{C}_{T}H^{\s} \;\;s.t.\;\;\|u\|_{\mathcal{C}_{T}^{\a}H^{s'}}\leq \delta^{-1}\big\}.
\end{equation*}
Thanks to the previous considerations, the set $K_{\delta}$ is compact. Finally, by   Markov and~\eqref{BOR} we get that 
\begin{equation*}
\nu_{N}(K^{c}_{\delta})\leq \delta \, \|u\|_{L^{1}_{\nu_{N}}\mathcal{C}_{T}^{\a}H^{s'} } \leq \delta  C,
\end{equation*}
which shows  the tightness of $(\nu_{N})$.     
 \end{proof}


The result of Proposition~\ref{Prop.tight} enables us to use the Prokhorov theorem: For each $T>0$ there exists a sub-sequence $\nu_{N_{k}}$ and a measure $\nu$ on the space $\mathcal{C}\big([-T,T]; X^{1/2}(\S^{3})\big)$ so that for all $\tau<1/2 $ and  all bounded continuous function $F: \mathcal{C}\big([-T,T]; H^{\tau}(\S^{3}) \big)\longrightarrow \R$
$$
\int_{\mathcal{C}\big([-T,T]; H^{\tau}\big) }F(u)\text{d}\nu_{N_{k}}(u)\longrightarrow  \int_{\mathcal{C}\big([-T,T]; H^{\tau}\big) }F(u)\text{d}\nu(u).
$$
By the Skorokhod theorem, there  exists a probability space $(\widetilde{\Omega},\widetilde{\mathcal{F}},\widetilde{\bf p})$, a sequence of random variables~$(\widetilde{u}_{N_{k}})$ and a random variable~$\widetilde{u}$ with values in $\mathcal{C}\big([-T,T]; X^{1/2}(\S^{3})\big)$ so that 
\begin{equation}\label{loi}
\mathscr{L}\big(\widetilde{u}_{N_{k}};t\in [-T,T]\big)=\mathscr{L}\big(u_{N_{k}};t\in [-T,T]\big)=\nu_{N_{k}}, \quad \mathscr{L}\big(\widetilde{u};t\in [-T,T]\big)=\nu,
\end{equation}
 and for all $\tau <1/2$
\begin{equation}\label{CV}
\widetilde{u}_{N_{k}}\longrightarrow \widetilde{u},\quad \;\;\widetilde{\bf p}-\text{a.s. in}\;\; \mathcal{C}\big([-T,T]; H^{\tau}(\S^{3})\big).
\end{equation}

We now claim that $\mathscr{L}_{X^{1/2}}( {u}_{N_{k}}(t))=\mathscr{L}_{X^{1/2}}(\widetilde{u}_{N_{k}}(t))=\rho_{N_{k}}$, for all~$t\in [-T,T]$ and $k\geq 1$. For all~$t\in [-T,T]$, the evaluation map 
\begin{equation*}
 \begin{array}{rcc}
R_{t}\, :\,\mathcal{C}\big([-T,T]; X^{1/2}(\S^{3})\big)  &\longrightarrow&X^{1/2}(\S^{3}) \\[3pt]
\dis  u&\longmapsto &\dis u(t,.),
 \end{array}
 \end{equation*}
is well defined and continuous. Then, for all $t\in [-T,T]$, $u_{N_{k}}(t)$ and $\wt{u}_{N_{k}}(t)$ have same distribution. Let us now determine the distribution of $u_{N_{k}}(t)$ which we denote by $\nu^{t}_{N_{k}}$. By definition of $\nu_{N_{k}}^{t}$ and $\nu_{N_{k}}$ we have for all measurable  $F: X^{1/2}(\S^{3}) \longrightarrow \R $
\begin{eqnarray*}
\int_{  X^{1/2}(\S^{3})}F(v)\text{d}\nu^{t}_{N_{k}}(v) &=&\int_{\mathcal{C}\big([-T,T]; X^{1/2}(\S^{3})\big)} F(R_{t}u)\text{d}\nu_{N_{k}}(u)\\
 &=&\int_{  X^{1/2}(\S^{3})}F\big(R_{t}\Phi_{N_{k}}(\cdot)w\big)\text{d}\rho_{N_{k}}(w)\\
  &=&\int_{  X^{1/2}(\S^{3})}F\big(\Phi_{N_{k}}(t)(w)\big)\text{d}\rho_{N_{k}}(w).
\end{eqnarray*} 
From the invariance of $\rho_{N_{k}}$ under $\Phi_{N_{k}}$ we get  $\nu^{t}_{N_{k}}=\rho_{N_{k}}$.

Thus from~\eqref{CV} and the convergence property of Proposition~\ref{prop.46},  we deduce that 
\begin{equation}\label{LOI}
\mathscr{L}_{X^{1/2}}( \wt{u}(t)) =\rho,\quad \forall\,t\in [-T,T].
\end{equation}
Let $k\geq 1$ and $t\in \R$ and consider the r.v. $X_{k}$ given by
\begin{equation*}
X_{k}=i\partial_{t}u_{N_{k}}+ (\Delta -1) u_{N_{k}}-   S_{N_{k}}\Big( | S_{N_{k}}u_{N_{k}}|^{r-1} S_{N_{k}}u_{N_{k}}  \Big).
\end{equation*}
Define  $\widetilde{X}_{k}$ similarly to  $X_{k}$ with $u_{N_{k}}$ replaced with $\widetilde{u}_{N_{k}}$. Then by~\eqref{loi}, $\mathscr{L}_{\mathcal{C}_{T}X^{1/2}}(\widetilde{X}_{N_{k}})=\mathscr{L}_{\mathcal{C}_{T}X^{1/2}}(X_{N_{k}})=\delta_{0}$, in other words, $\widetilde{X}_{k}=0$ ${\bf \widetilde{p}}$\,--\,a.s. and 
 $\widetilde{u}_{N_{k}}$ satisfies the following equation ${\bf  \widetilde{p}}$\,--\,a.s.
\begin{equation}\label{tilde}
i\partial_{t}\widetilde{u}_{N_{k}}+ (\Delta -1) \widetilde{u}_{N_{k}}=   S_{N_{k}}\Big( | S_{N_{k}}\widetilde{u}_{N_{k}}|^{r-1} S_{N_{k}}\widetilde{u}_{N_{k}}  \Big).
 \end{equation}

We now show that we can pass to the limit $k\longrightarrow+\infty$ in~\eqref{tilde} in order to show that $\widetilde{u}$ is ${\bf  \widetilde{p}}$\,--\,a.s. a   solution to~\eqref{NLS_bis}.
Firstly, from~\eqref{CV} we deduce the convergence of the linear terms of the equation. Indeed,  ${\bf  \widetilde{p}}$\,--\,a.s.\,, when $k\longrightarrow +\infty$
\begin{equation*}
i\partial_{t}\widetilde{u}_{N_{k}}+ (\Delta -1) \widetilde{u}_{N_{k}} \longrightarrow i\partial_{t}\widetilde{u}+ (\Delta -1) \widetilde{u}\quad \text{in}\quad \mathcal{D}'\big([-T,T]\times \S^{3}\big).
\end{equation*}
To handle the nonlinear term, we apply the next lemma.

\begin{lemm}\ph
Let $1\leq r<5$. Up to a sub-sequence, the following convergence holds true
\begin{equation*}
 \widetilde{u}_{N_{k}}\longrightarrow  \widetilde{u} ,\quad \;\;\widetilde{\bf p}-\text{a.s. in}\;\; L^{r}\big([-T,T] \times \S^{3}\big) . \end{equation*}
\end{lemm}

\begin{proof}
In order to simplify the notations in the proof, we drop all the tildes and write $N_{k}\equiv k$ and $L^{p}_{t,x}=L^{p}([-T,T]\times \S^{3})$. If $1\leq r\leq 2$, the result immediately follows from~\eqref{CV}. For $2<r<5$, by the H\"older inequality,  
\begin{equation}\label{416}
\|u_{k}-u\|_{L^{r}_{t,x}}\leq \|u_{k}-u\|^{\theta}_{L^{2}_{t,x}}\|u_{k}-u\|^{1-\theta}_{L^{r+1}_{t,x}},
\end{equation}
with $\theta=\frac{2}{r(r-1)}$. By~\eqref{CV}, a.s. in $\om\in \Omega$
\begin{equation}\label{417}
\|u_{k}-u\|_{L^{2}_{t,x}}  \longrightarrow 0.
\end{equation}
Let $\eps>0$ and $\lambda>0$. By the inclusion
$$ \forall\;\;X,Y\geq 0,\qquad\big\{XY>\lambda\big\}\subset \big\{X>\eps^{\theta}\lambda\big\} \cup \big\{Y>\eps^{-\theta}\big\},$$ 
together with~\eqref{416} and the Markov inequality we have
\begin{multline}\label{borneL}
{\bf p}\big(  \|u_{k}-u\|_{L^{r}_{t,x}}>\lambda      \big)
\\
\leq {\bf p}\big(  \|u_{k}-u\|^{\theta}_{L^{2}_{t,x}}>\eps^{\theta}\lambda\big)+
{\bf p}\big(  \|u_{k}-u\|^{1-\theta}_{L^{r+1}_{t,x}}>\eps^{-\theta}\big) 
\\
\leq
{\bf p}\big(  \|u_{k}-u\|_{L^{2}_{t,x}}>\eps\lambda^{1/\theta}\big)
+\eps^{2/(r-2)}\int_{\Omega}\|u_{k}-u\|^{r+1}_{L^{r+1}_{t,x}}\text{d}{\bf p}.
 \end{multline}
By~\eqref{Bor2.NLS} and the definition of $\nu_{k}$
\begin{equation*} 
\int_{\Omega}\|u_{k}\|^{r+1}_{L^{r+1}_{t,x}}\text{d}{\bf p}=\int\|w\|^{r+1}_{L^{r+1}_{t,x}}\text{d}{\nu_{k}(w)}\leq C_{T}.
\end{equation*}
Similarly, $\dis \int_{\Omega}\|u\|^{r+1}_{L^{r+1}_{t,x}}\text{d}{\bf p}\leq C_{T}$. Therefore $\dis \int_{\Omega}\|u_{k}-u\|^{r+1}_{L^{r+1}_{t,x}}\text{d}{\bf p}$ is bounded uniformly in $k$.  Thus, thanks to~\eqref{417} and~\eqref{borneL}, we get the following convergence in probability
\begin{equation*}
\forall\,\lambda>0,\;\;\;\;\;{\bf p}\big(  \|u_{k}-u\|_{L^{r}_{t,x}}>\lambda      \big)\longrightarrow 0, \quad \text{when}\quad k\longrightarrow +\infty,
\end{equation*}
and after passing to a sub-sequence, we obtain the announced almost sure convergence.
\end{proof}

\subsection{Conclusion of the proof of Theorem~\ref{thmNLS}}\label{subNLS}
Define $\wt{f}=\wt{u}(0)$. Then by~\eqref{LOI}, ${\mathscr{L}_{X^{1/2}}( \,\wt{f}\,) =\rho}$ and by the previous arguments, there exists $\wt{\Omega'}\subset \wt{\Omega}$ such that $\wt{\p}(\wt{\Omega'})=1$. 

Set $\Sigma=\wt{f}(\Omega')$, then $\rho(\Sigma)=\wt{\p}(\wt{\Omega'})=1$. Moreover, for $\om'\in \wt{\Omega'}$, the r.v. $\widetilde{u}$ satisfies  the equation
  \begin{equation}\label{dem} 
\left\{
\begin{aligned}
&i\partial_{t}\widetilde{u}+ (\Delta -1) \widetilde{u}=  |\widetilde{u}|^{r-1}\widetilde{u}, \quad   (t,x)\in \R\times \S^{3},\\
&\widetilde{u}(0,x)=  \widetilde{f}(x) \in X_{rad}^{1/2}(\S^{3}).
\end{aligned}
\right.
\end{equation}

It remains to check that we can construct a global dynamics. Take a sequence $T_{N}\to +\infty$, and perform the previous argument for $T=T_{N}$. For all $N\geq 1$, let  $\Sigma_{N}$ be the corresponding set of initial conditions and  set $\Sigma=\cap_{N\in \N}\Sigma_{N}$. Then $\rho(\Sigma)=1$ and  for all  $\wt f\in \Sigma$, there exists 
 $$\wt u\in \mathcal{C}\big(\R\,; X^{1/2}_{rad}(\mathbb{S}^{3})\big),$$
which solves~\eqref{dem}.

This completes the proof of Theorem~\ref{thmNLS}.

\section{The Benjamin-Ono equation}\label{Sect.5}
\subsection{Preliminaries}
As in~\cite{Tzvetkov3}, consider the following approximation of~\eqref{BO}
\begin{equation}\label{BOn}
\left\{
\begin{aligned}
&\partial_t u + \mathcal{H}\partial^{2}_{x} u   +\Pi_{N}\partial_{x}\big((\Pi_{N}u)^{2}\big) =0,\quad 
(t,x)\in\R\times \mathbb{S}^1,\\
&u(0,x)= f(x).
\end{aligned}
\right.
\end{equation}
This equation is a linear PDE for the high frequencies (modes larger than $2N$) and an ODE for the low frequencies. It is staightforward to check that the quantity $\|u\|_{L^{2}(\T)}$ is preserved by the equation, thus~\eqref{BOn} admits a global flow $\Phi_{N}(t)$. The motivation for introducing~\eqref{BOn}, is that it is given by the   Hamiltonian 
\begin{equation*}
H_{N}(u)=-\frac12\int_{\T}\big(|D_{x}|^{1/2}u\big)^{2}-\frac13\int_{\T}\big(\Pi_{N}u\big)^{3}.
\end{equation*}
As a consequence, we can check that  the measure $\rho_{N}$ as defined in~\eqref{def.rho.BO} is invariant by $\Phi_{N}$. See~\cite{Tzvetkov3} for more details. \medskip

We now state a technical result which we will need  in the sequel.
 \begin{lemm}\ph
Let $\alpha>1/2$, then  there exists $C_{\b}>0$ so that for all  $N\in \Z$
\begin{equation}\label{borne.1}
\sum_{n\in \Z}\frac{1}{\<n\>^{\alpha}\<n-N\>^{\alpha}}\leq \frac{C_{\b}}{\<N\>^{\beta}},
\end{equation}
for all   $\beta<2\a-1$ when $1/2<\a\leq 1$ and $\beta=\a$ when $\a>1$.
\end{lemm}
 
 \begin{proof}
Cut the sum in two parts
 \begin{equation}\label{Cut}
 \sum_{n\in \Z}\frac{1}{\<n\>^{\a}\<n-N\>^{\a}}\leq \sum_{|n|\leq N/2}\frac{1}{\<n\>^{\a}\<n-N\>^{\a}}+\sum_{|n|>N/2}\frac{1}{\<n\>^{\a}\<n-N\>^{\a}}.
 \end{equation}
 Assume that  $\a>1$. Then by~\eqref{Cut}
\begin{equation*}
 \sum_{n\in \Z}\frac{1}{\<n\>^{\a}\<n-N\>^{\a}}\leq \frac{C}{\<N\>^{\a}}\sum_{|n|\leq N}\frac{1}{\<n\>^{\a}}\leq \frac{C}{\<N\>^{\a}}.
\end{equation*}
Assume that $1/2<\a\leq 1$ and fix $\beta<2\a-1$. Then by~\eqref{Cut}
\begin{equation*}
 \sum_{n\in \Z}\frac{1}{\<n\>^{\a}\<n-N\>^{\a}}\leq \frac{C}{\<N\>^{\b}}\sum_{|n|\leq N/2}\frac{1}{\<n\>^{\a}\<n-N\>^{\a-\b}}+ \frac{C}{\<N\>^{\b}}\sum_{|n|> N/2}\frac{1}{\<n\>^{\a-\b}\<n-N\>^{\a}}\leq \frac{C}{\<N\>^{\b}},
\end{equation*}
 which completes the proof.  \end{proof}

  \subsection{Definition of the nonlinear term in~\eqref{BO}}
 To begin with, we have
 
  \begin{lemm}\ph\label{lem.52}
Let $\s>0$. Then there exists $C>0$ so that for all   $p\geq 2$ 
\begin{equation*} 
\big\|\|v\|_{H^{-\s}_{x}}\big\|_{L^{p}_{\mu}}\leq C\sqrt{p}.
\end{equation*}
\end{lemm}
The proof is analogous to~\eqref{CT1} and is omitted here.\medskip

We define the term $\partial_{x}(u^{2})$ in~\eqref{BO} on the support of $\mu$ as the limit of a Cauchy sequence. Recall the notation $u_{N}=\Pi_{N}u$ and set $\Pi^{0}=1-\Pi_{0}$ the orthogonal projection on 0-mean functions. The next result is inspired from~\cite[Lemma 5.1]{Tzvetkov3}
 \begin{lemm}\ph\label{Lem.cauchy}
For all $p\geq 2$, the sequence $\big(\Pi^{0}(u^{2}_{N})\big)_{N\geq1}$ is  Cauchy  in $L^{p}\big(X^{0}(\T),\mathcal{B},d\mu; H^{-\s}(\T)\big)$. Namely, for all $p\geq2$, there exist $\eta>0$ and $C>0$ so that for all $1\leq M<N$,
 \begin{equation*}
\int_{X^{0}(\T)}\|\Pi^{0}(u^{2}_{N})-\Pi^{0}(u^{2}_{M})\|^{p}_{H^{-\s}(\T)}\text{d}\mu(u)\leq \frac{C}{M^{\eta}}.
\end{equation*}
We denote by $\Pi^{0}(u^{2})$ its limit. This enables to define
\begin{equation*}
\partial_{x}\big(u^{2}\big):=\partial_{x}\big(\,\Pi^{0}(u^{2})\,\big).
\end{equation*}
 \end{lemm}
 
 \begin{proof}
 By the result~\cite[Proposition 2.4]{ThTz} on the Wiener chaos, we only have to prove the statement for $p=2$.
 
 Firstly, by definition of the measure $\mu$ 
\begin{equation*}
\int_{X^{0}(\T)}\|\Pi^{0}(u^{2}_{N})-\Pi^{0}(u^{2}_{M})\|^{2}_{H^{-\s}(\T)}\text{d}\mu(u)=\int_{\Omega}\|\Pi^{0}(\phi^{2}_{N})-\Pi^{0}(\phi^{2}_{M})\|^{2}_{H^{-\s}(\T)}\text{d}{\bf p}.
\end{equation*}
Therefore, it is enough to prove that  $\big(\Pi^{0}(\phi^{2}_{N})\big)_{N\geq1}$ is a Cauchy sequence in $L^{2}\big(\Omega; H^{-\s}(\T)\big)$. 
Let $1\leq M<N$, let $k\in \Z$ and denote by $\dis \e_{k}(x)=\e^{ikx}$. Then, by definition of $\phi_{N}$,
\begin{equation*}
\Pi^{0}(\phi^{2}_{N})=\sum_{\substack{0<|n_{1}|, |n_{2}|\leq N\\n_{1}\neq -n_{2}}}\frac{g_{n_{1}}{g}_{n_{2}}}{|n_{1}|^{\frac12}|n_{2}|^{\frac12}}\e^{i(n_{1}+n_{2})x},
\end{equation*}
and thus we get 
\begin{equation*}
\<\,\Pi^{0}(\phi^{2}_{N}-\phi^{2}_{M})\,|\,\e_{k}\,\>=\sum_{B^{(k)}_{M,N}}\frac{g_{n_{1}}{g}_{n_{2}}}{|n_{1}|^{\frac12}|n_{2}|^{\frac12}},
\end{equation*}
where $B^{(k)}_{M,N}$ is the set defined by
\begin{multline*}
B^{(k)}_{M,N}=\big\{(n_{1},n_{2})\in \Z^{2}\;\;\text{s.t.}\;\; 0<|n_{1}|, |n_{2}|  \leq N,\; n_{1}\neq -n_{2}, \\
\;\;\big(|n_{1}|>M         \;\;\text{or}\;\; |n_{2}|>M    \big) \;\;\text{and}\;\; n_{1}+n_{2}=k\big\}.
\end{multline*} 
Therefore we obtain 
\begin{equation*}
\big\|\<\,\Pi^{0}(\phi^{2}_{N}-\phi^{2}_{M})\,|\,\e_{k}\,\>\big\|^{2}_{L^{2}(\Omega)}=\\
{{\int_{\Omega}}}\sum_{\substack {(n_{1},n_{2})\in B^{(k)}_{M,N}\\(m_{1},m_{2})\in B^{(k)}_{M,N}}}
\frac{g_{n_{1}}{g}_{n_{2}}\ov{g}_{m_{1}}\ov{g}_{m_{2}}}
{|n_{1}|^{\frac12}|n_{2}|^{\frac12}|m_{1}|^{\frac12}|m_{2}|^{\frac12}}\text{d}{\bf p}.
\end{equation*}
Since $(g_{n})_{n\in \Z^{*}}$ are independent and centred Gaussians,  we deduce that each term in the r.h.s. vanishes, unless $(n_{1},n_{2})=(m_{1},m_{2})$ or $(n_{1},n_{2})=(m_{2},m_{1})$. Thus by interpolation between~\eqref{borne.1} and the inequality
\begin{equation*}
\sum_{|n|>M}\frac{1}{|n||n-k|}\leq \frac1{M^{\theta}}\sum_{n\neq 0}\frac{1}{|n|^{1-\theta}|n-k|}\leq \frac{C_{\theta}}{M^{\theta}},
\end{equation*}
we obtain that for all $0<\eta<1$ there exists $C>0$ so that for all $1<M<N$
\begin{eqnarray*} 
\big\|\<\,\Pi^{0}(\phi^{2}_{N}-\phi^{2}_{M})\,|\,\e_{k}\,\>\big\|^{2}_{L^{2}(\Omega)}&\leq &C\sum_{(n_{1},n_{2})\in B^{(k)}_{M,N}}\frac{1}{|n_{1}||n_{2}|}\nonumber \\
&\leq &C\sum_{|n|>M}\frac{1}{|n||n-k|}\leq \frac{C}{M^{\eta}\<k\>^{1-\eta}}.
\end{eqnarray*}
As a consequence  we get
\begin{eqnarray*}
\big\|\Pi^{0}(\phi^{2}_{N}-\phi^{2}_{M})\big\|^{2}_{L^{2}(\Omega;H^{-\s}(\T))}&=&\sum_{k\in \Z}\frac1{\<k\>^{2\s}} \big\|\<\,\Pi^{0}(\phi^{2}_{N}-\phi^{2}_{M})\,|\,\e_{k}\,\>\big\|^{2}_{L^{2}(\Omega)}  \\
&\leq &\frac{C}{M^{\eta}}\sum_{k\in \Z}\frac1{\<k\>^{1+2\s-\eta}}\leq \frac{C}{M^{\eta}},
\end{eqnarray*}
whenever  we choose $\eta<2\s$.
 \end{proof}
 
   \subsection{Study of the measure $\boldsymbol {\nu_{N}}$}
Consider the probability measure $\rho_{N}$ defined by~\eqref{def.rho.BO}. Define  the measure $\nu_{N}$ on $\mathcal{C}\big([-T,T]; X^{0}(\T)\big)$ as the image of $\rho_{N}$ by the map 
\begin{equation*}
 \begin{array}{rcc}
X^{0}(\T)&\longrightarrow& \mathcal{C}\big([-T,T]; X^{0}(\T)\big)\\[3pt]
\dis  v&\longmapsto &\dis \Phi_{N}(t)(v),
 \end{array}
 \end{equation*}
where $\Phi_{N}$ is the flow of~\eqref{BOn}. Then, we are able to prove the following bounds
\begin{lemm}\ph  
Let $\s>0$ and  $p\geq 2$. Then there exists $C>0$ such that for all $N\geq 1$
\begin{equation}\label{Cl1}
\big\| \|  u\|_{L^{p}_{T}H^{-\s}_{x}} \big\|_{L^{p}_{\nu_{N}}}\leq  C,
\end{equation}
and
\begin{equation}\label{Cl2}
\big\| \| \partial_{t}u\|_{L^{p}_{T}H^{-\s-2}_{x}}\big\|_{L^{p}_{\nu_{N}}}\leq  C.
\end{equation}
\end{lemm}

\begin{proof}
The bound~\eqref{Cl1} is obtained thanks to~\eqref{Tz},~\eqref{Bor1} and Lemma~\ref{lem.52}.      We now turn to~\eqref{Cl2}. From the equation
\begin{equation*}
\partial_{t}u=-\H\partial^{2}_{x}u-\Pi_{N}\partial_{x}\big((\Pi_{N}u)^{2}\big),
\end{equation*}
similarly to~\eqref{sim}, we deduce 
\begin{equation*}
\|  \partial_{t}u\|_{L^{p}_{\nu_{N}}L^{p}_{T}H^{-\s-2}_{x}}\leq    \|  u\|_{L^{p}_{\nu_{N}}L^{p}_{T}H^{-\s}_{x}}+\|  \Pi^{0}(\Pi_{N}u)^{2}\|_{L^{p}_{\nu_{N}}L^{p}_{T}H^{-\s}_{x}} . 
\end{equation*}
By the invariance of the measure $\rho_{N}$ by $\Phi_{N}$ we get
\begin{eqnarray}
\big\| \,\Pi^{0}\big[ (\Pi_{N}u)^{2}\big]\,\big\|^{p}_{L^{p}_{\nu_{N}}L^{p}_{T}H^{-\s}_{x}} &=&\int_{\mathcal{C}([-T,T];X^{0})}\big\| \,\Pi^{0}\big[ (\Pi_{N}u)^{2}]\,\big\|^{p}_{L^{p}_{T}H^{-\s}_{x}} \text{d}\nu_{N}(u)\nonumber \\
&=&\int_{X^{0}(\S^{1})}\Big\|   \,\Pi^{0}\big[ \big(\,\Pi_{N}\big[\Phi_{N}(t)(v)\big]\,\big)^{2}\big]\,\Big\|^{p}_{L^{p}_{T}H^{-\s}_{x}} \text{d}\rho_{N}(v)\nonumber \\
&=&\int_{X^{0}(\S^{1})}\big\| \,\Pi^{0}\big[  (\Pi_{N}v)^{2}\big]\,\big\|^{p}_{L^{p}_{T}H^{-\s}_{x}} \text{d}\rho_{N}(v)\nonumber \\
&=&2T \int_{X^{0}(\S^{1})}\big\| \,\Pi^{0}\big[ (\Pi_{N}v)^{2}\big]\,\big\|^{p}_{H^{-\s}_{x}} \Psi_{N}(v)\text{d}\mu(v),\label{u2}
\end{eqnarray}
and by  Cauchy-Schwarz and Lemma~\ref{Lem.cauchy}
\begin{equation*}
\big\|\,\Pi^{0}\big[ (\Pi_{N}u)^{2}\big]\,\big\|^{p}_{L^{p}_{\nu_{N}}L^{p}_{T}H^{-\s}_{x}} \leq C_{T}\big\|\,\Pi^{0}\big[ (\Pi_{N}v)^{2}\big]\,\big\|^{p}_{L^{2p}_{\mu}H^{-\s}_{x}} \| \Psi_{N}(v)\|_{L^{2}_{\mu}}\leq C,
\end{equation*}
 which concludes the proof.
 \end{proof}

\begin{prop}\ph\label{Prop.tight2}
Let $T>0$ and $\s>0$. Then the family of measures 
$$ \nu_{N}=\mathscr{L}_{\mathcal{C}_{T}H^{-\s}}\big(u_{N}(t);t\in [-T,T]\big)_{N\geq 1}$$
 is tight in $\mathcal{C}\big([-T,T]; H^{-\s}(\T)\big).$
\end{prop}

\begin{proof} The proof is similar to the proof of Proposition~\ref{Prop.tight}. Here we use  the estimates~\eqref{Cl1} and~\eqref{Cl2}.
\end{proof}

  \subsection{Proof of Theorem~\ref{thmBO}}
  
 By  Proposition~\ref{Prop.tight2} we can use the    Prokhorov theorem: For each $T>0$ there exists a sub-sequence $\nu_{N_{k}}$ and a measure $\nu$ on the space $\mathcal{C}\big([-T,T]; X^{0}(\T)\big)$ so that 
 $ \nu_{N_{k}}\longrightarrow \nu$ weakly on $\mathcal{C}\big([-T,T]; H^{-\s}(\T)\big)$, for all $\s>0$. By the Skorokhod theorem, there  exists a probability space $(\widetilde{\Omega},\widetilde{\mathcal{F}},\widetilde{\bf p})$, a sequence of random variables $(\widetilde{u}_{N_{k}})$ and a random variable $\widetilde{u}$ with values in $\mathcal{C}\big([-T,T]; X^{0}(\T)\big)$ so that 
\begin{equation*}
\mathscr{L}\big(\widetilde{u}_{N_{k}};t\in [-T,T]\big)=\mathscr{L}\big(u_{N_{k}};t\in [-T,T]\big)=\nu_{N_{k}}, \quad \mathscr{L}\big(\widetilde{u};t\in [-T,T]\big)=\nu,
\end{equation*}
 and for all $\s>0$
\begin{equation}\label{CV0}
\widetilde{u}_{N_{k}}\longrightarrow \widetilde{u},\quad \;\;\widetilde{\bf p}-\text{a.s. in}\;\; \mathcal{C}\big([-T,T]; H^{-\s}(\T)\big).
\end{equation}
We have  that $\mathscr{L}_{X^{0}(\S^{1})}( {u}_{N_{k}}(t))=\mathscr{L}_{X^{0}(\S^{1})}(\widetilde{u}_{N_{k}}(t))=\rho_{N_{k}}$, for all $t\in [-T,T]$ and $k\geq 1$. Therefore, for all $t\in [-T,T]$, $\mathscr{L}_{X^{0}(\S^{1})}( {u}(t)) =\rho$. Next, $\widetilde{u}_{N_{k}}$ satisfies the following equation ${\bf  \widetilde{p}}$\,--\,a.s.
\begin{equation*} 
\partial_{t}\widetilde{u}_{N_{k}}+ \H\,\partial^{2}_{x} \widetilde{u}_{N_{k}}+\Pi_{N_{k}}\partial_{x}\big((\Pi_{N_{k}}\wt{u}_{N_{k}})^{2}\big) =0.
 \end{equation*}
We now show that we can pass to the limit $k\longrightarrow+\infty$ in the previous equation.
Firstly, from~\eqref{CV0} we deduce the convergence of the linear terms of the equation. Indeed,  ${\bf  \widetilde{p}}$\,--\,a.s.\,, when $k\longrightarrow +\infty$
\begin{equation*}
\partial_{t}\widetilde{u}_{N_{k}}+ \H\,\partial^{2}_{x} \widetilde{u}_{N_{k}} \longrightarrow \partial_{t}\widetilde{u}+ \H\,\partial^{2}_{x} \widetilde{u}\quad \text{in}\quad \mathcal{D}'\big([-T,T]\times \T\big).
\end{equation*}
  The only difficulty is to pass to the limit in the non linear term. Here we can proceed as in~\cite{DPD}.
\begin{lemm}\ph\label{lem.56}
Let $\s>0$. Up to a sub-sequence, the following convergence holds true
\begin{equation*} 
\Pi^{0}\big[(\Pi_{N_{k}}\wt{u}_{N_{k}})^{2}\big] \longrightarrow  \Pi^{0}\big[\wt{u} ^{2}\big],\quad \;\;\widetilde{\bf p}-\text{a.s. in}\;\; L^{2}\big([-T,T]; H^{-\s}(\T)\big) . 
\end{equation*} 
\end{lemm}

 \begin{proof}
 In order to simplify the notations, in this proof we drop the tildes and write $N_{k}=k$. Let $M\geq1$ and write 
  \begin{equation*}
 \Pi^{0}\Big[(\Pi_{k}u_{k})^{2} - u^{2}\Big]=  \Pi^{0}\Big[\big((\Pi_{k}u_{k})^{2} - u_{k}^{2}\big)+\big(  u_{k}^{2}-(\Pi_{M}u_{k})^{2}\big)+\big((\Pi_{M}u_{k})^{2}-(\Pi_{M}u)^{2}\big)+\big( (\Pi_{M}u)^{2}-u^{2}\big)\Big].
 \end{equation*}
  To begin with, by continuity of the square in finite dimension, when $k\longrightarrow +\infty$
\begin{equation*}
 \Pi^{0}\big[(\Pi_{M}u_{k})^{2}\big] \longrightarrow  \Pi^{0}\big[(\Pi_{M}u)^{2}\big] ,\quad \;\;\widetilde{\bf p}-\text{a.s. in}\;\; L^{2}\big([-T,T]; H^{-\s}(\T)\big).
\end{equation*}
 We now deal with the other terms. It is sufficient to show  the convergence in the space $X:=L^{2}\big(\Omega \times[-T,T]; H^{-\s}(\T)\big)$, since  the almost sure convergence  follows after exaction of a sub-sequence. \\
With the same arguments as in~\eqref{u2} we obtain
  \begin{eqnarray*}
\big\| \,\Pi^{0}\big[ (\Pi_{M}u_{k})^{2}-u^{2}_{k}\big]\,\big\|^{2}_{X}&=&\int_{\mathcal{C}([-T,T];X^{0})}\big\| \,\Pi^{0}\big[ (\Pi_{M}v)^{2}-v^{2}\big]\,\big\|^{2}_{L^{2}_{T}H^{-\s}_{x}} \text{d}\nu_{k}(v)\\
&=&\int_{X^{0}(\S^{1})}\Big\| \,\Pi^{0}\Big[ \big[\Pi_{M}\Phi_{k}(t)(f)\big]^{2}-\big[\Phi_{k}(t)(f)\big]^{2}\Big]\,\Big\|^{2}_{L^{2}_{T}H^{-\s}_{x}} \text{d}\rho_{k}(f)\\
&=&\int_{X^{0}(\S^{1})}\big\| \,\Pi^{0}\big[ \big(\Pi_{M}f\big)^{2}-f^{2}\big]\,\big\|^{2}_{L^{2}_{T}H^{-\s}_{x}} \text{d}\rho_{k}(f)\\
&=&2T\int_{X^{0}(\S^{1})}\big\| \,\Pi^{0}\big[ \big(\Pi_{M}f\big)^{2}-f^{2}\big]\,\big\|^{2}_{H^{-\s}_{x}} \Psi_{k}(f)\text{d}\mu(f),
\end{eqnarray*}
and by Cauchy-Schwarz and~\eqref{Tz},
\begin{equation*}
\big\| \,\Pi^{0}\big[ (\Pi_{M}u_{k})^{2}-u^{2}_{k}\big]\,\big\|_{X}\leq C\big\| \,\Pi^{0}\big[ \big(\Pi_{M}f\big)^{2}-f^{2}\big]\,\big\|_{L^{4}_{\mu}H^{-\s}_{x}}.
\end{equation*}
  This latter term tends to 0 uniformly in $k\geq 1$ when $M\longrightarrow +\infty$, according to Lemma~\ref{Lem.cauchy}. The term $\big\| \,\Pi^{0}\big[  (\Pi_{M}u)^{2}-u^{2}\big]\,\big\|_{X}$ is treated similarly. \\
Finally, with the same argument we show
 \begin{equation*}
 \big\| \,\Pi^{0}\big[(\Pi_{k}u_{k})^{2} - u_{k}^{2}\big]\,\big\|_{X} \leq C  \big\| \,\Pi^{0}\big[  \big(\Pi_{k}f\big)^{2}-f^{2}\big]\,\big\|_{L^{4}_{\mu}H^{-\s}_{x}},
 \end{equation*}
which tends to 0 when $k\longrightarrow +\infty$. This completes the proof.
  \end{proof}

The conclusion of the proof of Theorem~\ref{thmBO} is similar to the argument in Subsection~\ref{subNLS}.

\section{The derivative nonlinear Schr\"odinger equation}\label{Sect.6}

\subsection{Hamiltonian formalism of DNLS}
To begin with, we recall some facts which are explained in the appendix of~\cite{ThTz}. 
We define the operator $\partial^{-1}$ by 
\begin{equation*}
\partial^{-1}\,:\,f(x)=\sum_{n\in \Z}\a_{n}\e^{inx} \longmapsto \sum_{n\in \Z\backslash\{0\}}\frac{\a_{n}}{in}\e^{inx},
\end{equation*} 
and the skew symmetric operator ($K(u,v)^{*}=-K(u,v)$)
 \begin{equation}\label{def.K}
K(u,v)=\left(\begin{array}{cc} 
-u\partial^{-1}u\cdot & -i+u\partial^{-1}v\cdot \\[2pt]
i+v\partial^{-1}u\cdot & -v\partial^{-1}v\cdot 
\end{array} \right). 
\end{equation}
Define   $H$ by
\begin{equation*} 
H(u(t))=\int_{\T}|\partial_{x}u|^{2}\text{d}x+\frac34i\int_{\T}\ov{u}^{2}\,\partial_{x}(u^{2})\text{d}x+\frac12\int_{\T}|u|^{6}\text{d}x,
\end{equation*}
 and introduce the Hamiltonian system 
\begin{equation}
\left(\begin{array}{c}\label{syst}
\partial_{t}u \\[2pt]
\partial_{t}v
\end{array} \right) =K(u,v)
\left(\begin{array}{c}
\frac{\delta H}{\delta u}(u,v) \\[4pt]
\frac{\delta H}{\delta v}(u,v)
\end{array} \right).
\end{equation}
Denote by 
\begin{equation}\label{deFu}
T_{u}(t)=2\,\text{Im}\,\int_{\T}u\partial_{x}\ov{u}+\frac32\int_{\T}|u|^{4},
\end{equation} 
then the system~\eqref{syst} is a Hamiltonian formulation of the equation 
\begin{equation}\label{eq.gauge}
i\partial_t u+\partial_{x}^{2} u   = i\partial_{x}\big(|u|^{2}u\big)+T_{u}(t)u,
\end{equation}
in the coordinates $(u,v)=(u,\ov{u})$ (see~\cite[Proposition A.2]{ThTz}). 
Now, if  we set 
\begin{equation}\label{eq.chgt}
v(t,x)=\e^{i\int_{0}^{t}T_{u}(s)\text{d}s}u(t,x),
\end{equation} 
then $v$ is the solution of the equation
\begin{equation*} 
\left\{
\begin{aligned}
&i\partial_t v+\partial_{x}^{2} v   = i\partial_{x}\big(|v|^{2}v\big),\;\;
(t,x)\in\R\times \mathbb{S}^1,\\
&v(0,x)= u_{0}(x).
\end{aligned}
\right.
\end{equation*}
Moreover, if $u$ and $v$ are linked by~\eqref{eq.chgt}, we have $T_{u}=T_{v}$.\medskip

Thanks to these observations, we can focus on the equation~\eqref{eq.gauge}. We introduce  a natural truncation for which we can construct an invariant Gibbs measure. Namely, let $K$ be given by~\eqref{def.K}, and consider  the following system 
\begin{equation}
\left(\begin{array}{c}\label{systN}
\partial_{t}u \\[2pt]
\partial_{t}v
\end{array} \right) =\Pi_{N}K(u_{N},v_{N})\Pi_{N}
\left(\begin{array}{c}
\frac{\delta H}{\delta u}(u_{N},v_{N}) \\[4pt]
\frac{\delta H}{\delta v}(u_{N},v_{N})
\end{array} \right).
\end{equation}
This is an Hamiltonian system with Hamiltonian $H(\Pi_{N}u,\Pi_{N}v)$. Now we assume that $v=\ov{u}$ and we compute the equation satisfied by $u_{N}$: this will be a finite dimensional approximation   of~\eqref{eq.gauge}. Denote by $\Pi^{\perp}_{N}=1-\Pi_{N}$, then in the coordinates $v_{N}=\ov{u_{N}}$, the system~\eqref{systN} reads
\begin{equation}\label{App.DNLS} 
i\partial_{t}u+ \partial_{x}^{2}u_{N}=  i\Pi_{N}\Big( \partial_{x}(|u_{N}|^{2}u_{N}  )\Big)+u_{N}T_{u_{N}}  +  R_{N}(u_{N}),\quad 
(t,x)\in\R\times \mathbb{S}^1,
\end{equation}
where 
\begin{eqnarray}
R_{N}(u_{N})&=&\frac32  \Pi_{N} \Big(u_{N}\partial^{-1}\Big[ u_{N}\Pi_{N}^{\perp}\big(u_{N}\partial_{x}(\ov{u_{N}}^{2})\big)+\ov{u_{N}}\Pi_{N}^{\perp}\big(\ov{u_{N}}\partial_{x}({u_{N}}^{2})\big)\Big]\Big)\nonumber\\
&&+\frac32i  \Pi_{N} \Big( u_{N}\partial^{-1}\Big[ {u_{N}}\Pi^{\perp}_{N}\big(|u_{N}|^{4}\ov{u_{N}}\big)-\ov{u_{N}}\Pi^{\perp}_{N}\big(|u_{N}|^{4}{u_{N}}\big) \Big] \Big)\nonumber\\
&:=&R^{1}_{N}(u_{N})+R^{2}_{N}(u_{N}).\label{def.R}
\end{eqnarray}
For all $N\geq 1$,  this equation is globally well-posed in $L^{2}(\T)$ and denote by $\Phi_{N}$ the   flowmap. Moreover, the measure $\rho_{N}$ defined in~\eqref{rho.DNLS}  is invariant by $\Phi_{N}$ (see~\cite[Proposition A.4]{ThTz}). \medskip

Recall that $\mu=\p\circ \phi^{-1}$ with $\phi$ as in~\eqref{phi.DNLS}. We need  to give a sense to the expression $T_{u}$ in~\eqref{deFu} on the support of $\mu$.
\begin{lemm}\ph \label{lem.61}
For all $p\geq 2$, the sequence $\big(T_{u_{N}}\big)_{N\geq1}$ is a Cauchy sequence in $L^{p}\big(X^{1/2}(\S^{1}),\mathcal{B},d\mu; \R\big)$. Namely, for all $p\geq	 2$, there exists  $C>0$ so that for all $1\leq M<N$,
 \begin{equation*}
\int_{X^{1/2}(\S^{1})}|T_{u_{N}}-T_{u_{M}}|^{p}\text{d}\mu(u)\leq \frac{C}{M}.
\end{equation*}
We denote by $T_{u}$ the limit of this sequence which is formally given by~\eqref{deFu}.
\end{lemm}

\begin{proof}
Denote by $\dis J(u)=\text{Im}\,\int_{\T}u\partial_{x}\ov{u}$. Let $1\leq M<N$. Then for $\dis \phi_{N}(\om,x)=\sum_{|n|\leq N}\frac{g_{n}(\om)}{\<n\>}\e^{inx}$ we compute
\begin{equation*}
J(\phi_{N})-J(\phi_{M})=-\sum_{M<|n|\leq N}\frac{n|g_{n}|^{2}}{\<n\>^{2}} =-\sum_{M<|n|\leq N}\frac{n(|g_{n}|^{2}-1)}{\<n\>^{2}} ,
\end{equation*}
where we used  that $\dis\sum_{M<|n|\leq N}\frac{n}{\<n\>^{2}} =0$. Define the r.v.  $G_{n}(\om)=|g_{n}(\om)|^{2}-1$, hence 
 \begin{equation}\label{410}
|J(\phi_{N})-J(\phi_{M})|^{2}=\sum_{M<|n_{1}|,|n_{2}|\leq N}\frac{n_{1}n_{2}G_{n_{1}}G_{n_{2}}}{\<n_{1}\>^{2}\<n_{2}\>^{2}} .
\end{equation}
By independence of the $g_{n}$, $\E[G_{n}G_{m}]=C\delta_{n,m}$. Thus by integration of~\eqref{410}
\begin{equation*}
\int_{\Omega}|J(\phi_{N})-J(\phi_{M})|^{2}\text{d}{\bf p}=\sum_{M<|n|\leq N}\frac{n^{2}}{\<n\>^{4}}\leq \frac{C}{M}.
\end{equation*}
By definition of $\mu$ we have proved the result for $p=2$. The general case $p\geq 2$ follows from the Wiener chaos estimates (see {\it e.g.}~\cite[Proposition 2.4]{ThTz}).
\end{proof}
\subsection{Study of the measure $\boldsymbol {\nu_{N}}$}

Now define the  measure $\nu_{N}=\rho_{N}\circ \Phi^{-1}_{N}$ on $\mathcal{C}\big([-T,T];X^{1/2}(\S^{1})\big)$ and we have  
\begin{lemm}\ph Let $\s<\frac12$ and  $p\geq 2$. Then for all $N\geq1$
\begin{equation}\label{Bor1DNLS}
\big\| \| u\|_{L^{p}_{T}H^{\s}_{x}}\big\|_{L^{p}_{\nu_{N}}}\leq  C.
\end{equation}
\begin{equation}\label{Bor3DNLS}
\big\| \|\partial_{t} u\|_{  L^{p}_{T}H^{\s-2}_{x}  }\big\|_{L^{p}_{\nu_{N}}}\leq  C.
\end{equation}
\end{lemm}

\begin{proof}
The estimate~\eqref{Bor1DNLS} is obtained with Proposition~\ref{Prop.Fonda} and the definition~\eqref{phi.DNLS} of $\phi$. Similarly,  we also have that 
for all $2\leq q\leq p$ 
\begin{equation}\label{Bor2DNLS}
\| u\|_{L^{p}_{\nu_{N}}L^{p}_{T}L^{q}_{x}} \leq  C.
\end{equation}
We turn to~\eqref{Bor3DNLS}. From the equation~\eqref{App.DNLS} we get (similarly to~\eqref{sim})
\begin{multline*}
\| \partial_{t} u\|_{L^{p}_{\nu_{N}}L^{p}_{T}H^{\s-2}_{x}}\leq\\
\begin{aligned}
&\leq \| \partial^{2}_{x} u\|_{L^{p}_{\nu_{N}}L^{p}_{T}H^{\s-2}_{x}}+\| \partial_{x}(|u_{N}|^{2}u_{N} )\|_{L^{p}_{\nu_{N}}L^{p}_{T}H^{\s-2}_{x}}+\| u_{N}T_{u_{N}}\|_{L^{p}_{\nu_{N}}L^{p}_{T}H^{\s-2}_{x}}+\| R_{N}(u_{N})\|_{L^{p}_{\nu_{N}}L^{p}_{T}H^{\s-2}_{x}}\\
&\leq \|  u\|_{L^{p}_{\nu_{N}}L^{p}_{T}H^{\s}_{x}}+\| u_{N}\|^{3}_{L^{p}_{\nu_{N}}L^{p}_{T}L^{6}_{x}}+\| u_{N}T_{u_{N}}\|_{L^{p}_{\nu_{N}}L^{p}_{T}L^{2}_{x}}+\| R_{N}(u_{N})\|_{L^{p}_{\nu_{N}}L^{p}_{T}H^{\s-2}_{x}}.
\end{aligned}
\end{multline*}
We estimate each term of the r.h.s. By~\eqref{Bor1DNLS} and~\eqref{Bor2DNLS} we only have to consider the two last ones. By Cauchy-Schwarz (recall that $T_{u}$ does not depend on $x$)
\begin{equation}\label{Tu}
\| u_{N}T_{u_{N}}\|_{L^{p}_{\nu_{N}}L^{p}_{T}L^{2}_{x}} \leq \| u_{N}\|_{L^{2p}_{\nu_{N}}L^{2p}_{T}L^{2}_{x}}\| T_{u_{N}}\|_{L^{2p}_{\nu_{N}}L^{2p}_{T}}.
\end{equation}
Then using the invariance of $\rho_{N}$ (see the proof of Proposition~\ref{Prop.Fonda}) and Lemma~\ref{lem.61} we have
\begin{eqnarray*}
\| T_{u_{N}}\|^{2p}_{L^{2p}_{\nu_{N}}L^{2p}_{T}}&=&2T\int_{X^{1/2}(\S^{1})}|T_{v_{N}}|^{2p}\Psi_{N}(v)\text{d}\mu(v)\\
&\leq& C\| T_{v_{N}}\|^{2p}_{L^{4p}_{\mu}}\| \Psi_{N}(v)\|_{L^{2}_{\mu}}\leq C,
\end{eqnarray*}
which by~\eqref{Tu} implies
\begin{equation*}
\| u_{N}T_{u_{N}}\|_{L^{p}_{\nu_{N}}L^{p}_{T}L^{2}_{x}} \leq C.
\end{equation*}
 The conclusion of the proof is given by the next result.
\end{proof}

\begin{lemm}\ph\label{lem.RN}
Let $\s>1/2$ and $p\geq2$. Then  
\begin{equation*}
\big\|\|R_{N}(u_{N})\|_{L^{p}_{T}H^{-\s}_{x}}\big\|_{L^{p}_{\nu_{N}}}\longrightarrow 0 \quad \text{when}\quad N\longrightarrow +\infty. 
\end{equation*}
\end{lemm}

\begin{proof}
To begin with, using the same arguments as in the proof of Proposition~\ref{Prop.Fonda} with $F(u)=\|R_{N}(\Pi_{N}u)\|^{p}_{L^{p}_{T}H^{-\s}_{x}}$ we have, 
\begin{equation*}
\|R_{N}(u_{N})\|_{L^{p}_{\nu_{N}}L^{p}_{T}H^{-\s}_{x}}\leq C \|R_{N}(v_{N})\|_{L^{2p}_{\mu}H^{-\s}_{x}},
\end{equation*}
where we used that $\|\Psi_{N}\|_{L^{2}_{\mu}}\leq C$.
We estimate each contribution in the r.h.s. of~\eqref{def.R}. \\
$\bullet$ Denote by 
$\dis Q_{N}(v_{N}) =v_{N}\Pi_{N}^{\perp}\big(v_{N}\partial_{x}(\ov{v_{N}}^{2})\big)$.
Then by Sobolev and Cauchy-Schwarz 
\begin{eqnarray}
 \|R^{1}_{N}(v_{N})\|_{L^{r}_{\mu}H^{-\s}_{x}} &\leq&C \|R^{1}_{N}(v_{N})\|_{L^{r}_{\mu}L^{1}_{x}}\nonumber\\
 &\leq &C\| v_{N}\partial^{-1}Q_{N}(v_{N})\|_{L^{r}_{\mu}L^{1}_{x}}\nonumber\\
 & \leq &\| v_{N}\|_{L^{2r}_{\mu}L^{2}_{x}}\| Q_{N}(u_{N})\|_{L^{2r}_{\mu}H^{-1}_{x}}\nonumber\\
  & \leq &C\| Q_{N}(u_{N})\|_{L^{2r}_{\mu}H^{-1}_{x}}.\label{sa1}
\end{eqnarray}
 Next, by the definition of $\mu$ and the Wiener chaos estimates
 \begin{eqnarray}
 \|R^{1}_{N}(v_{N})\|_{L^{r}_{\mu}H^{-\s}_{x}} &\leq&C\| Q_{N}(\phi_{N})\|_{L^{2r}_{\p}H^{-1}_{x}}\nonumber\\
 &\leq&C\| Q_{N}(\phi_{N})\|_{L^{2}_{\p}H^{-1}_{x}}.\label{sa2}
\end{eqnarray}
We now compute the term $\dis \| Q_{N}(\phi_{N})\|_{L^{2}_{\p}H^{-1}_{x}}$. We have
\begin{equation*}
\phi_{N}\partial_{x}\big(\ov{ \phi^{2}_{N}}\big)=-i\sum_{|n_{1}|,|n_{2}|,|n_{3}|\leq N}\frac{(n_{1}+n_{2})\ov{g_{n_{1}}}\,\ov{g_{n_{2}}}\,{g_{n_{3}}}}{\<n_{1}\>\<n_{2}\>\<n_{3}\>}\e^{i(n_{3}-n_{2}-n_{1})x},
\end{equation*}
so that 
\begin{equation*}
\partial^{-1}Q_{N}(\phi_{N})=-\sum_{n\in A_{N}} \frac{(n_{1}+n_{2})\ov{g_{n_{1}}}\,\ov{g_{n_{2}}}\,g_{n_{3}}\,g_{n_{4}}}{\<n_{1}\>\<n_{2}\>\<n_{3}\>\<n_{4}\>(n_{4}+n_{3}-n_{2}-n_{1})}\e^{i(n_{4}+n_{3}-n_{2}-n_{1})x},
\end{equation*}
where the set $A_{N}$ is given by
\begin{multline*}
A_{N}:=\Big\{n=(n_{1},n_{2},n_{3},n_{4})\in \Z^{4}\;\;\text{s.t.}\;\; |n_{1}|, |n_{2}|,  |n_{3}|, |n_{4}| \leq N,\\
\;\;|n_{1}+n_{2}-n_{3}|>N\;\;\text{and}\;\; n_{4}+n_{3}-n_{2}-n_{1}\neq 0\Big\}.
\end{multline*} 
As a consequence we obtain the following expression
\begin{equation}\label{423}
\|Q_{N}(\phi_{N})\|^{2}_{H_{x}^{-1}}=\sum_{n,m\in B_{N}} \frac{(n_{1}+n_{2})(m_{1}+m_{2})\ov{g_{n_{1}}}\,\ov{g_{n_{2}}}\,g_{n_{3}}\,g_{n_{4}}\,{g_{m_{1}}}\,{g_{m_{2}}}\,\ov{g_{m_{3}}}\,\ov{g_{m_{4}}}}{\<n_{1}\>\<n_{2}\>\<n_{3}\>\<n_{4}\>\<m_{1}\>\<m_{2}\>\<m_{3}\>\<m_{4}\>(n_{4}+n_{3}-n_{2}-n_{1})^{2}} ,
\end{equation}
with 
\begin{equation*}
B_{N}:=\Big\{n,m\in A_{N}\;\;\text{s.t.}\;\; m_{4}+m_{3}-m_{2}-m_{1}=n_{4}+n_{3}-n_{2}-n_{1}\Big\}.
\end{equation*} 
We take the expectation of~\eqref{423}. By independence of the $g_{n}$ and since they are centered, each contribution in the r.h.s. is zero, unless $\big\{  n_{1},n_{2},m_{3},m_{4}\big\}=\big\{  m_{1},m_{2},n_{3},n_{4}\big\}$. But coming back to the definition of $A_{N}$, the condition $|n_{1}+n_{2}-n_{3}|>N$ implies that $n_{3}\not \in \big\{n_{1},n_{2}\big\}$. Similarly, $m_{3}\not \in \big\{m_{1},m_{2}\big\}$. Therefore, up to permutation we have $n=m$ and by~\eqref{borne.1} with $\a=2$
\begin{eqnarray*} 
\int_{\Omega}\|Q_{N}(\phi_{N})\|^{2}_{H_{x}^{-1}}\text{d}{\bf p}&\leq&C\sum_{n\in A_{N}} \frac{(n_{1}+n_{2})^{2}}{\<n_{1}\>^{2}\<n_{2}\>^{2}\<n_{3}\>^{2}\<n_{4}\>^{2}(n_{4}+n_{3}-n_{2}-n_{1})^{2}}\\
 &\leq&CN^{2}\sum_{n\in A_{N}} \frac{1}{\<n_{1}\>^{2}\<n_{2}\>^{2}\<n_{3}\>^{2}\<n_{3}-n_{2}-n_{1}\>^{2}}\\
  &\leq&C\sum_{n\in A_{N}} \frac{1}{\<n_{1}\>^{2}\<n_{2}\>^{2}\<n_{3}\>^{2}}.
\end{eqnarray*} 
Next, use that on $A_{N}$, $\<n_{1}\>\<n_{2}\>\<n_{3}\>\geq CN$ to get that 
\begin{equation}\label{sa3}
\int_{\Omega}\|Q_{N}(\phi_{N})\|^{2}_{H_{x}^{-1}}\text{d}{\bf p}\leq\frac{C}{N^{1/2}}\sum_{n\in \Z^{3}} \frac{1}{\<n_{1}\>^{3/2}\<n_{2}\>^{3/2}\<n_{3}\>^{3/2}}\leq \frac{C}{N^{1/2}}.
\end{equation}
Finally, from~\eqref{sa1},~\eqref{sa2} and~\eqref{sa3} we conclude that 
\begin{equation*}
\|R^{1}_{N}(u_{N})\|_{L^{p}_{\nu_{N}}L^{p}_{T}H^{-\s}_{x}}\longrightarrow 0.
\end{equation*}
$\bullet$ We now consider the contribution of $R^{2}_{N}$. With the same arguments as previously,
\begin{eqnarray*}
 \|R^{2}_{N}(v_{N})\|_{L^{r}_{\mu}H^{-\s}_{x}} &\leq&C \|R^{2}_{N}(v_{N})\|_{L^{r}_{\mu}L^{1}_{x}}\nonumber\\
 &\leq &C\Big\|v_{N}\partial^{-1}\Big[ {v_{N}}\Pi^{\perp}_{N}\big(|v_{N}|^{4}\ov{v_{N}}\big)\Big]\Big\|_{L^{r}_{\mu}L^{1}_{x}}\nonumber\\
 & \leq &C\| v_{N}\|_{L^{2r}_{\mu}L^{2}_{x}}\| {v_{N}}\Pi^{\perp}_{N}\big(|v_{N}|^{4}\ov{v_{N}}\big)\|_{L^{2r}_{\mu}H^{-1}_{x}}\nonumber\\
  & \leq &C\| {v_{N}}\Pi^{\perp}_{N}\big(|v_{N}|^{4}\ov{v_{N}}\big)\|_{L^{2r}_{\mu}L^{1}_{x}}\nonumber\\
  & \leq &C\| \Pi^{\perp}_{N}\big(|v_{N}|^{4}\ov{v_{N}}\big)\|_{L^{4r}_{\mu}L^{2}_{x}}. 
\end{eqnarray*}
Denote by $V_{N}=|v_{N}|^{4}\ov{v_{N}}$. Then by~\cite[Lemma 2.2]{ThTz}, $(V_{N})_{N\geq 1}$ is a Cauchy sequence in $L^{4r}_{\mu}L^{2}_{x}$, and denote by $V$ its limit.
Write 
\begin{eqnarray*}
\| \Pi^{\perp}_{N}V_{N}\|_{L^{4r}_{\mu}L^{2}_{x}}&\leq &\| \Pi^{\perp}_{N}(V_{N}-V)\|_{L^{4r}_{\mu}L^{2}_{x}}+\| \Pi^{\perp}_{N}V\|_{L^{4r}_{\mu}L^{2}_{x}}\\
&\leq &\| V_{N}-V\|_{L^{4r}_{\mu}L^{2}_{x}}+\| \Pi^{\perp}_{N}V\|_{L^{4r}_{\mu}L^{2}_{x}},
\end{eqnarray*}
which tends to 0 when $N\longrightarrow +\infty$.
\end{proof}

\begin{prop}\ph\label{Prop.tight3}
Let $T>0$ and $\s<1/2$. Then the family of measures 
$$ \nu_{N}=\mathscr{L}_{\mathcal{C}_{T}H^{\s}}\big(u_{N}(t);t\in [-T,T]\big)_{N\geq 1}$$
 is tight in $\mathcal{C}\big([-T,T]; H^{\s}(\T)\big).$
\end{prop}

\begin{proof} The proof is similar to the proof of Proposition~\ref{Prop.tight}. Here we use  the estimates~\eqref{Bor1DNLS} and~\eqref{Bor3DNLS}.
\end{proof}

  \subsection{Proof of Theorem~\ref{thmDNLS}}
  
We can proceed as in the proofs of Theorems~\ref{thmNLS} and~\ref{thmBO}. By  Proposition~\ref{Prop.tight3} and  the    Prokhorov theorem we can extract  a sub-sequence $\nu_{N_{k}}$ and a measure $\nu$ on the space $\mathcal{C}\big([-T,T]; X^{1/2}(\T)\big)$ so that $\nu_{N_{k}}\longrightarrow \nu$ weakly on $\mathcal{C}\big([-T,T]; H^{\s}(\T)\big)$ for all $\s<1/2$. Thanks to  the Skorokhod theorem, there  exists a probability space $(\widetilde{\Omega},\widetilde{\mathcal{F}},\widetilde{\bf p})$, a sequence of random variables~$(\widetilde{u}_{N_{k}})$ and a random variable $\widetilde{u}$ with values in $\mathcal{C}\big([-T,T]; X^{1/2}(\T)\big)$ so that 
\begin{equation*}
\mathscr{L}\big(\widetilde{u}_{N_{k}};t\in [-T,T]\big)=\mathscr{L}\big(u_{N_{k}};t\in [-T,T]\big)=\nu_{N_{k}}, \quad \mathscr{L}\big(\widetilde{u};t\in [-T,T]\big)=\nu,
\end{equation*}
 and for all $\s<1/2$
\begin{equation*} 
\widetilde{u}_{N_{k}}\longrightarrow \widetilde{u},\quad \;\;\widetilde{\bf p}-\text{a.s. in}\;\; \mathcal{C}\big([-T,T]; H^{\s}(\T)\big).
\end{equation*}
Moreover, $\widetilde{u}_{N_{k}}$ satisfies $\widetilde{\bf p}$-a.s. the equation~\eqref{App.DNLS}. Passing to the limit in the linear terms  makes no difficulty, we only have to take care on the nonlinear terms.    
  Denote by 
  \begin{equation*}
  \mathcal{G}_{N}(u)= i\Pi_{N}\Big( \partial_{x}(|u_{N}|^{2}u_{N}  )\Big)+u_{N}T_{u_{N}} +R_{N}(u_{N}).
  \end{equation*}
  The next result completes the proof of Theorem~\ref{thmDNLS} (the conclusion of the proof  is similar to the argument in Subsection~\ref{subNLS}).
\begin{lemm}\ph
Up to a sub-sequence, the following convergence holds true. For any $\s>0$
\begin{equation*} 
\mathcal{G}_{N_{k}}(\wt{u}_{N_{k}})\longrightarrow i\partial_{x}(|\,\wt{u}\,|^{2}\,\wt{u})+\wt{u}T_{\wt{u}}  ,\quad \;\;\widetilde{\bf p}-\text{a.s. in}\;\; L^{2}\big([-T,T]; H^{-\s}(\T)\big) . 
\end{equation*} 
\end{lemm}

 \begin{proof}
We drop the tildes and write $N_{k}\equiv N$. Since $\mathscr{L}(u_{N})=\nu_{N}$, we can apply  Lemma~\ref{lem.RN}
\begin{equation*}
\|R_{N}(u_{N})\|_{L^{2}_{\p}L^{2}_{T}H^{-\s}_{x}}=\|R_{N}(u_{N})\|_{L^{2}_{\nu_{N}}L^{2}_{T}H^{-\s}_{x}}\longrightarrow 0,
\end{equation*}
when $N\longrightarrow +\infty$. The convergence of the two other terms is obtained as in Lemma~\ref{lem.56}.
\end{proof}

\begin{rema}
Observe that in all the proof, we only used the fact that $\Psi_{N}\in L^{2}(\text{d}\mu)$ uniformly in\;$N\geq 1$ (and not higher order integrability). Therefore the result of Theorem~\ref{thmDNLS} holds for $\kappa\leq \kappa_{2}$, and the support of $\rho$ is not empty. 
\end{rema}
\section{The half-wave equation}\label{Sect.7}

\subsection{Justification of the equation}
\begin{proof}[Proof of Proposition~\ref{prop.NL}]
We prove the result when $p=2$. The general case follows by the Wiener chaos estimates. 

To begin with, use that 
\begin{equation*}
\int_{X^{0}(\T)}\|G_{N}(u)-G_{M}(u)\|^{2}_{H^{-\s}(\T)}\text{d}\mu(u)=\int_{\Omega}\|G_{N}(\phi)-G_{M}(\phi)\|^{2}_{H^{-\s}(\T)}\text{d}{\bf p}.
\end{equation*}
Therefore, we are reduced to prove that  $\big(G_{N}(\phi)\big)_{N\geq1}$ is a Cauchy sequence in $L^{2}\big(\Omega; H^{-\s}(\T)\big)$. Denote by (with $\phi_{N}=\Pi_{N}\phi$)
\begin{equation*}
\chi_{N}=|\phi_{N}|^{2}\phi_{N}-2\|\phi_{N}\|^{2}_{L^{2}(\T)}\phi_{N}.
\end{equation*}
It is enough to show the result for $(\chi_{N})$, because once we know that $\chi_{N} \longrightarrow \chi$ in $L^{2}\big(\Omega; H^{-\s}(\T)\big)$, we deduce that $G_{N}(\phi)=\Pi_{N}\chi_{N} \longrightarrow \chi$ in $L^{2}\big(\Omega; H^{-\s}(\T)\big)$. In the sequel, we will use the notation  $[n]=1+|n|$. Then, by definition of $\phi_{N}$ we can compute 
\begin{eqnarray*}
\chi_{N}
&=&\sum_{{|n_{1}|,|n_{2}|,|n_{3}|\leq N}}\frac{g_{n_{1}}\ov{g}_{n_{2}}g_{n_{3}}}{[n_{1}]^{\frac12}[n_{2}]^{\frac12}[n_{3}]^{\frac12}}\e^{i(n_{1}-n_{2}+n_{3})x}-2\sum_{{|n_{1}|,|n_{3}|\leq N}}\frac{|g_{n_{1}}|^{2}g_{n_{3}}}{[n_{1}][n_{3}]^{\frac12}}\e^{in_{3}x}\\
&=&\sum_{\substack{|n_{1}|,|n_{2}|,|n_{3}|\leq N,\\n_{1}\neq n_{2},n_{3}\neq n_{2}}}\frac{g_{n_{1}}\ov{g}_{n_{2}}g_{n_{3}}}{[n_{1}]^{\frac12}[n_{2}]^{\frac12}[n_{3}]^{\frac12}}\e^{i(n_{1}-n_{2}+n_{3})x}.
\end{eqnarray*}
Next, denote by $\e_{k}(x)=\e^{ikx}$. Then for all $1\leq M\leq N$
\begin{equation}\label{crochet}
\<\chi_{N}-\chi_{M}\,|\,\e_{k}\>=\sum_{B^{(k)}_{M,N}}\frac{g_{n_{1}}\ov{g}_{n_{2}}g_{n_{3}}}{[n_{1}]^{\frac12}[n_{2}]^{\frac12}[n_{3}]^{\frac12}},
\end{equation}
 where the set $B^{(k)}_{M,N}$ is defined by
\begin{multline*}
B^{(k)}_{M,N}=\Big\{(n_{1},n_{2},n_{3})\in \Z^{3}\;\;\text{s.t.}\;\; 0<|n_{1}|, |n_{2}|  ,|n_{3}|\leq N,\;\;n_{1}\neq n_{2},\;\; n_{3}\neq n_{2},\\
\;\;\text{and}\;\;\big(|n_{1}|>M         \;\;\text{or}\;\; |n_{2}|>M   \;\;\text{or}\;\; |n_{3}|>M    \big) \;\;\text{and}\;\; n_{1}-n_{2}+n_{3}=k\Big\}.
\end{multline*} 
From~\eqref{crochet} we obtain 
\begin{equation*}
\big\|\<\chi_{N}-\chi_{M}\,|\,\e_{k}\>\big\|^{2}_{L^{2}(\Omega)}=
{{\int_{\Omega}}}\sum_{\substack {(n_{1},n_{2},n_{3})\in B^{(k)}_{M,N}\\(m_{1},m_{2},m_{3})\in B^{(k)}_{M,N}}}
\frac{g_{n_{1}}\ov{g}_{n_{2}}g_{n_{3}}\ov{g}_{m_{1}}{g}_{m_{2}}\ov{g}_{m_{3}}}
{[n_{1}]^{\frac12}[n_{2}]^{\frac12}[n_{3}]^{\frac12}[m_{1}]^{\frac12}[m_{2}]^{\frac12}[m_{3}]^{\frac12}}\text{d}{\bf p}.
\end{equation*}
Since the $(g_{n})$ are independent and centered, we deduce that each term in the r.h.s. vanishes, unless $n_{2}=m_{2}$ and $(n_{1},n_{3})=(m_{1},m_{3})$ or $(n_{1},n_{3})=(m_{3},m_{1})$. Thus 
\begin{equation*} 
\big\|\<\chi_{N}-\chi_{M}\,|\,\e_{k}\>\big\|^{2}_{L^{2}(\Omega)}\leq C\sum_{(n_{1},n_{2},n_{3})\in B^{(k)}_{M,N}}\frac{1}{\<n_{1}\>\<n_{2}\>\<n_{3}\>}.
\end{equation*}
By symmetry in the previous sum, we can assume that $M<|n_{1}|\leq N$, $0<|n_{2}|\leq N$ and write $n_{3}=k+n_{2}-n_{1}$. Then by~\eqref{borne.1} for some small $\eps>0$
\begin{eqnarray}\label{proj}
\big\|\<\chi_{N}-\chi_{M}\,|\,\e_{k}\>\big\|^{2}_{L^{2}(\Omega)}&\leq &C\sum_{M<|n_{1}|\leq N}\frac1{\<n_{1}\>}\sum_{n_{2}\in \Z}\frac{1}{\<n_{2}\>\<n_{2}-(n_{1}-k)\>}\nonumber \\
& \leq &C\sum_{M<|n_{1}|\leq N}\frac{1}{\<n_{1}\>\<n_{1}-k\>^{1-\eps}} \leq \frac{C}{M^{\eps}\<k\>^{1-2\eps}}.
\end{eqnarray}
Now, by~\eqref{proj} we get
\begin{eqnarray*}
\big\|\chi_{N}-\chi_{M}\big\|^{2}_{L^{2}(\Omega;H^{-\s}(\T))}&=&\sum_{k\in \Z}\frac1{\<k\>^{2\s}} \big\|\<\chi_{N}-\chi_{M}\,|\,\e_{k}\>\big\|^{2}_{L^{2}(\Omega)}  \\
&\leq &\frac{C}{M^{\eps}}\sum_{k\in \Z}\frac1{\<k\>^{1+2\s-2\eps}}\leq \frac{C}{M^{\eps}},
\end{eqnarray*}
if we choose $\eps<\s$, and this concludes the proof.

As a conclusion, we are able to define a limit  $G(u)$ so that  for all $p\geq 2$  
\begin{equation}\label{Glp}
\|G(u) \|_{L^p_{\mu}H^{-\s}(\S^1)}\leq C_p,
\end{equation}
hence the result.
\end{proof}

\subsection{Construction of the measure $\boldsymbol \rho$}
 In this section $\phi$ is given by~\eqref{phi.HW}. Denote by ${[n]=1+|n|}$, then define  $\dis \alpha_{N}=\sum_{|n|\leq N}\frac1{[n]}$ and  
\begin{equation*}
g_{N}(u)=\|\Pi_{N}u\|^{2}_{L^{2}}-\alpha_{N}.
\end{equation*}
\subsubsection{Preliminar results}
We begin with the following result due to N. Tzvetkov. See~\cite[Lemma~4.8]{Tzvetkov3} for a proof.
\begin{lemm}\ph\label{lem.71} 
The sequence $\big(g_{N}(u)\big)_{N\geq1}$ is  Cauchy in $L^{2}\big(X^{0}(\T),\mathcal{B},d\mu\big)$. Moreover there exists $c>0$ so that for all $\lambda>0$ and   $N>M\geq 1$
\begin{equation*} 
\mu\Big(u\in X^{0}(\S^{1})\;:\;   |g_{N}(u)-g_{M}(u)|>\lambda \Big) \leq C \e^{-c \lambda M^{1/2}}.
\end{equation*} 
\end{lemm}

Define the sequence 
\begin{equation}\label{def.fn_bis}
 f_{N}(u)=-\int_{\T}|u_{N}|^{4}+2\big(\int_{\T}|u_{N}|^{2}\big)^{2}=-\|u_{N}\|^{4}_{L^{4}}+2\|u_{N}\|^{4}_{L^{2}}.
\end{equation}
 \begin{prop}\ph\label{prop.72}
The sequence $(f_{N})_{N\geq1}$ is  Cauchy in $L^{2}\big(X^{0}(\T),\mathcal{B},d\mu\big)$. More precisely, there exists $C>0$ so that for all $N>M\geq 1$
\begin{equation}\label{L2_pak}
\|f_{N}(u)-f_{M}(u)\|_{L^{2}\big(X^{0}(\T), \mathcal{B},\text{d}\mu\big)}\leq \frac{C}{M^{\frac12}}.
\end{equation} 
Moreover, for all $p\geq 2$ and  $N>M\geq 1$
\begin{equation}\label{LP}
\|f_{N}(u)-f_{M}(u)\|_{L^{p}\big(X^{0}(\T), \mathcal{B},\text{d}\mu\big)}\leq  \frac{C\,(p-1)^{2}}{M^{\frac12}}.
\end{equation} 
\end{prop}

\begin{coro}\ph\label{coro.73}
There exists $c>0$ so that for all $\lambda>0$ and   $N>M\geq 1$
\begin{equation*} 
\mu\Big(u\in X^{0}(\S^{1})\;:\;   |f_{N}(u)-f_{M}(u)|>\lambda \Big) \leq C \e^{-c \lambda^{1/2}M^{1/4}}.
\end{equation*} 
\end{coro}

\begin{proof}[Proof of Corollary~\ref{coro.73}]
By Markov and~\eqref{LP} we have that for all $p\geq 2$
\begin{equation*}
\mu\Big(u\in X^{0}(\S^{1})\;:\;   |f_{N}(u)-f_{M}(u)|>\lambda \Big) \\\leq  \frac1{\lambda^{p}}\|f_{N}(u)-f_{M}(u)\|^{p}_{L^{p}\big(X^{0}(\S^{1}), \mathcal{B},\text{d}\mu\big)}
\leq \big(\frac{Cp^{2}}{\lambda M^{1/2}}\big)^{p}.
\end{equation*}
Then choose $p=c_{0}\lambda^{1/2}M^{1/4}$ for $c_{0}>0$ small enough.
\end{proof}

\begin{proof}[Proof of Proposition~\ref{prop.72}]
We prove~\eqref{L2_pak}. The estimate~\eqref{LP} immediately  follows from~\cite[Proposition 2.4]{ThTz}. Firstly, we have $\dis \int_{\T} |\phi_{N}|^{2}=\sum_{|n|\leq N} \frac{|g_{n}|^{2}}{[n]}$, with the notation $[n]=1+|n|$. Thus 
\begin{equation}\label{carre}
\big(   \int_{\T} |\phi_{N}|^{2}\big)^{2}=\sum_{|n|,|m|\leq N} \frac{|g_{n}|^{2}|g_{m}|^{2}}{[n][m]}.
\end{equation}
Similarly, we explicitly obtain
\begin{equation}\label{fi4}
    \int_{\T} |\phi_{N}|^{4} =\sum_{\substack{     |n_{1}|,|n_{2}|, |n_{3}|,|n_{4}|\leq N\\ n_{1}-n_{2}+n_{3}-n_{4}=0  }} \frac{g_{n_{1}}\ov{g_{n_{2}}}{g_{n_{3}}}\ov{g_{n_{4}}}}{[n_{1}]^{\frac12}[n_{2}]^{\frac12}[n_{3}]^{\frac12}[n_{4}]^{\frac12}}.
\end{equation}
We introduce the set 
\begin{equation*}
A_{N}=\{(n_{1},n_{2},n_{3},n_{4})\in \Z^{4}\;\;\text{s.t.}\;\; |n_{1}|, |n_{2}|,  |n_{3}|, |n_{4}| \leq N\;\;\text{and}\;\; n_{1}-n_{2}+n_{3}-n_{4}=0\}.
\end{equation*}

 We now split the sum~\eqref{fi4} in two parts, by distinguishing the cases $n_{3}=n_{1}$ and $n_{3}\neq n_{1}$ in $A_{N}$ and write 
\begin{equation}\label{decomp}
\int_{\T} |\phi_{N}|^{4} =X_{N}+Y_{N},
\end{equation} 
with 
\begin{equation*}
X_{N}=\sum_{B_{N}} \frac{g_{n_{1}}\ov{g_{n_{2}}}g_{n_{3}}\ov{g_{n_{4}}}}{[n_{1}]^{\frac12}[n_{2}]^{\frac12}[n_{3}]^{\frac12}[n_{4}]^{\frac12}},
\end{equation*} 
where $B_{N}=A_{N}\cap\{\,n_{1}=n_{2}\;\;\text{or}\;\;n_{1}=n_{4}\,\}$, and
\begin{equation}\label{yn}
Y_{N}=\sum_{\substack{A_{N}, n_{1}\neq n_{2}\\n_{1}\neq n_{4}}} \frac{g_{n_{1}}\ov{g_{n_{2}}}g_{n_{3}}\ov{g_{n_{4}}}}{[n_{1}]^{\frac12}[n_{2}]^{\frac12}[n_{3}]^{\frac12}[n_{4}]^{\frac12}}.
\end{equation} 
We observe that if $(n_{1},n_{2},n_{3},n_{4})\in B_{N}$, then either $(n_{1},n_{3})=(n_{2},n_{4})$ or $(n_{1},n_{3})=(n_{4},n_{2})$. Thus
\begin{eqnarray*}
X_{N}&=&\sum_{|n_{1}|,|n_{3}|\leq N} \frac{|g_{n_{1}}|^{2}|g_{n_{3}}|^{2}}{[n_{1}][n_{3}]}+\sum_{\substack{|n_{1}|,|n_{3}|\leq N\\n_{1}\neq n_{3}}} \frac{|g_{n_{1}}|^{2}|g_{n_{3}}|^{2}}{[n_{1}][n_{3}]}\nonumber \\
&=&2\big(   \int_{\T} |\phi_{N}|^{2}\big)^{2}-\sum_{|n|\leq N}\frac{|g_{n}|^{4}}{[n]^{2}},
\end{eqnarray*}
where in the last line we used~\eqref{carre}. Thus, with~\eqref{decomp} we obtain
\begin{equation*}
f_{N}(\phi_{N})=-\int_{\T} |\phi_{N}|^{4}+2\big(\int_{\T} |\phi_{N}|^{2}\big)^{2}=\sum_{|n|\leq N}\frac{|g_{n}|^{4}}{[n]^{2}}-Y_{N}.
\end{equation*}
We now show that $(Y_{N})_{N\geq 1}$ is Cauchy in $L^{2}(\Omega, \mathcal{F},{\bf p})$. Let $1\leq N<M$, then we define 
\begin{multline*}
A_{M,N}=\big\{(n_{1},n_{2},n_{3},n_{4})\in \Z^{4}\;\;\text{s.t.}\;\; M<|n_{1}|, |n_{2}|,  |n_{3}|, |n_{4}| \leq N,\\
n_{1}-n_{2}+n_{3}-n_{4}=0\;\;\text{and s.t.}\;\;|n_{j}|>M\;\;\text{for some}\;\; 1\leq j\leq4\big\}.
\end{multline*} 
Thus, thanks to~\eqref{yn} we have
\begin{equation*} 
(Y_{M}-Y_{N})^{2}=\sum_{\substack{A_{M,N},\\ n_{1}\neq n_{2}\\n_{1}\neq n_{4}}}\sum_{\substack{A_{M,N},\\ m_{1}\neq m_{2}\\m_{1}\neq m_{4}}} \frac{g_{n_{1}}\ov{g_{n_{2}}}g_{n_{3}}\ov{g_{n_{4}}}}{[n_{1}]^{\frac12}[n_{2}]^{\frac12}[n_{3}]^{\frac12}[n_{4}]^{\frac12}} \frac{\ov{g_{m_{1}}}{g_{m_{2}}}\ov{g_{m_{3}}}{g_{m_{4}}}}{[m_{1}]^{\frac12}[m_{2}]^{\frac12}[m_{3}]^{\frac12}[m_{4}]^{\frac12}}.
\end{equation*} 
We take the integral over $\Omega$ of the previous sum. By the independence of the Gaussians each term vanishes unless $\{n_{1},n_{2},n_{3},n_{4}\}=\{m_{1},m_{2},m_{3},m_{4}\}$. Thus 
\begin{equation*}
\|Y_{M}-Y_{N}\|^{2}_{L^{2}(\Omega)}\leq C \sum_{A_{M,N}} \frac{1}{\<n_{1}\>\<n_{2}\>\<n_{3}\>\<n_{4}\>}.
\end{equation*}
By symmetry of the sum, we can assume that $|n_{1}|\geq M$ and we replace $n_{4}=n_{1}-n_{2}+n_{3}$. Then by~\eqref{borne.1}
\begin{eqnarray*}
\|Y_{M}-Y_{N}\|^{2}_{L^{2}(\Omega)}&\leq &C \sum_{  \substack{n_{1},n_{2},n_{3}\in \Z\\ |n_{1}|>M}} \frac{1}{\<n_{1}\>\<n_{2}\>\<n_{3}\>\<n_{1}-n_{2}+n_{3}\>}\\
&\leq &  C \sum_{  \substack{n_{1},n_{2}\in \Z\\ |n_{1}|>M}}  \frac{1}{\<n_{1}\>\<n_{2}\>\<n_{1}-n_{2}\>^{1-\eps}}   \\
&\leq&C\sum_{|n_{1}|\geq M}\frac1{\<n_{1}\>^{2-2\eps}}\leq \frac{C}{M^{1-2\eps}},
\end{eqnarray*}
which was the claim.
\end{proof}
 
\subsubsection{The crucial estimate} We now have all the ingredients to prove the following proposition,
which is the key point in the proof of Theorem~\ref{thm2}. Recall the definition~\eqref{def.fn_bis}.
\begin{prop}\ph\label{prop.lp}
Let $\chi\in \mathcal{C}^{\infty}_{0}\big([-R,R]\big)$. Then for all $1\leq p<\infty$ there exists $C>0$ such that for every $N\geq 1$,
$$
\Big\|
\chi\Big(\|\Pi_{N}u\|_{L^2(\T)}^{2}-\alpha_{N}\Big)
\e^{f_{N}(u)}
\Big\|_{L^p(\text{d}\mu(u))}\leq C\,.
$$
\end{prop}
\begin{proof}
Our aim is to show that the integral
$\int_{0}^{\infty}\lambda^{p-1}\mu(A_{\lambda,N})d\lambda$ 
is convergent uniformly with respect to $N$, where
\begin{equation*}
A_{\lambda,N}=
\Big\{u\in X^{0}(\S^{1})\,:\,
\chi\Big(\|\Pi_{N}u\|_{L^2(\T)}^{2}-\alpha_{N}\Big)
e^{f_{N}(u)}>\lambda
\Big\}.
\end{equation*}
Proposition~\ref{prop.lp} is a straightforward consequence of the following lemma.
\end{proof}
\begin{lemm}\ph For any $L>0$, there exists $C>0$ such that for every $N$ and
every $\lambda\geq 1$, 
$$\mu(A_{\lambda, N}) \leq C \lambda ^{-L}.$$
\end{lemm}

\begin{proof}
\noindent Firstly, observe that we can assume that $\lambda\geq C_{R}$ for any constant $C_{R}>0$. Let $c_{0}>0$ a small number which will be fixed later and set 
$$\dis M=\e^{c_{0}(\ln \lambda)^{1/2}}.$$
 To begin with $\mu(A_{\lambda,N})\leq \mu(\wt{A}_{\lambda,N})$, where 
\begin{equation*} 
\wt{A}_{\lambda,N}=\Big\{u\in X^{0}(\S^{1})\,:\, 
f_{N}(u)>\ln \lambda,\quad |g_{N}(u)|\leq R\Big\}.
\end{equation*}
$\bullet$ Assume that $N\leq M$. 
On the set $\big\{\,|g_{N}(u)|\leq R+1\,\big\}$ we have 
\begin{equation*}
f_{N}(u)\leq 2\|\Pi_{N}u\|^{4}_{L^{2}(\T)}\leq 2(C\ln N+R)^{2}\leq 2(C\ln M+R)^{2}=Cc^{2}_{0}\ln \lambda,
\end{equation*}
if $\lambda \geq C_{R}$ large enough. We fix  $c_{0}>0$ so that $Cc^{2}_{0}<1/4$. In particular $\mu(A_{\lambda,N})\leq \mu(\wt{A}_{\lambda,N})=0$.\\
$\bullet$ Assume that $N\geq M$. First observe that if we define 
\begin{equation*}
B_{\lambda,N}=\Big\{u\in X^{0}(\S^{1})\,:\,
\big|g_{N}(u)-g_{M}(u)\Big|>1\Big\},
\end{equation*}
by Lemma~\ref{lem.71} and the definition of $M$, we get  for any $L\geq 1$
$$
\mu(B_{\lambda,N})\leq C\exp(-cM^{1/2})\leq C_{L}\lambda^{-L}\,.
$$
Similarly, set 
\begin{equation*}
C_{\lambda,N}=\Big\{u\in X^{0}(\S^{1})\,:\,
\big|f_{N}(u)-f_{M}(u)\Big|>1\Big\},
\end{equation*}
then by Corollary~\ref{coro.73},  for any $L\geq 1$ we have 
$$
\mu(C_{\lambda,N})\leq C\exp(-cM^{1/4})\leq C_{L}\lambda^{-L}\,.
$$
We have $\wt{A}_{\lambda,N} \subset C_{\lambda,N} \cup D_{\lambda,N}$ where
\begin{equation*}
D_{\lambda,N}=\Big\{u\in X^{0}(\S^{1})\,:\, f_{M}(u)>\frac12\ln \lambda ,\quad |g_{N}(u)|\leq R\Big\}.
\end{equation*}
Then observe that  $\dis \big\{\,|g_{N}(u)|\leq R\,\big\}\cap \big\{\,|g_{N}(u)-g_{M}(u)|\leq 1\,\big\}\subset \big\{|g_{M}(u)|\leq R+1\big\}$, therefore we can write $ D_{\lambda,N}\subset B_{\lambda,N}\cup E_{\lambda,N}$ where 
\begin{equation*}
E_{\lambda}=\Big\{u\in X^{0}(\S^{1})\,:\, f_{M}(u)>\frac12\ln \lambda ,\quad |g_{M}(u)|\leq R+1\Big\}.
\end{equation*}
In the first part of the proof, we have already shown that $\mu(E_{\lambda})=0$. Finally, we put all the estimates together and obtain $\mu(A_{\lambda,N})\leq C_{L}\lambda^{-L}$.
\end{proof}
\subsubsection{Convergence to the mesure $\rho$}

We now have all the ingredients to complete the proof of Theorem~\ref{thm2}. 

First we define the density   $\Theta\,:\,X^{0}(\S^{1})\longrightarrow \R$ with respect to the measure $\mu$ of the measure $\rho$. By Lemma~\ref{lem.71} and Proposition~\ref{prop.72}, we have the following convergences in the $\mu$ measure: $g_{N}(u)$ converges  to $g(u)$ and $f_{N}(u)$ to $f(u)$. Then, by composition and multiplication of continuous functions, we obtain 
\begin{equation*} 
\Theta_{N}(u)
\longrightarrow \beta \chi\big(g(u)\big)
\e^{f(u)}\equiv \Theta(u),
\end{equation*}
in measure, with respect to the measure $\mu$, and where $\beta>0$ is so that $\text{d}\rho(u)=\Theta(u) \text{d}\mu(u)$ is a probability measure on $X^{0}(\S^{1})$. By this construction,  $\Theta$ is measurable from $\big(X^{0}(\S^{1}), \mathcal{B}\big)$ to~$\R$.

Then, we can extract a sub-sequence $\Theta_{N_{k}}(u)$ so that $\Theta_{N_{k}}(u)\longrightarrow \Theta(u)$, $\mu$ a.s. and by Proposition~\ref{prop.lp} and the Fatou lemma, for all $p\in [1,+\infty)$,
\begin{equation*}
\int_{X^{0}(\S^{1})}|\Theta(u)|^{p}\text{d}\mu(u)\leq \liminf_{k\to \infty} \int_{X^{0}(\S^{1})}|\Theta_{N_{k}}(u)|^{p}\text{d}\mu(u)\leq C,
\end{equation*}
thus $\Theta(u)\in L^{p}(\text{d}\mu(u))$.

It remains to prove the convergence of $\Theta_{N}(u)$ in $L^{p}(\text{d}\mu(u))$: Here we can follow the proof of Proposition~\ref{prop.46}. We do not write the details. 
 \subsection{Study of the measure $\boldsymbol {\nu_{N}}$}

Let $N\geq 1$ and consider the  equation~\eqref{HW.N}. Observe that $u_{N}=\Pi_{N}u$ satisfies an ODE, while $u^{\perp}_{N}=(1-\Pi_{N})u$ is solution to the linear problem $(i\partial_{t}-\Lambda)u^{\perp}_{N}=0$. Since the $L^{2}(\S^{1})$-norm of a solution $u$ to~\eqref{HW.N} is preserved, it follows that the equation is globally well-posed in $L^{2}(\T)$. We denote by $\Phi_{N}$ the   flowmap. Moreover, because of the Hamiltonian structure and the Liouville theorem, the measure $\rho_{N}$  is invariant by $\Phi_{N}$.

Similarly to the  previous section, for $T>0$ we define  the measure $\nu_{N}$ on $\mathcal{C}\big([-T,T]; X^{0}(\T)\big)$ as the image of $\rho_{N}$ by the flowmap 
\begin{equation*}
 \begin{array}{rcc}
X^{0}(\T)&\longrightarrow& \mathcal{C}\big([-T,T]; X^{0}(\T)\big)\\[3pt]
\dis  v&\longmapsto &\dis \Phi_{N}(t)(v).
 \end{array}
 \end{equation*}
Using this definition, we can prove
\begin{lemm}\ph Let $\s>0$,  then  for all $p\geq 2$
\begin{equation}\label{Je0}
\big\| \|  G(u)\|_{L^{p}_{T}H^{-\s}_{x}}\big\|_{L^{p}_{\nu_{N}}}\leq C.
\end{equation}
\end{lemm}

\begin{proof}
 By definition, invariance of $\rho_N$ and Cauchy-Schwarz
 \begin{eqnarray*} 
\|G(u)\|^{p}_{L^{p}_{\nu_{N}}L^{p}_{T}H^{-\s}_{x}}&=&\int_{  \mathcal{C}\big([-T,T]; X^{0}\big)}\|G(u)\|^{p}_{L^{p}_{T}H^{-\s}_{x}}\text{d}\nu_{N}(u) \\
&=&\int_{X^{0}} \|G\big(\Phi_{N}(t)(v)\big)\|^{p}_{L^{p}_{T}H^{-\s}_{x}} \text{d}\rho_{N}(v) \\
&=&2T\int_{X^{0}} \|G(v)\|^{p}_{H^{-\s}_{x}}\theta_N(v) \text{d}\mu(v) \\
&\leq &2T \|G(v)\|^{p}_{L^{2p}_{\mu}H^{-\s}_{x}}\|\theta_N(v)\|_{L^{2}_{\mu}}. 
\end{eqnarray*}
 We conclude with~\eqref{Glp} and Proposition~\ref{prop.lp}.
\end{proof}

\begin{lemm}\ph Let $\s>0$,  then   for all $p\geq 2$
\begin{equation}\label{Je1}
\big\|\|  u\|_{L^{p}_{T}H^{-\s}_{x}}\big\|_{L^{p}_{\nu_{N}}}\leq  C,
\end{equation}
\begin{equation}\label{Je2}
\big\| \| u\|_{W^{1,p}_{T}H^{-\s-1}_{x}}\big  \|_{L^{p}_{\nu_{N}}}\leq  C.
\end{equation}
\end{lemm}

\begin{proof}
The proof of~\eqref{Je1} is a consequence of~\eqref{Bor1} and  Lemma~\ref{lem.52}. The estimate~\eqref{Je2} is obtained from~\eqref{Je0} and~\eqref{Je1}: The proof is similar to~\eqref{Cl2} and we do not write the details.
\end{proof}

As a consequence we can show
\begin{prop}\ph 
Let $T>0$ and $\s>0$. Then the family of measures 
$$ \nu_{N}=\mathscr{L}_{\mathcal{C}_{T}H^{-\s}}\big(u_{N}(t);t\in [-T,T]\big)_{N\geq 1}$$
 is tight in $\mathcal{C}\big([-T,T]; H^{-\s}(\T)\big).$
\end{prop}

\subsection{Proof of Theorem~\ref{thmHW}} The proof is similar to the Benjamin-Ono case. The only difficulty lies in the limit of the nonlinear term.  Recall the definition~\eqref{gauge}, then
\begin{lemm}\ph
Up to a sub-sequence, the following convergence holds true
\begin{equation*} 
G_{N_{k}}(\wt{u}_{N_{k}}) \longrightarrow G(\wt{u}),\quad \;\;\widetilde{\bf p}-\text{a.s. in}\;\; L^{2}\big([-T,T]; H^{-\s}(\T)\big),
\end{equation*}
where $G$ is defined by Proposition~\ref{prop.NL}.
\end{lemm}

\begin{proof} We only give the main lines, and we refer to the proof of  Lemma~\ref{lem.56} for the details.  We drop the tildes and write $N_{k}=k$. Let $M\geq1$ and write 
  \begin{multline*}
 G_{k}(u_{k})-G(u)= \\ =\big(   G_{k}(u_{k})-G(u_{k})\big)+ \big(   G(u_{k})-G_{M}(u_{k})\big)+\big(   G_{M}(u_{k})-G_{M}(u)\big)+\big(   G_{M}(u)-G(u)\big).
 \end{multline*}
  For fixed $M\geq 1$, using that $G_{M}$ is continuous  
\begin{equation*}
      G_{M}(u_{k})\underset{k\to +\infty}{\longrightarrow} G_{M}(u) ,\quad \;\;\widetilde{\bf p}-\text{a.s. in}\;\; L^{2}\big([-T,T]; H^{-\s}(\T)\big).
\end{equation*}
For the other terms, we use the definition of the measure $\nu_{k}$ and Proposition~\ref{prop.NL} to prove    the convergence (when $k\to +\infty$ or $M\to +\infty$) in the space $X:=L^{2}\big(\Omega \times[-T,T]; H^{-\s}(\T)\big)$. Then  the almost sure convergence  follows after exaction of a sub-sequence. 
\end{proof}
\section{The two dimensional nonlinear Schr\"odinger equation on an arbitrary domain }\label{Sect.8}
\subsection{Estimates on the spectral function in mean value}\label{sec.8.1}
The following propositions which will be proved in Section~\ref{sec.85} are the key elements in our argument. 
The first one is a rough bound which states that (in a mean value meaning with respect to the index $n$), the eigenfunctions are uniformly bounded on $M$. We state it for windows of size $1$ for the spectral projector.
\begin{prop} \ph\label{WW1}
There exists $C>0$ such that for any orthonormal basis $(\varphi_n)_{n\geq 0}$ of eigenfunctions of $-\Delta_g$,  and any $\mu>0$, any $x\in M$, we have 
\begin{equation*} 
\sum_{\lambda _m \in [\mu, \mu+ 1)} \frac 1 { \lambda_m^2 + 1}  | \phi_m|^2( x) 
\leq C  \sum_{\lambda _m \in [\mu, \mu+ 1)} \frac 1 { \lambda_m^2 + 1} .
\end{equation*}
\end{prop}
The second one is more precise and states that (again in a mean value meaning), the eigenfunctions are actually constant on $M$ (away from the boundary and modulo errors). Remark that our assumption that $\text{Vol}(M) =1$ and the $L^2$ normalization of eigenfunctions imply that this constant has to be $1$. 
\begin{prop}\ph\label{ww}
There exists $C>0$ such that for any orthonormal basis $(\varphi_n)_{n\geq 0}$ of eigenfunctions of $-\Delta_g$,  and any $\mu>0, \delta \in [0,1]$, we have 
\begin{equation*}
\sum_{\lambda _m \in [\mu, \mu+ \delta \mu^{1/2})} \frac 1 { \lambda_m^2 + 1}  | \phi_m|^2( x) \\
= \sum_{\lambda _m \in [\mu, \mu+ \delta \mu^{1/2})} \frac 1 { \lambda_m^2 + 1} + G_{\mu}(x),
\end{equation*}
with 
\begin{equation} \label{nul}
\int_MG_{\mu}(x)dx =0,
\end{equation}
and
\begin{equation}\label{defG_bis}
|G_{\mu}(x)|\leq C\mu^{-3/4}+ C\mu^{-1/2}{\bf 1}_{\{d(x)<\mu^{-1/2}\}},
\end{equation}
where for $x\in M$, $d(x)$ is the distance of $x$ to the boundary of $M$.
\end{prop}
\begin{rema} The introduction of windows $[\mu, \mu+ \delta\mu^{1/2}[$ is required by the analysis near the boundary but are unnecessary for manifolds without boundaries, in which case elementary version of Proposition~\ref{WW1} are sufficient.
\end{rema}
\subsection{Definition of the Gibbs measure}\label{sec.8.2}
The aim of this paragraph is to prove Theorem ~\ref{gib}.  To begin with, we decompose, on the support of the measure $\mu$, the quartic term $\|\Pi_N u\|_{L^4(M)}^4$ in the Hamiltonian in order to get a suitable renormalisation in the following section.
\subsubsection{Decomposition of  $\|\Pi_N u\|_{L^4(M)}^4$ on the support of $\mu$}
Recall that $\alpha_N$ is defined in\;\eqref{defa} and that       $\Pi_{N}u=\sum_{n\leq N} c_n\varphi_n$   for ${u=\sum_{n\geq 0}c_n\varphi_n}$. Therefore we can write  
$$
\|\Pi_N u\|_{L^4(M)}^4=
\sum_{n_1,n_2,n_3,n_4\leq N}
c_{n_1}\overline{c_{n_2}}c_{n_3}\overline{c_{n_4}}\,\gamma(n_1,n_2,n_3,n_4),
$$
where 
$$
\gamma(n_1,n_2,n_3,n_4)=\int_{M}\varphi_{n_1}\overline{\varphi_{n_2}}\varphi_{n_3}\overline{\varphi_{n_4}}\,.
$$
Next, we set
$$
\Lambda=\big\{(n_1,n_2,n_3,n_4)\in \N^4\,:\, \{n_1,n_3\}=\{n_2,n_4\}\big\}\,.
$$
We denote by $\Lambda^c$ the complementary of $\Lambda$ in $\N^4$.
Therefore, we can split
$$
\|\Pi_N u\|_{L^4(M)}^4=
X_{1,N}(u)+X_{2,N}(u)+X_{3,N}(u),
$$
where
$$
X_{1,N}(u)=
\sum_{\substack{
(n_1,n_2,n_3,n_4)\in \Lambda^c
\\
n_1,n_2,n_3,n_4\leq N
}
}
c_{n_1}\overline{c_{n_2}}c_{n_3}\overline{c_{n_4}}\,\gamma(n_1,n_2,n_3,n_4),
$$
$$
X_{2,N}(u)=2\sum_{n_1,n_2\leq N}
|c_{n_1}|^2|c_{n_2}|^2\gamma(n_1,n_1,n_2,n_2),
$$
 and finally
 $$
 X_{3,N}(u)=
 -\sum_{n\leq N}|c_n|^4\gamma(n,n,n,n)\,.
$$
As we shall see the singular part of the $L^4$ norm on the support of $\mu$ is given by the contribution of~$X_{2,N}$. 
Indeed, let us study the behavior of $X_{2,N}$ on the support of $\mu$. Write
$$
X_{2,N}
\Big(
\sum_{n\geq 0}\frac{g_{n}(\omega)}{(\lambda^2_n+1)^{\frac{1}{2}}}\, \varphi_n(x)
\Big) 
=
2\sum_{n_1,n_2\leq N}
\frac{|g_{n_1}(\omega)|^2\, |g_{n_2}(\omega)|^2}
{(\lambda_{n_1}^2+1)(\lambda_{n_2}^2+1)}
\int_{M}|\varphi_{n_1}|^2|\varphi_{n_2}|^2\,.
$$
We can split the last expression as $I+II+III$, where
$$
I=2\sum_{n_1,n_2\leq N}
\frac{(|g_{n_1}(\omega)|^2-1)\, (|g_{n_2}(\omega)|^2-1)}
{(\lambda_{n_1}^2+1)(\lambda_{n_2}^2+1)}
\int_{M}|\varphi_{n_1}|^2|\varphi_{n_2}|^2\,,
$$
$$
II=4\sum_{n_1,n_2\leq N}
\frac{|g_{n_1}(\omega)|^2}
{(\lambda_{n_1}^2+1)(\lambda_{n_2}^2+1)}
\int_{M}|\varphi_{n_1}|^2|\varphi_{n_2}|^2\,,
$$
and
$$
III=
-2\sum_{n_1,n_2\leq N}\frac{1}{(\lambda_{n_1}^2+1)(\lambda_{n_2}^2+1)}
\int_{M}|\varphi_{n_1}|^2|\varphi_{n_2}|^2\,.
$$

$\bullet$ {\bf Study of the term $I$:} The term $I$ is a regular term and gives a contribution to $X_N(u)$. 
\medskip 

 $\bullet$ {\bf Study of the term $II$:} This one   will require the most delicate analysis. We have
$$
II=
4\sum_{n\leq N}
\frac{|g_{n}(\omega)|^2}
{\lambda_{n}^2+1}\sum_{m\leq N} \frac{1} { \lambda_m^2+1}\int_M |\phi_{n}|^2  | \phi_m|^2 dx.
$$

Let 
$$ \alpha_N = \sum_{0\leq m\leq N} \frac{1} { \lambda_m^2+1}= \mathbb{E}_{\mu} \big[ \|u_N\|_{L^2(M)}^2\big]$$
(notice that according to Weyl formula, $|\alpha_{N}|\lesssim (\ln(N))$).

We can write the segment $[\lambda_0, \lambda_N]$ as a disjoint union of intervals  $E_k= [\mu_k, \mu_k + d_k \mu_k ^{1/2}), {1 \leq k \leq M_N}$, with $d_k =1$ for  $k<M_N$ and $d_{M_N} \in [0,1]$. In other words we define $\mu_k$ by $\mu_0=\lambda_0$, 
$\mu_{k+1}=\mu_{k}+\mu_{k}^{1/2}$ for $k<M_N$.

 We can check that 
 \begin{equation}\label{sim_bis}
 \mu_k\sim c k^2,
 \end{equation}
 and this will be used in the sequel to study convergence of series.

By Proposition~\ref{ww}, if we denote by $\dis\sum_{k=1}^{M_N}G_{\mu_k}(x)=K_N(x)$ we have
\begin{equation}\label{decomp2}
\sum_{m\leq N} \frac 1 { \lambda_m^2 + 1}  | \phi_m|^2( x)=\a_N+K_N(x),
\end{equation}
and with this decomposition we can split
$$
II=II_{1}+II_{2},
$$
where 
$$
II_{1}=4\alpha_{N}\sum_{n_1\leq N}
\frac{|g_{n_1}(\omega)|^2}
{\lambda_{n_1}^2+1}
$$
and 
$$
II_2=
4\sum_{n\leq N}
\frac{|g_{n}(\omega)|^2}
{\lambda_{n}^2+1}\sum_{k\leq M_N} \int_M |\phi_{n}|^2 G_{\mu_k}(x) dx.
$$
We deduce 
$$
II_{2}=II_{21}+II_{22},
$$
where
\begin{equation*}
II_{21}=
4\sum_{n\leq N}
\frac{(|g_{n}(\omega)|^2-1)}
{\lambda_{n}^2+1}\sum_{k\leq M_N} \int_M |\phi_{n}|^2 G_{\mu_k}(x) dx
\end{equation*}
and 
\begin{equation*}
II_{22}=
4\sum_{n\leq N}
\frac{1}
{\lambda_{n}^2+1}\sum_{k\leq M_N} \int_M |\phi_{n}|^2 G_{\mu_k}(x) dx.
\end{equation*}
By Sobolev $ \| \phi_n \|_{L^\infty}\leq C\lambda_n$, then by interpolation
$$\int_{\{x\in M; d(x) <\mu_k^{-\frac 12 } \}} |\phi_{n}|^2(x)dx \leq C \inf(1, \mu_k^{-\frac 1 2} \| \phi_n \|_{L^\infty}^2)\leq  C\mu_k^{-\frac 1 {16}}  \lambda_n ^{\frac 1 4}.$$
Thanks to the previous inequality, we get
\begin{eqnarray*}
\sum_{k\leq M_N} \int_M |\phi_{n}|^2 G_{\mu_k}(x) dx&\leq& C\sum_{k\leq M_N} \Big( \mu_k^{-3/4}+ \mu_k^{-1/2} \int_{\{x\in M; d(x) <\mu_k^{-\frac 12 } \}} |\phi_{n}|^2(x)dx \Big)\\
&\leq& C\lambda^{1/4}_n,
\end{eqnarray*}
where we have used~\eqref{sim_bis}.

We are now able to show that the term $II_{21}$ is regular because 
\begin{equation*}
\E \big[  \vert II_{21}\vert^2     \big] =C \sum_{n\leq N} \frac 1 
{(\lambda_{n}^2+1)^2}\Bigl|  \sum_{k\leq M_N} \int_M |\phi_{n}|^2 G_{\mu_k}(x) dx \Bigr| ^2
\leq C\sum_{n\geq 0} \frac {1} { (\lambda_{n}^2 + 1) ^{7/4} } < + \infty ,
\end{equation*}
where in the last line  we used that by Weyl formula, we have $\lambda_n \sim \sqrt{n} $.

On the other hand, $II_{22}$ is a constant (and hence can be renormalized). However, we want to keep track of the necessary renormalization involved (and to compare them with the usual ones). Hence, we apply again Proposition~\ref{WW1} to the index $n$ now  (with the same decomposition on $[\lambda_1, \lambda_N]$). With~\eqref{decomp2} and\;\eqref{nul} we get 
\begin{eqnarray*} 
II_{22}&=& 4 \alpha_N\sum_{k\leq M_N}\int_M G_{\mu_k}(x) dx+ 4 \sum_{k,\ell \leq M_N}\int_M G_{\mu_k}(x) G_{\mu_{\ell}}(x)  dx \nonumber \\
&=&4\int_M  K^2_N(x)  dx,
\end{eqnarray*}
and with~\eqref{defG_bis} we prove that this term  is  uniformly bounded with respect to\;$N$. Actually
\begin{equation}\label{equiibis}
\begin{aligned}
II_{22}&\leq C\sum_{k,\ell \leq M_N}\int_{d(x) <\min(\mu_k^{-1/2 },\mu^{-1/2}_{\ell} )} G_{\mu_k}(x) G_{\mu_{\ell}}(x)  dx +\\
& +C\sum_{k\leq \ell \leq M_N}\int_{\mu_{\ell}^{-1/2 } \leq d(x) \leq \mu^{-1/2}_{k} } G_{\mu_k}(x) G_{\mu_{\ell}}(x) dx +\\
& +C\sum_{k,\ell \leq M_N}\int_{d(x) >\max(\mu_k^{-1/2 },\mu^{-1/2}_{\ell} )} G_{\mu_k}(x) G_{\mu_{\ell}}(x)  dx +\\
&\leq  C\sum_{k,\ell \leq M_N} \frac{\min(\mu_k^{-1/2 },\mu^{-1/2}_{\ell} )}{\mu^{1/2}_k \mu^{1/2}_{\ell}} + C\sum_{k\leq \ell \leq M_N}  \mu^{-3/4}_{k} \mu^{-3/4}_{\ell}+C (\sum_{k \leq M_N} \mu^{-3/4}_{k})^2 \\
&\leq  C.
\end{aligned}
\end{equation}
 \medskip 

 $\bullet$ {\bf Study of the term $III$:} 
We have
\begin{eqnarray*}
III&=&-2\int_M \Big( \sum_{n_1\leq N}\frac{\vert \phi_{n_1}\vert^2}{\lambda^2_{n_1}+1}\Big) \Big( \sum_{n_2\leq N}\frac{\vert \phi_{n_2}\vert^2}{\lambda^2_{n_2}+1}\Big)dx\\
&=&-2  \Big( \a^2_N  +\int_M \big(\sum_{k\leq M_N}  G_{\mu_k}(x) \big)^2dx\Big).
\end{eqnarray*}
  The last term in the previous line is $I_{22}$ up to a factor, hence we can write
  \begin{equation*}
  III=-2\a_N ^2+X_{4,N},
  \end{equation*}
where   $X_{4,N}=II_{22}$ up to a factor.\medskip\medskip
 \begin{rema}
We could also define  $\Pi_N$ as the smooth projector    $\Pi_N=\chi(\sqrt{-\Delta}/N)$, where  ${\chi\in C_0^\infty(\R)}$  is equal to $1$ on $[-1,1]$. 
Then 
$$
\Pi_{N}(\sum_{n}c_n\varphi_n)=\sum_{n}\chi(\lambda_n/N)c_n\varphi_n\,.
$$
With such a definition of $\Pi_N$ the singular term $II$ becomes
$$
II=
\sum_{n}\frac{\chi(\lambda_n/N)}{1+\lambda_n^2}|g_n(\omega)|^2
\sum_{m}\frac{\chi(\lambda_m/N)}{1+\lambda_m^2}\int_{M}|\varphi_n(x)\varphi_{m}(x)|^2dx\,.
$$
In this context, the main contribution of this term is 
$$
\sum_{m}\frac{\chi(\lambda_m/N)}{1+\lambda_m^2}|\varphi_{m}(x)|^2=K_{N}(x,x),
$$
where  $K_N(x,y)$ is the kernel of the operator   $(1-\Delta)^{-1}\chi(\sqrt{-\Delta}/N).$ Therefore $K_N$ is a regularised version of the Green function  $G$ of $1-\Delta$, namely $K_N(\cdot,y)=\Pi_N G(\cdot,y)$. In the case of a manifold without boundary, we can use the Helffer-Sj\"ostrand formula  (see for instance~\cite{BGT2}) and a partition of unity to prove that 
$$
K_{N}(x,x)=\alpha_{N}+\mathcal{O}(1),
$$
where
$$
\alpha_{N}\equiv \sum_{m}\frac{\chi(\lambda_m/N)}{1+\lambda_m^2}\approx \ln N\,.
$$
Therefore measuring  the singularity on the diagonal of the truncated Green function of $1-\Delta$ is the key point of the analysis of the singular term. The above interpretation of $K_N(x,y)$  is  in the spirit of  the analysis in Simon~\cite{Simon}.  
 \end{rema}
\subsubsection{An $L^2$ estimate}
A crucial step in the proof is the following
\begin{lemm}\ph \label{PropL2}
\begin{equation}\label{L2_bis}
\|f_{N}(u)-f_{M}(u)\|_{L^2(d\mu(u))}\lesssim M^{-\sigma'},
\end{equation}
for some positive constant $\sigma'$. 
\end{lemm}
\begin{proof}
Thanks to the analysis in the previous section, we can write 
$$
\|f_{N}(u)-f_{M}(u)\|_{L^2(d\mu(u))}\lesssim J_1+J_2+J_3+J_4,
$$
where $J_1$, $J_2$, $J_3$ and $J_4$ are defined as follows.
The term $J_1$ is the contribution of $X_{1,N}(u)$ and thus it is defined by 
\begin{equation*}
J_1=
\Big\|
\sum_{\substack{
(n_1,n_2,n_3,n_4)\in \Lambda^c
\\
\max(n_1,n_2,n_3,n_4)\geq M
\\ n_1,n_2,n_3,n_4\leq N
}
}
\frac{g_{n_1}(\omega)}{(\lambda^2_{n_1}+1)^{\frac{1}{2}}}
\frac{\overline{g_{n_2}(\omega)}}{(\lambda^2_{n_2}+1)^{\frac{1}{2}}}
\frac{g_{n_3}(\omega)}{(\lambda^2_{n_3}+1)^{\frac{1}{2}}}
\frac{\overline{g_{n_4}(\omega)}}{(\lambda^2_{n_4}+1)^{\frac{1}{2}}}
\,\gamma(n_1,n_2,n_3,n_4)
\Big\|_{L^{2}(\Omega)}\,.
\end{equation*}
The term $J_2$ is the contribution of $X_{3,N}(u)$ and thus it is defined by
$$
J_2=
\Big\|
\sum_{M\leq n\leq N}
\frac{|g_{n}(\omega)|^4}{(\lambda^2_{n}+1)^{2}}
\gamma(n,n,n,n)
\Big\|_{L^{2}(\Omega)}\,.
$$
The term $J_3$ is the contribution of $I$ and thus it is defined by 
$$
J_3=
\Big\|
\sum_{\substack{
\max(n_1,n_2)\geq M
\\
n_1,n_2\leq N
}
}
\frac{(|g_{n_1}(\omega)|^2-1)\, (|g_{n_2}(\omega)|^2-1)}
{(\lambda_{n_1}^2+1)(\lambda_{n_2}^2+1)}
\int_{M}|\varphi_{n_1}|^2|\varphi_{n_2}|^2
\Big\|_{L^{2}(\Omega)}\,.
$$
Finally, the term $J_4$ is the contribution of the renormalized part of $II$ and therefore it is defined by
$$
J_4=
\Big\|
\sum_{M\leq n\leq N}
\frac{ G_{N}(n) (|g_{n}(\omega)|^2-1)}
{\lambda_{n}^2+1}
\Big\|_{L^{2}(\Omega)}\,.
$$
Let us first estimate $J_4$. Using 
Proposition~\ref{ww} and orthogonality, we get that
$$
J_4^2\lesssim
\sum_{M\leq n\leq N}\frac{1}{n^2}
\lessim M^{-1}\,.
$$
For the estimate of $J_2$, we will not use an orthogonality in $\omega$, we will simply rely on the triangle inequality. The estimates for $J_1$ and $J_3$ will rely on orthogonality arguments. The estimates for $J_1$, $J_2$ and $J_3$ will rely on the following key estimate for the behavior of $\gamma(n_1,n_2,n_3,n_4)$
\begin{lemm}\ph\label{key}
There exists  $\delta>0$ such that 
\begin{equation}\label{S1}
\sum_{\substack{
\max(\lambda_{n_1},\lambda_{n_2},\lambda_{n_3},\lambda_{n_4})\geq M
}
}
\frac{|\gamma(n_1,n_2,n_3,n_4)|^2}
{(\lambda^2_{n_1}+1)(\lambda^2_{n_2}+1)(\lambda^2_{n_3}+1)(\lambda^2_{n_4}+1)}
\lesssim M^{-\delta}\, .
\end{equation}
\end{lemm}

We postpone the proof of this result and finish the proof of Lemma\;\ref{PropL2}. Using an orthogonality  argument, we can estimate $J_1$ as follows
$$
J_1^2\lesssim
\sum_{\substack{
\max(n_1,n_2,n_3,n_4)\geq M
\\
n_1,n_2,n_3,n_4\leq N
}
}
\\
\frac{|\gamma(n_1,n_2,n_3,n_4)|^2}{(n_1+1)(n_2+1)(n_3+1)(n_4+1)}
$$
which can be readily estimated by an application of Lemma~\ref{key}.

To  estimate $J_2$, we use the bound~\eqref{sog}, which implies that 
$$\| \varphi_n\|_{L^4} \leq C \lambda_n^{\frac 1 4},
$$ 
and therefore  
$$
J_2\lesssim 
\sum_{M\leq n\leq N}\frac{1}{n^{3/2}}
\lesssim M^{-1/2}\,.
$$
Finally concerning $J_3$, we can use another orthogonality argument in order to write
\begin{eqnarray*}
J_3^2 & \lesssim &
 \sum_{\substack{
\max(n_1,n_2)\geq M
\\
n_1,n_2\leq N
}
}
\frac{|\gamma(n_1,n_1,n_2,n_2)|^2}{(n_1+1)^2(n_2+1)^2}
\\
& & +
\sum_{M\leq n_1,n_2\leq N}
\frac{|\gamma(n_1,n_1,n_1,n_1)\gamma(n_2,n_2,n_2,n_2)|}{(n_1+1)^2(n_2+1)^2}.
\end{eqnarray*}
The first term in the right hand side is estimated similarly as $J_1^2$ (it is a sub-case), while the second term is estimated as $J_2$.
This completes the proof of Lemma~\ref{PropL2}.
\end{proof} 

\begin{proof}[Proof of Lemma\,\ref{key}]
By a symmetry argument, we can estimate the left hand-side of~\eqref{S1} by
\begin{equation}\label{S2}
\sum_{\substack{
N_1\geq N_2\geq N_3\geq N_4
\\
N_1\geq M
\\
(N_1,N_2,N_3,N_4)-{\rm dyadic}
}
}
(N_1 N_2 N_3 N_4)^{-2}
\sum_{\substack{
\lambda_{n_j}\sim N_j
\\
j=1,2,3,4
}
}
|\gamma(n_1,n_2,n_3,n_4)|^2\,.
\end{equation}
Now we can perform the $n_1$ summation and estimate~\eqref{S2} as
\begin{equation}\label{S3}
\sum_{\substack{
N_1\geq N_2\geq N_3\geq N_4
\\
N_1\geq M
}
}
(N_1 N_2 N_3 N_4)^{-2}
\sum_{\substack{
\lambda_{n_j}\sim N_j
\\
j=2,3,4
}
}
\|\varphi_{n_2}\varphi_{n_3}\varphi_{n_4}\|_{L^2(M)}^2\,.
\end{equation}
Next, we can write
$$
\sum_{\substack{
\lambda_{n_j}\sim N_j
\\
j=2,3
}
}
\|\varphi_{n_2}\varphi_{n_3}\varphi_{n_4}\|_{L^2(M)}^2
=
\int_{M}
|\varphi_{n_4}(x)|^2
\Big(\sum_{\lambda_{n_2}\sim N_2}|\varphi_{n_2}(x)|^2\Big)
\Big(\sum_{\lambda_{n_3}\sim N_3}|\varphi_{n_3}(x)|^2\Big)
dx.
$$
Now, from Proposition~\ref{prop.2} applied to control
$$ e(x, 2M, M)= e(x, 2M, 2M-1) + \cdots + e(x, M+1, M),$$
we get the pointwise bound
$$
\sum_{\lambda_{n}\sim N}|\varphi_{n}(x)|^2\lessim N^2
$$
and (integrating on the manifold) 
$$
\#\{
n\,:\, \lambda_{n}\sim N
\}
\lesssim N^2.
$$
This gives  that~\eqref{S3} can be estimated by
$$
\sum_{\substack{
N_1\geq N_2\geq N_3\geq N_4
\\
N_1\geq M
}
}
(N_1 N_2 N_3 N_4)^{-2}
N_4^2 (N_2^2N_3^2)
$$
which clearly can be bounded by $M^{-\delta}$ for any $0<\delta<2$.
\end{proof}

\subsubsection{Proof of Theorem~\ref{gib}}
To prove Theorem~\ref{gib}, the main point is to estimate the measure
\begin{equation}\label{measure}
\mu\big(u\,:\, -f_{N}(u)>\ln(\lambda)\big)
\end{equation}
and to show that
\begin{equation}\label{int}
\int_0^\infty\mu\big(u\,:\, -f_{N}(u)>\ln(\lambda)\big) \,\lambda^p d\lambda <C,
\end{equation}
where $C$ is {\it independent} of $N$.
Standards arguments show that~\eqref{measure} implies Theorem~\ref{gib} (see~\cite{BTT,ThTz, Tzvetkov3}).
Now we can write
$$
f_{N}(u)=\frac12
\int_{M}\big( (\Pi_{N}(u))^2-2\alpha_N \big)^2-\a_N^2.
$$
Therefore, we have the pointwise bound 
\begin{equation}\label{defoc}
-f_{N}(u)\lesssim (\ln(N))^4\,.
\end{equation}
The power of $\ln(N)$ of the last estimate is not of importance for the further analysis. Notice that here we make a crucial use of the defocusing nature of the
nonlinear interaction.
Using~\eqref{defoc} we obtain that if $M$ is such that
$$
\ln(\lambda)-C (\ln(M))^4\geq 1,
$$ 
where $C$ is the implicit constant appearing in~\eqref{defoc}, then 
$$
\mu(u\,:\, -f_{N}(u)>\ln(\lambda))
\leq \mu(u\,:\, -(f_{N}(u)-f_{M}(u))\geq 1)\,.
$$
We therefore choose $M$ such that $\ln(\lambda)\approx (\ln(M))^4$, i.e. 
$$
M\approx e^{(\ln(\lambda))^{\frac{1}{4}}}\,.
$$
The result clearly follows from the following bound
\begin{equation}\label{main_prop}
\exists c, \delta >0; \forall N>M, \mu(u\,:\, -(f_{N}(u)-f_{M}(u))\geq 1)\lessim e^{-cM^{\delta}}\,.
\end{equation}

Finally,~\eqref{main_prop} follows from the $L^2$-bound~\eqref{L2_bis} and classical hypercontractivity estimates (see for instance~\cite[Proposition~2.4]{ThTz}). 
This in turn completes the proof of Theorem~\ref{gib}.

\subsection{Definition of the nonlinearity}
The aim of this paragraph is to prove Proposition~\ref{prop.NL_bis}.

  By the result~\cite[Proposition 2.4]{ThTz} on the Wiener chaos, we only have to prove the statement for $p=2$.
 
Recall the definition of $F_N$ in~\eqref{gauge_bis} and let $\s>2$. To begin with, use that 
\begin{equation*}
\int_{X^{0}(M)}\|F_{N}(u_{N})-F_{M}(u_{M})\|^{2}_{H^{-\s}(M)}\text{d}\mu(u)=\int_{\Omega}\|F_{N}(\Psi_{N})-F_{M}(\Psi_{M})\|^{2}_{H^{-\s}(M)}\text{d}{\bf p}.
\end{equation*}
Therefore, we are reduced to prove that  $\big(F_{N}(\Psi_{N})\big)_{N\geq1}$ is a Cauchy sequence in $L^{2}\big(\Omega; H^{-\s}(M)\big)$. Denote by 
\begin{equation*}
\chi_{N}=|\Psi_{N}|^{2}\Psi_{N}-2\a_N\Psi_{N}.
\end{equation*}
It is enough to show the result for $(\chi_{N})$, because once we know that $\chi_{N} \longrightarrow \chi$ in $L^{2}\big(\Omega; H^{-\s}(M)\big)$, we deduce that $F_{N}(\phi_{N})=\Pi_{N}\chi_{N} \longrightarrow \chi$ in $L^{2}\big(\Omega; H^{-\s}(M)\big)$.  In the sequel, we will use the notation  $[n]=\lambda^2_n+1$. By the Weyl formula we have $[n]\sim n$ when $n\longrightarrow +\infty$. Then, by definition of $\Psi_{N}$ we can compute 
\begin{multline*}
|\Psi_{N}|^{2}\Psi_{N}
=\sum_{{n_{1},n_{2},n_{3}\leq N}}\frac{g_{n_{1}}\ov{g}_{n_{2}}g_{n_{3}}}{[n_{1}]^{\frac12}[n_{2}]^{\frac12}[n_{3}]^{\frac12}}\phi_{n_{1}}\ov{\phi_{n_{2}}}\phi_{n_{3}}\\
\begin{aligned}
&=\sum_{\substack{n_{1},n_{2},n_{3}\leq N,\\n_{1}\neq n_{2},n_{3}\neq n_{2}}}\frac{g_{n_{1}}\ov{g}_{n_{2}}g_{n_{3}}}{[n_{1}]^{\frac12}[n_{2}]^{\frac12}[n_{3}]^{\frac12}}\phi_{n_{1}}\ov{\phi_{n_{2}}}\phi_{n_{3}}+2\sum_{{n,m\leq N}}\frac{|g_{n}|^{2}g_{m}}{[n][m]^{\frac12}}|\phi_{n}|^{2}\phi_{m}-\sum_{{n\leq N}}\frac{|g_{n}|^{2}g_{n}}{[n]^{\frac32}}|\phi_{n}|^{2}\phi_{n}\\
&:=\Sigma_1(N)+\Sigma_2(N)-\Sigma_3(N).
\end{aligned}
\end{multline*}
Then for $\Sigma_2(N)$ we use the decomposition~\eqref{decomp2}
\begin{eqnarray*}
\Sigma_2(N)
&=&2 \Psi_N \sum_{{n \leq N}}\frac{|g_{n}|^{2} }{[n]}|\phi_{n}|^{2}  \\
&=&2 \Psi_N \sum_{{n \leq N}}\frac{|g_{n}|^{2}-1 }{[n]}|\phi_{n}|^{2} +2\a_N \Psi_N +2 K_N\Psi_N\\
&:=&\Sigma_{21}(N)+\Sigma_{22}(N)-\Sigma_{23}(N).
\end{eqnarray*}
Observe that $\Sigma_{22}(N)$ is the term which is removed in the definition of $\chi_N$. Therefore we are reduced to study the contribution of the other terms.\medskip

{\bf $\bullet $ Contribution of $\boldsymbol {\Sigma_1}$:}
For all $1\leq M\leq N$
\begin{equation}\label{cr}
\<\Sigma_{1}(N)-\Sigma_{1}(M)\,|\,\phi_{k}\>=\sum_{B_{M,N}}\frac{g_{n_{1}}\ov{g}_{n_{2}}g_{n_{3}}}{[n_{1}]^{\frac12}[n_{2}]^{\frac12}[n_{3}]^{\frac12}}\gamma(n_1,n_2,n_3,k),
\end{equation}
 where the set $B_{M,N}$ is defined by
\begin{multline*}
 B_{M,N}=\Big\{(n_{1},n_{2},n_{3})\in \N^{3}\;\;\text{s.t.}\;\; 0\leq n_{1}, n_{2}  ,n_{3}\leq N,\;\;n_{1}\neq n_{2},\;\; n_{3}\neq n_{2},\\
\qquad \;\;\text{and}\;\;\max(n_1,n_2,n_3)>M      \Big\}.
\end{multline*} 
From~\eqref{cr} we obtain 
\begin{multline*}
\big\|\<\Sigma_{1}(N)-\Sigma_{1}(M)\,|\,\phi_{k}\>\big\|^{2}_{L^{2}(\Omega)}=\\
{{\int_{\Omega}}}\sum_{\substack {(n_{1},n_{2},n_{3})\in B_{M,N}\\(m_{1},m_{2},m_{3})\in B_{M,N}}}
\frac{g_{n_{1}}\ov{g}_{n_{2}}g_{n_{3}}\ov{g}_{m_{1}}{g}_{m_{2}}\ov{g}_{m_{3}}}
{[n_{1}]^{\frac12}[n_{2}]^{\frac12}[n_{3}]^{\frac12}[m_{1}]^{\frac12}[m_{2}]^{\frac12}[m_{3}]^{\frac12}}   \gamma(n_1,n_2,n_3,k)\ov{\gamma(m_1,m_2,m_3,k)}\text{d}{\bf p}.
\end{multline*}
Since the $(g_{n})$ are independent and centred, we deduce that each term in the r.h.s. vanishes, unless $n_{2}=m_{2}$ and $(n_{1},n_{3})=(m_{1},m_{3})$ or $(n_{1},n_{3})=(m_{3},m_{1})$. Thus 
\begin{equation*} 
\big\|\<\Sigma_{1}(N)-\Sigma_{1}(M)\,|\,\phi_{k}\>\big\|^{2}_{L^{2}(\Omega)}\leq C\sum_{(n_{1},n_{2},n_{3})\in B_{M,N}}\frac{\vert \gamma(n_1,n_2,n_3,k)\vert^2}{\<n_{1}\>\<n_{2}\>\<n_{3}\>}.
\end{equation*}

Let $\s>1$, then we get
\begin{eqnarray*}
\big\|\Sigma_{1}(N)-\Sigma_{1}(M)\big\|^{2}_{L^{2}(\Omega;H^{-\s}(M))}&=&\sum_{k\in \N}\frac1{\<k\>^{\s}} \big\|\<\Sigma_{1}(N)-\Sigma_{1}(M)\,|\,\phi_{k}\>\big\|^{2}_{L^{2}(\Omega)}  \\
&\leq &C\sum_{\substack {(n_{1},n_{2},n_{3})\in B_{M,N}\\k\in \N}}\frac{\vert \gamma(n_1,n_2,n_3,k)\vert^2}{\<k\>^{\s}\<n_{1}\>\<n_{2}\>\<n_{3}\>},
\end{eqnarray*}
and this term can be estimated as $J_1$, namely for some $\eta>0$
\begin{equation*}
\big\|\Sigma_{1}(N)-\Sigma_{1}(M)\big\|_{L^{2}(\Omega;H^{-\s}(M))}
\leq \frac{C}{M^\eta} .
\end{equation*}
\medskip

{\bf $\bullet $ Contribution of $\boldsymbol {\Sigma_3}$:} For all $1\leq M\leq N$
\begin{equation*} 
\<\Sigma_{3}(N)-\Sigma_{3}(M)\,|\,\phi_{k}\>=\sum_{M<n\leq N}\frac{\vert g_n\vert^2g_{n}}{[n]^{\frac32}}\gamma(n,n,n,k).
\end{equation*}
By the  the bound~\eqref{sog} we infer 
\begin{equation*}
\vert \gamma(n,n,n,k)\vert \leq C \<n\>^{1/4}\<k\>^{1/4}
\end{equation*}
and from the independence of the Gaussians we get 
\begin{equation*}
\big\|\<\Sigma_{3}(N)-\Sigma_{3}(M)\,|\,\phi_{k}\>\big\|^{2}_{L^{2}(\Omega)}=C
\sum_{M<n\leq N}\frac{\vert \gamma(n,n,n,k)\vert^2}{[n]^{3}}
\leq  C\<k\>^{1/2}  \sum_{M<n\leq N}\frac{1}{\<n\>^{5/2}} \leq C\frac{\<k\>^{1/2}}{M} .
\end{equation*}
 Finally, if $\s>3/2$  we obtain
\begin{equation*}
\big\|\Sigma_{3}(N)-\Sigma_{3}(M)\big\|^{2}_{L^{2}(\Omega;H^{-\s}(M))}=\sum_{k\in \N}\frac1{\<k\>^{\s}} \big\|\<\Sigma_{3}(N)-\Sigma_{3}(M)\,|\,\phi_{k}\>\big\|^{2}_{L^{2}(\Omega)}  
\leq \frac{C}{M} .
\end{equation*}
\medskip

{\bf $\bullet $ Contribution of $\mathbf{\Sigma_2}$:} For all $1\leq M\leq N$
\begin{equation*} 
\<\Sigma_{21}(N)-\Sigma_{21}(M)\,|\,\phi_{k}\>=2\sum_{C_{M,N}}\frac{(\vert g_{n}\vert^2 -1) {g}_{m}}{[n][m]^{\frac12}}\gamma(n,n,m,k),
\end{equation*}
 where the set $C_{M,N}$ is defined by
\begin{equation*}
 C_{M,N}=\Big\{(n,m)\in \N^{2}\;\;\text{s.t.}\;\; (M< n\leq N\;\;\text{and}\;\;m\leq N)
  \;\;\text{or}\;\;(M< m\leq N\;\;\text{and}\;\;n\leq M)      \Big\}.
\end{equation*} 
Then by the orthogonal properties of the Gaussians we obtain 
\begin{equation*}
\big\|\<\Sigma_{21}(N)-\Sigma_{21}(M)\,|\,\phi_{k}\>\big\|^{2}_{L^{2}(\Omega)}\leq C
\sum_{\substack {n,m\leq N\\\max(n,m)>M}}
\frac{ \vert  \gamma(n,n,m,k)         \vert^2 }
{\<n\>^2  \<m\>} ,
\end{equation*}
which (when multiplied by $ \<k\>^{-1}$) is estimated by Lemma~\ref{key} in the previous section (it is actually a sub-case) fixing $n_1=n_2=n, n_3=m, n_4=k$): 
\begin{multline*}
\sum_k \<k\>^{-1}\big\|\<\Sigma_{21}(N)-\Sigma_{21}(M)\,|\,\phi_{k}\>\big\|^{2}_{L^{2}(\Omega)}\leq C 
\sum_k \sum_{\substack {n,m\leq N\\\max(n,m)>M}}\frac{ \vert  \gamma(n,n,m,k)         \vert^2 }
{\<n\>^2  \<m\>  \<k\>} \\
\leq \sum_{\substack {n_1, n_2,n_3, n_4\leq N\\\max(n_1, n_2, n_3, n_4>M}}\frac{ \vert  \gamma(n_1, n_2, n_3, n_4)         \vert^2 }
{\<n_1\> \<n_2\> \<n_3\>  \<n_4\>}\leq \frac{C }{M^\delta}.
\end{multline*}
Therefore, if $\s\geq 1$  we obtain
\begin{equation*}
\big\|\Sigma_{21}(N)-\Sigma_{21}(M)\big\|^{2}_{L^{2}(\Omega;H^{-\s}(M))} \leq \frac{C}{M} .
\end{equation*}

We now study the term $\Sigma_{23}$. We have 
 \begin{equation}\label{defs3} 
 \Sigma_{23}(N)-\Sigma_{21}(M) =2 (K_N-K_M)\Psi_N+2 K_M(\Psi_N-\Psi_M).
\end{equation}
Let's consider the contribution of the first term in the previous line.  For $k\geq 0$ 
\begin{equation*}
\< (K_N-K_M)\Psi_N\,|\,\phi_{k}\> = 
\sum_{n\leq N} \frac{g_{n}}{[n]^{\frac12}}\int_M (K_N-K_M)\phi_n \ov{\phi_k}.
\end{equation*}
By the orthogonal properties of the Gaussians and Parseval, we obtain
\begin{eqnarray*}
\big\|\< (K_N-K_M)\Psi_N\,|\,\phi_{k}\>\big\|^{2}_{L^{2}(\Omega)}&=& 
\sum_{n\leq N} \frac{1}{[n]}\Big \vert \int_M (K_N-K_M)\phi_n \ov{\phi_k}\Big\vert^2\\
&\leq& C
\sum_{n\leq N} \Big \vert \int_M (K_N-K_M)\phi_n \ov{\phi_k}\Big\vert^2\\
&\leq& C \int_M (K_N-K_M)^2\vert \phi_k\vert^2 \\
&\leq& C \<k\>\int_M (K_N-K_M)^2.
\end{eqnarray*}
Now we estimate $\int_M (K_N-K_M)^2$ as in~\eqref{equiibis} and we get that for $\s>2$
\begin{equation*}
\big\|(K_N-K_M)\Psi_N\big\|_{L^{2}(\Omega;H^{-\s}(M))} \leq \frac{C}{M^{\eta}},
\end{equation*}
for some $\eta>0$. The estimate of the second term in~\eqref{defs3} is similar.\medskip

As a conclusion, we are able to define a limit  $F(u)$ so that  for all $p\geq 2$  
\begin{equation}\label{tightness} 
\|F(u) \|_{L^p_{\mu}H^{-\s}(M)}\leq C_p.
\end{equation}

\subsection{Proof of Theorem~\ref{thmNLS_bis}}
Once estimate~\eqref{tightness} is proved, 
the analysis in the proof of Theorem~\ref{thmNLS_bis} is the same as in the proof of Theorem~\ref{thmHW} (see Section~\ref{Sect.7}).
\subsection{Proof of Proposition~\ref{ww}}\label{sec.85}
The proof of Proposition~\ref{ww} will rely on two results about the spectral function of the Laplace operator on a compact manifold $M$. On the one hand a precise asymptotic  "away" from the boundary essentially due to H\"ormander~\cite[Theorem 17.5.10]{Ho3} and on the other hand a general  bound near the boundary due to Sogge~\cite{So}. Let 
$$ e(x, \lambda, \mu)= \sum_{\mu \leq \lambda_n < \lambda} |\phi_n (x) |^2,$$
be the spectral function. Then we have the following

\begin{prop}\ph\label{prop.1}
Let $d(x)=d(x, \partial M)$ be the distance of the point $x\in M$ to the boundary $\partial M$. 
There exists $C>0$ such that for any $\lambda >1$, any  $x\in M$ satisfying $d(x) \geq \lambda^{- 1/2}$ and any $\delta\in [0,1]$, we have 
\begin{equation}\label{spect}
| e(x, \lambda + \delta\lambda^{1/2}, \lambda) - \frac{ \delta} {2 \pi} \lambda^{3/2} | \leq C \lambda^{5/4}.
\end{equation}
\end{prop}

This result can be deduced from Seeley~\cite[Estimate (0.1) with $n=2$ and $\tau=\lambda$]{Seeley}. Here we give an argument based on  H\"ormander~\cite{Ho3}.
\begin{proof}
We first treat the case of  Dirichlet boundary conditions.  For $y\in \R^2$ and $\tau>0$, denote by 
\begin{equation*}
e_0(y,\tau^2)=\frac1{4\pi^2} \int_{\{\xi \in \R^2 \,:\,\vert \xi\vert< \tau \}}e^{i y\cdot \xi} d\xi.
\end{equation*}
Then by~\cite[Theorem 17.5.10]{Ho3}
\begin{equation}\label{horm}
| e(x, \lambda, 0) - e_0(0,\lambda^2)+e_0(2d(x),\lambda^2)   | \leq C \lambda.
\end{equation}
 (Notice that in~\cite{Ho3} the parameter  $\lambda $ denotes the eigenvalues while here $\lambda$ are the square root of the eigenvalues and that in our setting $\gamma(x) = (\text{det}( g^{i,j} (x)))^{-1/2}$). The function $e_0$ is radial in $y$ and with a change of variables we get 
 \begin{equation}\label{defJ}
 e_0(y,\tau^2)=\frac{\tau^2}{4\pi} J(\tau \vert y\vert), \quad \text{with}\quad J(t)=  \int_{\{\xi \in \R^2 \,:\,\vert \xi\vert< 1\} }e^{i  t \xi_1} d\xi.
 \end{equation}
 We now claim that for all $t\in \R$
 \begin{equation}\label{ineq}
 J(t)\leq C\<t\>^{-3/2}, \quad  J'(t)\leq C\<t\>^{-3/2}.
 \end{equation}
For the first   inequality we write for $\vert t\vert \geq 1$
\begin{eqnarray*}
J(t)=\int_{-1}^1 \int_{-\sqrt{1-\xi^2_2}}^{\sqrt{1-\xi^2_2}} e^{it \xi_1} d\xi_1 d\xi_2&=&\frac{1}{it} \int_{-1}^1 \big(   e^{it \sqrt{1-\xi^2_2}}-   e^{-it \sqrt{1-\xi^2_2}}  \big)  d\xi_2\\
&=&\frac{2}{it} \int_{0}^{\pi/2} \big(   e^{it \cos \theta}-   e^{-it \cos \theta}  \big)  \cos \theta d\theta
\end{eqnarray*}
and the result follows  from  the stationary phase. The estimate on $J'$ is obtained similarly.

Clearly, from $e_{0}(0,\lambda^{2})=\lambda^{2}/(4\pi)$ we get 
\begin{equation}\label{este0}
e_0\big(0, (\lambda +\delta \lambda^{1/2})^2\big)-e_0(0, \lambda^2)=\frac{\delta}{2\pi} \lambda^{3/2}+\mathcal{O}(\lambda),
\end{equation}
and by~\eqref{defJ} and~\eqref{ineq}
\begin{multline*}
\big \vert e_0\big(2d(x), (\lambda +\delta \lambda^{1/2})^2\big)-e_0(2d(x), \lambda^2) \big \vert \leq \\
\leq C\lambda^2 \delta d(x)\lambda^{1/2} \<d(x)\lambda\>^{-3/2}+C\delta \lambda^{3/2} \<d(x)\lambda\>^{-3/2} \leq C \delta\lambda^{5/4},
\end{multline*}
under the condition $d(x)\geq \lambda^{-1/2}$.
Then by~\eqref{horm} and~\eqref{este0} we get~\eqref{spect}.
\end{proof}

\begin{prop}\ph\label{prop.2}
There exists $C>0$ such that for any $\lambda >1$ and $x\in M$ 
\begin{equation}\label{sogge}
 e(x, \lambda + 1, \lambda) \leq C \lambda.
\end{equation}
\end{prop}
In particular, the previous result implies the bound
\begin{equation}\label{sog} 
\| \varphi_n \|_{L^\infty}\leq C \lambda_n ^{1/ 2} \leq C\<n\>^{1/4}.
\end{equation}

Proposition~\ref{prop.2} is the case $q = + \infty$ in~\cite[(1.5)]{SmSo}, which in turn is an easy consequence of~\cite{So}.  \medskip

Let us now prove Proposition~\ref{ww}.  
For any $\mu>0$ and $x\in M$, we write 
\begin{multline}
 \sum_{\lambda_m \in [\mu, \mu+ \delta \mu^{1/2})} \frac 1 { \lambda_m^2 +1} |\phi_m|^2(x) =\\
\begin{aligned}
& = \sum_{\lambda_m \in [\mu, \mu+ \delta \mu^{1/2})}  \frac {  |\phi_m|^2(x)} { \mu^2 +1}+\sum_{\lambda_m \in [\mu, \mu+ \delta \mu^{1/2})} \bigl(\frac 1 { \lambda_m^2 +1}- \frac 1 { \mu^2 +1}  \Bigr)  |\phi_m|^2(x)\nonumber  \\
&:=  F_{1,\mu}(x)+  F_{2,\mu}(x).
\end{aligned}
   \end{multline}
   First observe that by~\eqref{sogge} we have for all $x\in M$
   \begin{equation}\label{borne1}
    e(x, \mu+ \delta\mu^{1/2}, \mu)\leq C \mu^{3/2}.
   \end{equation}
   Then from the previous bound, we deduce that 
 \begin{equation*}
 \vert  F_{2,\mu}(x)    \vert\leq C   \frac {\mu^{3/2}} { \mu^2 +1}\sum_{\lambda_m \in [\mu, \mu+ \delta \mu^{1/2})}  |\phi_m|^2(x)\leq C\mu^{-1},
 \end{equation*}
 which is an acceptable bound.
 
   We now turn to the estimate of $F_{1,\mu}$. For $x\in M$ such $d(x)>\mu^{-1/2}$ we can use~\eqref{spect} to write
   \begin{equation*}
F_{1,\mu}(x) =\frac{ 1} {\mu^2+1} e(x, \mu+ \delta\mu^{1/2}, \mu) = \frac{ \delta  \mu^{\frac 3 2}} {2\pi (\mu^2+1)}+ \mathcal{O} ( \mu^{-3/4}).
 \end{equation*}
 But thanks to the Weyl formula
 $$ \sharp \big\{ n ; \lambda_n \in [\mu, \mu+ \delta \mu^{\frac 1 2})\big\} = \frac {\delta} {2\pi} \mu^{\frac 3 2} \big( 1+ \mathcal{O} ( \mu^{- \frac 1 4})\big), $$
 which is obtained by integrating~\eqref{spect} and using~\eqref{borne1}, we get 
 \begin{equation}\label{wf}
  \sum_{\lambda_m \in [\mu, \mu+ \delta \mu^{1/2})}\frac 1{  \lambda_m^2 +1}= \frac{ \delta  \mu^{\frac 3 2}} {2\pi (\mu^2+1)}+ \mathcal{O} ( \mu^{-3/4}),
 \end{equation}
 which was the claim.
 
 Now we assume that $d(x)\leq \mu^{-1/2}$. Then by~\eqref{borne1} we have 
    \begin{equation*}
F_{1,\mu}(x) =\frac{ 1} {\mu^2+1} e(x, \mu+ \delta\mu^{1/2}, \mu) \leq C \frac{   \mu^{\frac 3 2}} {  \mu^2+1 }=\mathcal{O} ( \mu^{-1/2}).
 \end{equation*}
This proves~\eqref{defG_bis}. The fact~\eqref{nul} is obtained by integration of $F_{\mu}$. The proof of Proposition~\ref{ww} is complete.
  
  \begin{proof}[Proof of Proposition~\ref{WW1}]
 By~\eqref{wf} with $\delta=1$,  we observe that Proposition\;\ref{ww} directly implies Proposition~\ref{WW1}.
  \end{proof}

\end{document}